\documentclass[11pt]{amsart}
\usepackage{amsfonts}

\newtheorem{theorem}{Theorem}
\theoremstyle{plain}

\newtheorem{corollary}{Corollary}

\newtheorem{definition}{Definition}

\newtheorem{lemma}{Lemma}

\newtheorem{proposition}{Proposition}

\numberwithin{equation}{section}
\numberwithin{theorem}{section}
\numberwithin{equation}{section}

\begin{document}
	\title[Generalized Smoothness and the Cauchy Problem]{On the Cauchy problem for parabolic integro-differential equations in generalized H\"{o}lder spaces}
	\author{R. Mikulevi\v{c}ius and Fanhui Xu}
	\address{University of Southern California, Los Angeles}
	\date{March 22, 2018}
	\subjclass{60H10, 60H35, 41A25}
	\keywords{Generalized smoothness, non-local parabolic Kolmogorov equations, Levy processes, strong solutions}
	
	\begin{abstract}
		An integro-differential Kolmogorov equation is considered in H\"{o}lder-type spaces defined by a scalable L\'{e}vy measure. Some properties of those spaces and estimates of the solution are derived by using probabilistic representations. 
	\end{abstract}
	
	\maketitle
	\tableofcontents

\section{Introduction}

Let $\left( \Omega,\mathcal{F},\mathbf{P}\right)$ be a complete probability space. Given a L\'{e}vy measure $\nu$ on $\mathbf{R}^d_0=\mathbf{R}^d\backslash\{0\}$, we suppose there exists an adapted Poisson random measure $J\left( ds, dy\right)$ on $\left( \Omega,\mathcal{F},\mathbf{P}\right)$ such that
\begin{eqnarray*}
	&&\qquad\mathbf{E}\left[ J\left( ds, dy\right)\right]= \nu\left( dy\right)ds,\\
	&&\tilde{J}\left( ds, dy\right)= J\left( ds, dy\right)-\nu\left( dy\right)ds.
\end{eqnarray*}
Then there is a L\'{e}vy process $Z_t^{\nu}$ associated to $\nu$ in the way that
\begin{eqnarray}\label{lev}
\qquad Z_t^{\nu}=\int_0^t\int_{\mathbf{R}^d_0} \chi_{\alpha}\left( y\right) y \tilde{J}\left( ds, dy\right)+\int_0^t\int_{\mathbf{R}^d_0} \left( 1-\chi_{\alpha}\left( y\right)\right) y J\left( ds, dy\right),
\end{eqnarray}
where, as a convention, 
\begin{eqnarray*}
	\alpha:=\inf\{\sigma\in\left( 0,2\right):\int_{\left\vert y\right\vert\leq 1}\left\vert y\right\vert ^{\sigma}\nu\left( dy\right)<\infty \}
\end{eqnarray*}
is the order of $\nu$, and $\chi_{\alpha}\left( y\right):=1_{\alpha\in\left(1,2\right)}+1_{\alpha=1}1_{\left\vert y\right\vert\leq 1}$.

The aim of this paper is twofold. One is to introduce function spaces of generalized smoothness and reveal the embedding relations among them. The other is to study the Cauchy problem of the following parabolic-type Kolmogorov equation within the framework of such generalized smoothness:
\begin{eqnarray}\label{eq1}
\partial_t u\left( t,x\right)&=&Lu\left( t,x\right)-\lambda u\left( t,x\right)+f\left( t,x\right), \lambda\geq 0,\\
u\left( 0,x\right)&=& 0,\quad\left( t,x\right)\in \left[0,T\right]\times\mathbf{R}^d,\nonumber
\end{eqnarray}
where $L$ is the infinitesimal generator of $Z_t^{\nu}$.
 
Study on function spaces of generalized smoothness dates back to the seventies, signified by the work of H. Triebel \cite{tr}, G.A. Kalyabin \cite{kg}, P.I. Lizorkin \cite{kg2} and so on. It is a natural development after the theory of differentiable functions of multi-variables and has been thriving for decades due to its close relation to interpolation theory, potential theory and the theory of differential operators. What is of most interest to us is the possibility to use the language of generalized smoothness to describe and investigate some special L\'{e}vy processes, $\eqref{lev}$ in particular. We know by the L\'{e}vy-Khinchine formula that each L\'{e}vy process $\left(Z_t\right)_{t\geq 0}$ is determined by a continuous negative definite function which is called the symbol. Generally speaking, by assuming the symbol $\tilde{\psi}$ of $\left(Z_t\right)_{t\geq 0}$ behaves up to a perturbation like $\psi$, one could expect the scales of spaces associated with $\psi$ plays the same role for $\tilde{\psi}$ as the classical Besov spaces do for elliptic operators. This was illustrated in \cite{fw} and \cite{fw2} and was a motivation for defining such spaces. In this paper, we utilize a continuous function $w$ to capture the discrepancy generated by scaling and we support this viewpoint by investigating $\eqref{eq1}$ in $w$-scaled Besov spaces.

\begin{definition}
	A continuous function $w: \left( 0,\infty\right)\rightarrow \left( 0,\infty\right)$ is called a scaling function if
	\begin{eqnarray*}
		\lim_{r\rightarrow 0}w\left( r\right)=0, \quad\lim_{R\rightarrow \infty}w\left( R\right)=\infty
	\end{eqnarray*}
	and if there is a nondecreasing continuous function $l\left(\varepsilon\right),\varepsilon>0$ such that $\lim_{\varepsilon\rightarrow 0}l\left(\varepsilon\right)=0$ and 
	\begin{equation}\label{scale}
	w\left( \varepsilon r\right)\leq l\left(\varepsilon\right)w\left( r\right), \quad\forall r,\varepsilon>0.
	\end{equation}
	We call $l$ the scaling factor of $w$.
\end{definition} 

Assumptions throughout this paper will be summarized in Section 2.3, denoted by \textbf{A(w,l)}. Intuitively, a measure satisfying \textbf{A(w,l)} is non-degenerate and has a scaling effect on integrability that can be compensated by $w$, which voices for a large family of L\'{e}vy measures, including $\alpha$-stable measures, $\alpha$-stable-like measures and certain radical-and-angular expressed measures. (See Section 2.3.) We will fix a L\'{e}vy measure $\mu$ that meets \textbf{A(w,l)} as our \textbf{reference measure} and use $w$ to define generalized Besov (resp. H\"{o}lder) norms $\left\vert \cdot\right\vert_{\beta,\infty}$ (resp. $\left\vert \cdot\right\vert_{\beta}$) and generalized Besov (resp. H\"{o}lder) spaces $\tilde{C}^{\beta}_{\infty,\infty},\beta>0$ (resp. $\tilde{C}^{\beta}$). (See Section 2.2.) Write $H_T=\left[0,T\right]\times\mathbf{R}^d$. One of the main results of this paper is:

\begin{theorem}\label{thm2}
	Let $\beta\in\left(0,\infty\right), \lambda\geq 0$ and $\nu$ be a L\'{e}vy measure satisfying \textbf{A(w,l)}. If $f\left( t,x\right)\in  \tilde{C}^{\beta}_{\infty,\infty}\left(H_T\right)$. Then there is a unique solution $u\in\left( t,x\right)\in \tilde{C}^{1+\beta}_{\infty,\infty}\left(H_T\right)$ to
	\begin{eqnarray}\label{eqq}
	\partial_t u\left( t,x\right)&=&L^{\nu}u\left( t,x\right)-\lambda u\left( t,x\right)+f\left( t,x\right), \lambda\geq 0,\\
	u\left( 0,x\right)&=& 0,\quad\left( t,x\right)\in H_T,\nonumber
	\end{eqnarray}
	where for any function $\varphi\in  C_b^{2}\left(\mathbf{R}^d\right)$,
	\begin{eqnarray}\label{op}
L^{\nu}\varphi\left( x\right):=\int\left[ \varphi\left( x+y\right)-\varphi\left( x\right)-\chi_{\alpha}\left( y\right)y\cdot \nabla \varphi\left( x\right)\right]\nu\left(d y\right).
	\end{eqnarray}
	Moreover, there exists a constant $C$ depending on $\kappa,\beta, d, T,\mu,\nu$ such that
	\begin{eqnarray}
	\left\vert u\right\vert_{\beta,\infty}&\leq& C\left( \lambda^{-1}\wedge T\right) \left\vert f\right\vert_{\beta,\infty},\label{est6}\\
	\left\vert u\right\vert_{1+\beta,\infty}&\leq& C\left\vert f\right\vert_{\beta,\infty}\label{est3}
	\end{eqnarray}
And there is a constant $C$ depending on $\kappa,\beta,d, T,\mu,\nu$ such that for all $0\leq s<t\leq T$, $\kappa\in\left[ 0,1\right]$,
	\begin{equation}\label{est4}
	\left\vert u\left(t,\cdot\right)-u\left(s,\cdot\right)\right\vert_{\kappa+\beta,\infty}\leq C\left\vert t-s\right\vert^{1-\kappa}\left\vert f\right\vert_{\beta,\infty}.
	\end{equation}
\end{theorem}

As it will be seen later, when $\nu$ behaves like an $\alpha$-stable measure, $\eqref{est6}$-$\eqref{est4}$ are ordinary Besov (equiv. H\"{o}lder-Zygmund) regularity estimates.

By norm equivalence stated in Section 3, Theorem \ref{thm2} implies immediately 
\begin{theorem}\label{thm3}
	Let $\beta\in\left(0,\infty\right), \lambda\geq 0$ and $\nu$ be a L\'{e}vy measure satisfying \textbf{A(w,l)}. If $f\left( t,x\right)\in  \tilde{C}^{\beta}\left(H_T\right)$ and
	\begin{eqnarray*}
	\int_0^1 l\left( t\right)^{\beta}\frac{dt}{t}+\int_1^{\infty} l\left( t\right)^{\beta}\frac{dt}{t^2}<\infty.
	\end{eqnarray*}
	 Then there is a unique solution $u\in\left( t,x\right)\in \tilde{C}^{1+\beta}\left(H_T\right)$ to $\eqref{eqq}$. Moreover, there exists a constant $C$ depending on $\kappa,\beta, d, T,\mu,\nu$ such that
	\begin{eqnarray}
	\left\vert u\right\vert_{\beta}&\leq& C\left( \lambda^{-1}\wedge T\right) \left\vert f\right\vert_{\beta},\label{est62}\\
	\left\vert u\right\vert_{1+\beta}&\leq& C\left\vert f\right\vert_{\beta}\label{est32}
	\end{eqnarray}
	And there is a constant $C$ depending on $\kappa,\beta,d, T,\mu,\nu$ such that for all $0\leq s<t\leq T$, $\kappa\in\left[ 0,1\right]$,
	\begin{equation}\label{est42}
	\left\vert u\left(t,\cdot\right)-u\left(s,\cdot\right)\right\vert_{\kappa+\beta}\leq C\left\vert t-s\right\vert^{1-\kappa}\left\vert f\right\vert_{\beta}.
	\end{equation}
\end{theorem}
In \cite{mp2}, a parabolic-type Kolmogorov equation with an operator $\mathcal{L}=A+B$ was considered in the standard H\"{o}lder-Zygmund space, where $B$ is the lower order part and the principal part $A$ assumes a form of 
\begin{eqnarray*}
Au\left(t, x\right):=\int\left[ u\left(t, x+y\right)-u\left( t,x\right)-\chi_{\alpha}\left( y\right)y\cdot \nabla u\left( t,x\right)\right]\rho\left( t,x,y\right)\frac{dy}{\left\vert y\right\vert^{d+\alpha}}.
\end{eqnarray*}
While in \cite{mp}, a parabolic integro-differential equation perturbed by Gaussian noise was studied in the stochastic H\"{o}lder spaces. Operators were introduced as 
\begin{eqnarray*}
	\mathcal{L}u\left(t, x\right)&:=&\int\left[ u\left(t, x+y\right)-u\left( t,x\right)-1_{\alpha\geq 1}1_{\left\vert y\right\vert\leq 1}y\cdot \nabla u\left( t,x\right)\right]\nu\left( t,x,dy\right)\\
	&+&1_{\alpha=2}a^{ij}\left(t, x\right)\partial^2_{ij}u\left(t, x\right)+1_{\alpha\geq 1}\tilde{b}^{i}\left(t, x\right)\partial_i u\left(t, x\right)+l\left(t, x\right) u\left(t, x\right),
\end{eqnarray*}
and results were expressed in terms of moments. A similar operator was adopted in \cite{mp3} and the corresponding deterministic model was studied in the little H\"{o}lder-Zygmund spaces. Besides, the Cauchy problem for a linear parabolic SPDE of the second order was considered in \cite{rm} and \cite{br} in standard H\"{o}lder classes.

Our note is organized as follows. 

In section 2, notation is introduced and spaces are defined. Meanwhile, we collect all assumptions that are needed in this paper and provide with examples that satisfy all the assumptions. A few more defining properties of our model are discussed as well.

In section 3, we elaborate embedding relations among function spaces. Probability representations are used to extend operations to all functions in $C_b^{\infty}\left(\mathbf{R}^d\right)$. After the extension, those operations become bijections on $C_b^{\infty}\left(\mathbf{R}^d\right)$. Norm equivalence then follows from continuity of operators. 

Regularity estimates in the case of smooth inputs are collected in section 4, while those for Besov (equiv. H\"{o}lder) inputs are put in section 5. Section 6 accommodates existing results that are used in our proofs.

\section{Notation, Spaces and Models}
\subsection{Basic Notation}
$\mathbf{N}=\{0,1,2,3,\ldots,\}$, $\mathbf{N}_{+}=\mathbf{N}\backslash\{0\}$. $H_T=\left[0,T\right]\times\mathbf{R}^d$. $S^{d-1}$ is the unit sphere in $\mathbf{R}^d$. $\Re$ is a notation for the real part of a complex-valued quantity.

For a function $u=u\left( t,x\right) $ on $H_T$, we denote its partial derivatives by $\partial _{t}u=\partial u/\partial t$, $\partial
_{i}u=\partial u/\partial x_{i}$, $\partial _{ij}^{2}u=\partial
^{2}u/\partial x_{i}x_{j}$, and denote its gradient with respect to $x$ by $%
\nabla u=\left( \partial _{1}u,\ldots ,\partial _{d}u\right) $ and $%
D^{|\gamma |}u=\partial ^{|\gamma |}u/\partial x_{1}^{\gamma _{1}}\ldots
\partial x_{d}^{\gamma _{d}}$, where $\gamma =\left( \gamma _{1},\ldots
,\gamma _{d}\right) \in \mathbf{N}^{d}$ is a multi-index.

We use $C_b^{\infty}\left( \mathbf{R}^d\right)$ to denote the set of infinitely differentiable functions on $\mathbf{R}^d$ whose derivative of arbitrary order is finite, and $C^{k }\left( \mathbf{R}^{d}\right),k\in\mathbf{N} $ the class of $k$-times continuously differentiable functions.
 
$\mathcal{S}\left( \mathbf{R}^d\right)$ denotes the Schwartz space on $\mathbf{R}^d$ 
and $\mathcal{S}'\left( \mathbf{R}^d\right)$ denotes the space of continuous functionals on $\mathcal{S}\left( \mathbf{R}^d\right)$, i.e. the space of tempered distributions.

We adopt the normalized definition for Fourier and its inverse transforms for functions in $\mathcal{S}\left( \mathbf{R}^d\right)$, i.e.,
\begin{eqnarray*}
	\mathcal{F}\varphi\left(\xi\right) &=& \hat{\varphi}\left(\xi\right) :=\int e^{-i2\pi x\cdot \xi}\varphi\left(x\right)dx, \\
	\mathcal{F}^{-1}\varphi\left(x\right)&=&\check{\varphi}\left(x\right) := \int e^{i2\pi x\cdot \xi}\varphi\left(\xi\right)d\xi, \enskip \varphi\in\mathcal{S}\left( \mathbf{R}^d\right).
\end{eqnarray*}
Recall that Fourier transform can be extended to a bijection on $\mathcal{S}'\left( \mathbf{R}^d\right)$.

$\mu$ always refers to our reference measure, and $\alpha$ denotes its order unless otherwise specified. 

Throughout the sequel, $Z_t^{\nu}$ represents the L\'{e}vy process associated to the L\'{e}vy measure $\nu$ in the way as $\eqref{lev}$.

For any L\'{e}vy measure $\nu$, any $R>0$ and $\forall B\in\mathcal{B}\left( \mathbf{R}^d_0\right)$, 
\begin{eqnarray}
\nu_R \left(B\right)&=&\int 1_B\left( y/R\right)\nu\left( dy\right),\label{mea}\\ 
\tilde{\nu}_R\left( dy\right) &:=&w\left( R\right) \nu_R\left( dy\right).\label{meas}
\end{eqnarray}
Without loss of generality, we normalize $w$ by a constant so that $w\left( 1\right)=1$ and $\tilde{\nu}_1\left( dy\right)=\nu\left( dy\right)$. Meanwhile, we introduce for any L\'{e}vy measure $\nu$,
\begin{eqnarray}
	\bar{\nu}\left( dy\right):=\frac{1}{2}\left( \nu\left( dy\right)+\nu\left(- dy\right)\right).
\end{eqnarray}

We have specific values assigned for $\alpha_1,\alpha_2,c_0,c_1, c_2, N_0,N_1$, but we allow $C$ to vary from line to line. In particular, $C\left(\cdots\right)$ represents a constant depending only on quantities in the parentheses.

\subsection{Function Spaces of Generalized Smoothness}

By definition of the scaling factor, there is a constant $N>3$ such that $l\left( N^{-1}\right)<1<l\left( N\right)$. For such a $N$, by Lemma 6.1.7 in \cite{bl} and appropriate scaling, there exists $\phi\in C_0^{\infty}\left(\mathbf{R}^d\right)$ such that $supp\left(\phi\right)=\{ \xi:\frac{1}{N}\leq \left\vert \xi\right\vert\leq N\}$, $\phi\left(\xi\right)> 0$ in the interior of its support, and 
\begin{equation*}
\sum_{j=-\infty}^{\infty}\phi\left( N^{-j}\xi\right)=1 \mbox{ if } \xi\neq 0.
\end{equation*}
We denote throughout this paper
\begin{eqnarray*}
	&&\varphi_j =\mathcal{F}^{-1}\left[\phi\left( N^{-j}\xi\right)\right],\quad  j=1,2,\ldots,\xi\in \mathbf{R}^d,\label{j1}\\
	&&\varphi_0 =\mathcal{F}^{-1}\left[ 1-\sum_{j=1}^{\infty}\phi\left( N^{-j}\xi\right)\right].\label{j0}
\end{eqnarray*}
Apparently, $\varphi_j\in \mathcal{S}\left( \mathbf{R}^d\right),j\in\mathbf{N}$. If we write
\begin{eqnarray*}
&& \tilde{\varphi_j}=\varphi_{j-1}+\varphi_{j}+\varphi_{j+1}, j\geq 2,\label{ssch}\\
&& \tilde{\varphi}_1=\check{\phi}+\varphi_{1}+\varphi_{2},\quad \tilde{\varphi}_0=\varphi_{0}+\varphi_{1},\label{sch}
\end{eqnarray*}
then, 
\begin{eqnarray*}\label{schw}
\mathcal{F}\tilde{\varphi}_j\left( \xi\right)=\hat{\tilde{\varphi}}_j\left( \xi\right)=\mathcal{F}\tilde{\varphi}\left( N^{-j}\xi\right), \quad \xi\in\mathbf{R}^d, j\geq 1,
\end{eqnarray*}
where
\begin{eqnarray*}\label{schwa}
\mathcal{F}\tilde{\varphi}\left( \xi\right)=\phi\left( N\xi\right)+\phi\left( \xi\right)+\phi\left( N^{-1}\xi\right).
\end{eqnarray*}
Note that $\phi$ is necessarily $0$ on the boundary of its support. Then,
\begin{eqnarray*}
	\mathcal{F}\varphi_j=\mathcal{F}\varphi_j\mathcal{F}\tilde{\varphi}_j, j\geq 0,
\end{eqnarray*}
and then
\begin{eqnarray}\label{convol}
\varphi_j=\varphi_j\ast\tilde{\varphi}_j, j\geq 0,
\end{eqnarray}
where in particular
\begin{eqnarray*}
	\tilde{\varphi}_j\left( x\right) =N^{jd}\tilde{\varphi}\left( N^{j}x\right), j\geq 1.
\end{eqnarray*}
$\varphi_j,j\geq 0$ are convolution functions we use to define \textbf{generalized Besov spaces}. Namely, we write $\tilde{C}^{\beta}_{\infty,\infty}\left( \mathbf{R}^{d}\right) $ as the set of functions in $\mathcal{S}'\left( \mathbf{R}^d\right)$ for which the norm
\begin{eqnarray*}
	\left\vert u\right\vert_{\beta,\infty}:=\sup_{j} w\left( N^{-j}\right)^{-\beta}\left\vert u\ast \varphi_j\right\vert _{0}<\infty, \quad \beta\in\left( 0,\infty\right).
\end{eqnarray*}

For $\kappa\in\left[ 0,1\right],\beta\in\left(0,\infty\right)$, $C^{\mu,\kappa,\beta}\left( \mathbf{R}^{d}\right) $ denotes the collection of functions in $\mathcal{S}'\left( \mathbf{R}^d\right)$ whose norm
\begin{eqnarray*}
	\left\vert u\right\vert _{\mu,\kappa,\beta }:=\left\vert u\right\vert _{0}+\left\vert\mathcal{F}^{-1}\left[ \psi^{\mu,\kappa}\mathcal{F}u\right]\right\vert_{\beta,\infty}=\left\vert u\right\vert _{0}+\left\vert L^{\mu,\kappa}u\right\vert_{\beta,\infty}<\infty,
\end{eqnarray*}
where 
\begin{eqnarray*}
	\psi^{\mu}\left(\xi\right):=\int \left[ e^{i2\pi \xi\cdot y}-1-i2\pi\chi_{\alpha}\left( y\right)\xi\cdot y\right]\mu\left( dy\right), \xi\in\mathbf{R}^d,
\end{eqnarray*}
is the L\'{e}vy symbol associated to $L^{\mu}$, 
\begin{eqnarray*}
\psi^{\mu,\kappa}:=\left\{\begin{array}{ll}
	\psi^{\mu} & \mbox{ if } \kappa=1,\\
	-\left(-\Re\psi^{\mu}\right)^{\kappa}  & \mbox{ if } \kappa\in\left( 0,1\right),\\
	1 & \mbox{ if } \kappa=0,
	\end{array}\right.
\end{eqnarray*}
and 
\begin{eqnarray}\label{opp}
	L^{\mu,\kappa}u:=\mathcal{F}^{-1}\left[ \psi^{\mu,\kappa}\mathcal{F}u\right], u\in \mathcal{S}'\left( \mathbf{R}^d\right).
\end{eqnarray}
When $\kappa=0$, $L^{\mu,\kappa}u:=u$, then $C^{\mu,\kappa,\beta}\left( \mathbf{R}^{d}\right) $ is $\tilde{C}^{\beta}_{\infty,\infty}\left( \mathbf{R}^{d}\right) $. When $\kappa=1$, we simply write $L^{\mu,\kappa}=L^{\mu}$ and write $\left\vert u\right\vert _{\mu,\kappa,\beta }$ as $\left\vert u\right\vert _{\mu,\beta }$. In this case, if $u\in  C_b^{1}\left(\mathbf{R}^d\right)$, then definition $\eqref{opp}$ coincides with
\begin{eqnarray}
L^{\mu}\varphi\left( x\right):=\int\left[ \varphi\left( x+y\right)-\varphi\left( x\right)-\chi_{\alpha}\left( y\right)y\cdot \nabla \varphi\left( x\right)\right]\mu\left(d y\right).
\end{eqnarray}

For $\kappa\in\left[ 0,1\right],\beta\in\left(0,\infty\right)$, $\tilde{C}^{\mu,\kappa,\beta}\left( \mathbf{R}^{d}\right) $ is the class of functions in $\mathcal{S}'\left( \mathbf{R}^d\right)$ whose norm
\begin{eqnarray*}
	\left\Vert u\right\Vert _{\mu,\kappa,\beta }:=\left\vert \left( I-L^{\mu}\right)^{\kappa}u\right\vert_{\beta,\infty}<\infty,
\end{eqnarray*}
where 
\begin{eqnarray*}
	\left( I-L^{\mu}\right)^{\kappa}u:=\left\{\begin{array}{ll}
		\left( I-L^{\mu}\right)u & \mbox{ if } \kappa=1,\\
		\mathcal{F}^{-1}\left[ \left(1-\Re\psi^{\mu}\right)^{\kappa}\mathcal{F}u\right]  & \mbox{ if } \kappa\in\left[ 0,1\right).
	\end{array}\right.
\end{eqnarray*}
When $\kappa=0$, $\left( I-L^{\mu}\right)^{\kappa}u:=u$, then $\tilde{C}^{\mu,\kappa,\beta}\left( \mathbf{R}^{d}\right) $ is again $\tilde{C}^{\beta}_{\infty,\infty}\left( \mathbf{R}^{d}\right)$. When $\kappa=1$, we simply write $\left\Vert u\right\Vert _{\mu,\kappa,\beta }$ as $\left\Vert u\right\Vert _{\mu,\beta }$. 

We will see in Section 3 that $L^{\mu,\kappa}$ and $\left( I-L^{\mu}\right)^{\kappa}$ could be defined for functions in $C_b^{\infty}\left( \mathbf{R}^{d}\right) $ even if $\kappa\in\left( 1,2\right)$. That is 
\begin{eqnarray*}
L^{\mu,\kappa}:=L^{\mu,\kappa/2}\circ L^{\mu,\kappa/2},\quad \left( I-L^{\mu}\right)^{\kappa}:=\left( I-L^{\mu}\right)^{\kappa/2}\circ \left( I-L^{\mu}\right)^{\kappa/2},
\end{eqnarray*}
where $\circ$ means composition. 

There will also be \textbf{generalized H\"{o}lder spaces}. Using the scaling function, we write for $\beta \in \left( 0,1/\alpha\right)$
\begin{eqnarray*}
	\left\vert u\right\vert _{0} &=&\sup_{t,x}\left\vert u\left( t,x\right)
	\right\vert , \\
	\left[ u\right] _{\beta } &=&\sup_{t,x,h\neq 0}\frac{\left\vert u\left(
		t,x+h\right) -u\left( t,x\right) \right\vert }{w\left( \left\vert h\right\vert\right)
		^{\beta }}.
\end{eqnarray*}
$\tilde{C}^{\beta}\left( \mathbf{R}^{d}\right) $ denotes the set of functions such that the norm
\begin{eqnarray*}
	\left\vert u\right\vert _{\beta }:=\left\vert u\right\vert
	_{0}+\left[
	u\right] _{\beta }<\infty , \quad\beta \in \left( 0,1/\alpha\right).
\end{eqnarray*}
And $\tilde{C}^{1+\beta}\left( \mathbf{R}^{d}\right) $ denotes the set of functions such that the norm
\begin{eqnarray*}
	\left\vert u\right\vert_{1+\beta }:=\left\vert u\right\vert
	_{0}+\left\vert L^{\mu}u\right\vert
	_{0}+\left[L^{\mu}
	u\right]_{\beta }<\infty,\quad\beta \in \left( 0,1/\alpha\right).
\end{eqnarray*}

\subsection{Assumptions and Examples}

All the assumptions needed in this paper are collected in this section. Because of their dependence on $w,l$, we denote them by \textbf{A(w,l)}. Let $\nu$ be a L\'{e}vy measure, i.e.
\begin{eqnarray*}
\int_{\mathbf{R}^d_0} \left( 1\wedge \left\vert y\right\vert^2\right) \nu\left( dy\right)<\infty.
\end{eqnarray*}
Recall definitions $\eqref{mea}$ and $\eqref{meas}$.

\noindent\textbf{A(w,l)}. (i) For all $R>0$, $\tilde{\nu}_R\left( dy\right)\geq \mu^0 \left( dy\right)$, where $\mu^0$ is a L\'{e}vy measure supported on the unit ball $B\left( 0\right)$ and 
\begin{equation}\label{mu0}
\int \left\vert y\right\vert^2 \mu^0\left( dy\right) +\int \left\vert\xi\right\vert^4\left[1+\upsilon\left( \xi\right)\right]^{d+3}\exp\{-\zeta^0\left( \xi\right) \}d\xi<\infty,
\end{equation} 
in which
\begin{eqnarray*}
	&&\upsilon\left( \xi\right)=\int \chi_{\alpha}\left( y\right) \left\vert y\right\vert \left[ \left( \left\vert \xi\right\vert\left\vert y\right\vert\right)\wedge 1\right]\mu^0\left(dy\right),\\
	&&\quad\zeta^0\left( \xi\right) =\int \left[ 1-\cos \left( 2\pi\xi\cdot y\right)\right]\mu^0\left( dy\right).
\end{eqnarray*}
In addition, for all $\xi\in S_{d-1}=\{\xi\in\mathbf{R}^d:\left\vert\xi\right\vert=1 \}$, there is a constant $c_1>0$, such that
\begin{equation*}
\int_{\left\vert y\right\vert\leq 1} \left\vert \xi\cdot y\right\vert^2 \mu^0\left( dy\right)\geq c_0.
\end{equation*} 
(ii) If $\alpha=1$, then 
\begin{equation}\label{alpha1}
\int_{r<\left\vert y\right\vert <R} y\nu\left( dy\right)=0 \quad\text{ for all } 0<r<R<\infty.
\end{equation}
(iii) There exist constants $\alpha_1\geq\alpha_2$ such that $\alpha_1,\alpha_2\in\left( 0,1\right)$ if $\alpha\in\left( 0,1\right)$, $\alpha_1,\alpha_2\in \left( 1,2\right]$ if $\alpha\in\left( 1,2\right)$, $\alpha_1\in\left( 1,2\right]$ and $\alpha_2\in\left[0,1\right)$ if $\alpha=1$, and 
\begin{equation*}
\int_{\left\vert y\right\vert\leq 1} \left\vert y\right\vert^{\alpha_1}\tilde{\nu}_R\left( dy\right)+\int_{\left\vert y\right\vert> 1} \left\vert y\right\vert^{\alpha_2}\tilde{\nu}_R\left( dy\right)\leq N_0
\end{equation*}
for some positive constant $N_0$ that is independent of $R$. \\
\noindent(iv) $\varsigma\left( r\right):=\nu\left( \left\vert y\right\vert >r\right),r>0$ is continuous in $r$ and
\begin{eqnarray*}
	\int_0^{1}s \varsigma\left( rs\right)\varsigma\left( r\right)^{-1} ds\leq C_0,
\end{eqnarray*}
for some $C_0>0$ independent of $r$.

We assume both the reference measure $\mu$ and the operator measure $\nu$ satisfy \textbf{A(w,l)}.

Though looking heavy, \textbf{A(w,l)} embraces various models that have been receiving wide attention. For instance, in \cite{zh}, $\nu$ is confined by two $\alpha$-stable L\'{e}vy measures, namely, 
\begin{eqnarray}\label{as1}
&&\int_{S_{d-1}}\int_{0}^{\infty }1_{B}\left(
rw\right) \frac{dr}{r^{1+\alpha }}\Sigma _{1}\left( dw\right) \nonumber\\
&\leq &\nu \left( B\right) \leq \int_{S_{d-1}}\int_{0}^{\infty }1_{B}\left( rw\right) \frac{dr}{%
	r^{1+\alpha }}\Sigma _{2}\left( dw\right)
\end{eqnarray}
for any Borel measurable set $B$, where $\Sigma_1$ and $\Sigma_2$ are two finite measures defined on the unit sphere and $\Sigma_1$ is nondegenerate. As a result, $\nu$ satisfies \textbf{ A(w,l)} for $w\left( r\right)=l\left( r\right)=r^{\alpha},r>0$. 

To see some other examples, let us adopt for now the radial and angular coordinate system and write $\nu$ as 
\begin{equation}\label{mod}
\nu\left( B\right)=\int_0^{\infty}\int_{\left\vert w\right\vert=1} 1_B\left( rw\right)a\left( r,w\right)j\left(r\right)r^{d-1}S\left( dw\right)dr, \quad \forall B\in\mathcal{B}\left( \mathbf{R}^d_0\right),
\end{equation}
where $S\left( dw\right)$ is a finite measure on the unit sphere. 

Suppose $\Lambda\left( dt\right)$ is a measure on $\left(0,\infty\right)$ such that $\int_0^{\infty} \left( 1\wedge t\right) \Lambda\left( dt\right)<\infty$, and $\phi\left(r\right)=\int_0^{\infty}\left( 1-e^{-rt}\right)\Lambda\left( dt\right), r\geq 0$ is the associated Bernstein function. Set in $\eqref{mod}$ $S\left(dw\right)$ to be the usual Lebesgue measure, $a\left( r,w\right)=1$, and
\begin{equation*}
j\left( r\right)=\int_0^{\infty}\left( 4\pi t\right)^{-d/2}\exp\left( -\frac{r^2}{4t}\right)\Lambda\left( dt\right), r>0.
\end{equation*}
Futhermore, assume\\
\textbf{H. }(i) There is $C>1$ such that
\begin{equation*}
C^{-1}\phi\left( r^{-2}\right)r^{-d}\leq j\left(r\right)\leq C\phi\left( r^{-2}\right)r^{-d}.
\end{equation*}
(ii) There are $0<\delta_1\leq \delta_2<1$ and $C>0$ such that for all $0<r\leq R$
\begin{equation*}
C^{-1}\left( \frac{R}{r}\right)^{\delta_1}\leq \frac{\phi\left( R\right)}{\phi\left( r\right)}\leq C\left( \frac{R}{r}\right)^{\delta_2}.
\end{equation*}
\textbf{G. }There is a function $\rho_0\left(w\right)$ defined on the unit sphere such that $\rho_0\left(w\right)\leq a\left( r,w\right)\leq 1, \forall r>0$, and for all $\left\vert \xi\right\vert=1$,
\begin{equation*}
\int_{S^{d-1}}\left\vert \xi\cdot w\right\vert^2\rho_0\left( w\right) \geq c>0.
\end{equation*}

Options for such $\Lambda$ and thus $\phi$ could be\\
(1) $\phi\left( r\right) =\Sigma_{i=1}^{n} r^{\alpha_i}, \alpha_i\in\left( 0,1\right), i=1,\ldots,n$;\\
(2) $\phi\left( r\right) =\left( r+r^{\alpha}\right)^{\beta}, \alpha,\beta\in\left( 0,1\right)$;\\
(3) $\phi\left( r\right) =r^{\alpha}\left(\ln\left(1+r\right)\right)^{\beta}, \alpha\in\left( 0,1\right),\beta\in\left( 0,1-\alpha\right)$;\\
(4) $\phi\left( r\right) =\left[ \ln\left( \cosh \sqrt{r}\right)\right]^{\alpha}, \alpha\in\left( 0,1\right)$.

It can be shown that Assumptions \textbf{H} and \textbf{G} offer us quite a few candidates of the \textbf{A(w,l)}, with the scaling function $w\left( r\right)=j\left( r\right)^{-1}r^{-d}, r>0$ and the scaling factor
\begin{equation*}
l\left( r\right)=\left\{\begin{array}{ll}
Cr^{2\delta_1} & \mbox{ if } r\leq 1,\\
Cr^{2\delta_2} & \mbox{ if } r> 1
\end{array}\right.
\end{equation*}
for some $C>0$. (See \cite{cr} for details.) In \cite{cr} and \cite{cr2}, Cauchy problems have been considered in the $L^p$-space and $H^{\mu,s}_p$-space respectively which are defined by L\'{e}vy measures from the \textbf{A(w,l)} class.

\subsection{More Discussion about the Model}

Some estimates on magnitude of the scaling function $w$ and the scaling factor $l$ can be extracted merely from their definitions. 
\begin{lemma}\label{kl}
	Let $w: \left( 0,\infty\right)\rightarrow \left( 0,\infty\right)$ be a scaling function and $l$ be an associated scaling factor which satisfies $l\left( N^{-1}\right)<1<l\left( N\right)$. $r_1=\inf \{r>0: N^r\geq l\left( N\right)\}$, $r_2=\sup \{r>0: N^{-r}\geq l\left( N^{-1}\right)\}$. Then
	
	\noindent(i) there exist $c_0, C_0>0$ such that 
	\begin{equation*}
	c_0\leq w\left( x\right)\leq C_0, \quad \forall x\in\left[ N^{-1},N\right],
	\end{equation*}
	(ii) $r_1\geq r_2$, and for $c_0, C_0$ above, 
	\begin{eqnarray*}
		c_0\left( x^{r_1}\wedge x^{r_2}\right)\leq w\left( x\right)\leq C_0 \left( x^{r_1}\vee x^{r_2}\right), \quad \forall x\in\mathbf{R}_{+},
	\end{eqnarray*}
	(iii) for the same $c_0$ and $C_0$,
	\begin{eqnarray*}
		l\left( x\right)\geq \frac{c_0}{C_0} \left( x^{r_1}\wedge x^{r_2}\right), \quad \forall x\in\mathbf{R}_{+}.
	\end{eqnarray*}
(iv) $\gamma\left(x\right):=\inf\{s:l\left( s\right)\geq x\}$. For the same $c_0$ and $C_0$,
\begin{eqnarray*}
	\gamma\left( x\right)\leq \frac{C_0}{c_0} \left( x^{\frac{1}{r_1}}\vee x^{\frac{1}{r_2}}\right), \quad \forall x\in\mathbf{R}_{+}.
\end{eqnarray*}
\end{lemma}
\begin{proof}
	(i) Utilize $\eqref{scale}$ and monotonicity of $l$, for  $x\in\left[ N^{-1},N\right]$,
	\begin{eqnarray*}
		l\left(N\right)^{-1}w\left( 1\right)\leq l\left(x^{-1}\right)^{-1}w\left( 1\right)\leq w\left( x\right)\leq l\left( x\right)w\left( 1\right)\leq l\left( N\right)w\left( 1\right).
	\end{eqnarray*}
	Set $c_0=l\left(N\right)^{-1}w\left( 1\right)$, $C_0=l\left( N\right)w\left( 1\right)$.
	
	(ii) Apply $\eqref{scale}$ iteratively. For $x\in\left[ N^{j},N^{j+1}\right], \forall j\in\mathbf{N}$, we have
	\begin{eqnarray*}
		c_0 x^{r_2}\leq l\left(N^{-1}\right)^{-j-1}w\left( N^{-j-1}x\right)\leq w\left( x\right)\leq l\left( N\right)^{j}w\left( N^{-j}x\right)\leq C_0 x^{r_1}.
	\end{eqnarray*}
	Since this is true for arbitrarily large $x$, we can conclude $r_1\geq r_2$. For $x\in\left[ N^{-j-1},N^{-j}\right], \forall j\in\mathbf{N}$,
	\begin{eqnarray*}
		c_0 x^{r_1}\leq l\left(N\right)^{-j}w\left( N^{-j}x\right)\leq w\left( x\right)\leq l\left( N^{-1}\right)^{j+1}w\left( N^{j+1}x\right)\leq C_0 x^{r_2}.
	\end{eqnarray*}
	As a summary, 
	\begin{eqnarray*}
		c_0\left( x^{r_1}\wedge x^{r_2}\right)<w\left( x\right)<C_0 \left( x^{r_1}\vee x^{r_2}\right), \quad \forall x\in\mathbf{R}_{+}.
	\end{eqnarray*}
	
	(iii) By (ii) and $\eqref{scale}$,
	\begin{eqnarray*}
		l\left( x\right)\geq w\left( 1\right)^{-1} w\left( x\right)\geq \frac{c_0}{C_0}\left( x^{r_1}\wedge x^{r_2}\right), \quad \forall x\in\mathbf{R}_{+}.
	\end{eqnarray*}

(iv) is a direct conclusion from (iii).
\end{proof}

\noindent\textbf{Remark:} Suggested by the bounds in (ii), we redefine
\begin{eqnarray*}
r_1&=&\inf \left\{ \sigma>0:\limsup_{r\rightarrow 0}\frac{%
	r^{\sigma }}{w \left( r\right) }=0\right\}\vee\sup \left\{ \sigma>0:\liminf_{r\rightarrow \infty}\frac{%
	r^{\sigma }}{w \left( r\right) }=0\right\},\\
r_2&=&\sup \left\{ \sigma>0:\liminf_{r\rightarrow 0}\frac{%
	r^{\sigma }}{w \left( r\right) }=\infty\right\}\wedge\inf \left\{ \sigma>0:\limsup_{r\rightarrow \infty}\frac{%
	r^{\sigma }}{w \left( r\right) }=\infty\right\}.
\end{eqnarray*}
Namely, we take the smallest possbile $r_1$ and the largest possible $r_2$ such that (ii) holds. In Lemma \ref{order}, we shall show that the order $\alpha\in\left[ r_2,r_1\right]$. In $\alpha$-stable-like examples, $r_1=r_2=\alpha$. For models illustrated by Bernstein functions, $r_1\leq 2\delta_2, r_2\geq 2\delta_1$. (iii) and (iv) still hold with the amended parameters.

It can be told from the next lemma that the scaling function is rather genetic to the measure $\nu$, if \textbf{A(w,l)} holds for $\nu$.

\begin{lemma}\label{ess}
	Let $\nu$ be a L\'{e}vy measure and $w$ be the scaling function which $\nu$ satisfies \textbf{A(w,l)} for. Then,\\
	\noindent a) there are constants $C_1, C_2>0$ such that 
	\begin{eqnarray}
	C_1\varsigma\left( r\right)\leq w\left( r\right)^{-1}\leq C_2 \varsigma\left( r\right), \quad\forall r>0.
	\end{eqnarray}
	\noindent b) $\int_{\left\vert y\right\vert \leq 1}w\left( \left\vert y\right\vert
	\right) \nu\left( dy\right)=+\infty$.\\
	\noindent c) For any $\varepsilon>0$, $\int_{\left\vert y\right\vert \leq 1}w \left( \left\vert y\right\vert
	\right) ^{1+\varepsilon }\nu\left( dy\right) <\infty$.\\
	\noindent d) For any $\varepsilon >0$, $\int_{\left\vert y\right\vert \leq 1}\left\vert y\right\vert ^{\varepsilon
	}w\left( \left\vert y\right\vert \right) \nu\left( dy\right) <\infty$.
\end{lemma}
\begin{proof}
a) First,
\begin{eqnarray}\label{11}
w\left( r\right)^{-1}\int_{\left\vert y\right\vert>1}\tilde{\nu}_r\left(  dy\right)=\int_{\left\vert y\right\vert>r}\nu\left(  dy\right)=\varsigma\left( r\right),\quad \forall r>0,
\end{eqnarray}
then by (iii) in \textbf{A(w,l)},
\begin{eqnarray}\label{10}
\varsigma\left( r\right)\leq C w\left( r\right)^{-1}, \quad\forall r>0.
\end{eqnarray}

On the other hand, for all $r>0$,
\begin{equation}\label{12}
	w\left( r\right)^{-1}\int_{\left\vert y\right\vert\leq 1}\left\vert y\right\vert^2\tilde{\nu}_r\left(  dy\right)=r^{-2}\int_{\left\vert y\right\vert\leq r}\left\vert y\right\vert^2\nu\left(  dy\right)=-r^{-2}\int_0^{r}s^2 d\varsigma\left( s\right).
\end{equation}
By the generalized formula of integration by parts from stochastic calculus, 
\begin{eqnarray}
	&&\int_0^{r}s^2 d\varsigma\left( s\right)=\lim_{\varepsilon\rightarrow 0}\int_{\varepsilon}^{r}s^2 d\varsigma\left( s\right)\nonumber\\
	&=&\lim_{\varepsilon\rightarrow 0} \left( s^2 \varsigma\left( s\right)\big\arrowvert_{\varepsilon}^r-2\int_{\varepsilon}^{r}s \varsigma\left( s\right)ds\right)\nonumber\\
	&=& r^2 \varsigma\left( r\right)-2\int_{0}^{r}s \varsigma\left( s\right)ds.\label{13}
\end{eqnarray}
Note that in above derivation, we used the fact that $\lim_{\varepsilon\rightarrow 0}  \varepsilon^2 \varsigma\left( \varepsilon\right)=0$. This is due to \textbf{A(w,l)}(i), $\eqref{10}$ and $\eqref{12}$, which implies
\begin{eqnarray*}
\lim_{\varepsilon\rightarrow 0}  \varepsilon^2 \varsigma\left( \varepsilon\right)\leq C\lim_{\varepsilon\rightarrow 0}  \varepsilon^2 w\left( \varepsilon\right)^{-1}\leq C\lim_{\varepsilon\rightarrow 0}\int_{\left\vert y\right\vert\leq \varepsilon}\left\vert y\right\vert^2\nu\left(  dy\right)=0.
\end{eqnarray*}
Combine $\eqref{11}$, $\eqref{12}$ and $\eqref{13}$.
\begin{eqnarray}\label{14}
\qquad\quad w\left( r\right)^{-1}\int\left(\left\vert y\right\vert^2\wedge 1\right)\tilde{\nu}_r\left(  dy\right)=2r^{-2}\int_0^{r}s \varsigma\left( s\right) ds=2\int_0^{1}s \varsigma\left( rs\right) ds.
\end{eqnarray}
Again by \textbf{A(w,l)}(i),
\begin{eqnarray*}
w\left( r\right)^{-1}\leq C\varsigma\left( r\right)\int_0^{1}s \varsigma\left( rs\right)\varsigma\left( r\right)^{-1} ds\leq C\varsigma\left( r\right).
\end{eqnarray*}
\noindent b) Use a) and the It\^{o} formula.
\begin{eqnarray*}
&&\int_{\left\vert y\right\vert \leq 1}w\left( \left\vert y\right\vert
\right) \nu\left( dy\right)=-\int_0^1 w\left( r
\right)d\varsigma\left(r
\right)\\
&\leq& -C\int_0^1 \varsigma\left( r
\right)^{-1}d\varsigma\left(r
\right)=-C\lim_{\varepsilon\rightarrow 0}\int_{\varepsilon}^1 \varsigma\left( r
\right)^{-1}d\varsigma\left(r
\right)\\
&=&C\left( \lim_{\varepsilon\rightarrow 0}\ln \varsigma\left( \varepsilon\right)-\ln \varsigma\left( 1\right)\right)=\infty.
\end{eqnarray*}

\noindent c) For any $\varepsilon >0$,
\begin{equation*}
\int_{\left\vert y\right\vert \leq 1}w\left( \left\vert y\right\vert
\right) ^{1+\varepsilon }\nu\left( dy\right) \leq C\int_{0}^{1}\varsigma \left( r\right)
^{-1-\varepsilon }d\varsigma \left( r\right) =C\varsigma \left( r\right)
^{-\varepsilon }|_{0}^{1}<\infty .
\end{equation*}
d) By c), for any $\varepsilon >0$, $\sigma'>\inf \left\{ \sigma\in\left( 0,2\right):\limsup_{r\rightarrow 0}\frac{r^{\sigma }}{w \left( r\right) }=0\right\}$,
	\begin{eqnarray*}
	\int_{\left\vert y\right\vert \leq 1}\left\vert y\right\vert ^{\varepsilon
	}w \left( \left\vert y\right\vert \right) \nu\left( dy\right) &=&\int_{\left\vert
		y\right\vert \leq 1}\left( \frac{\left\vert y\right\vert ^{\sigma ^{\prime }}%
	}{w\left( \left\vert y\right\vert \right) }\right) ^{\frac{\varepsilon 
		}{\sigma ^{\prime }}}w\left( \left\vert y\right\vert \right) ^{1+\frac{%
			\varepsilon }{\sigma ^{\prime }}}\nu\left( dy\right)\\
	&\leq &C\int_{\left\vert y\right\vert \leq 1}w \left( \left\vert
	y\right\vert \right) ^{1+\frac{\varepsilon }{\sigma ^{\prime }}}\nu\left( dy\right) <\infty.
\end{eqnarray*}
\end{proof}

\begin{lemma}\label{order}
	Let $\nu$ be a L\'{e}vy measure and $w$ be the scaling function which $\nu$ satisfies \textbf{A(w,l)} for. $\alpha$ is the order of $\nu$. Then  
	 \begin{equation*}
	 \alpha=\inf \left\{ \sigma:\limsup_{r\rightarrow 0}\frac{%
	 	r^{\sigma }}{w \left( r\right) }=0\right\} .
	 \end{equation*}%
\end{lemma}
\begin{proof}
	Denote $\alpha'=\inf \left\{ \sigma\in\left( 0,\infty\right):\limsup_{r\rightarrow 0}\frac{%
		r^{\sigma }}{w \left( r\right) }=0\right\}$. We first show that $\alpha'\leq \alpha$. Note if $\sigma\in \left(0,\alpha'\right)$, $\limsup_{r\rightarrow 0}\frac{
		r^{\sigma }}{w \left( r\right) }=\infty$. Otherwise
	\begin{eqnarray*}
	\limsup_{r\rightarrow 0}\frac{
		r^{\frac{\sigma+\alpha'}{2 }}}{w \left( r\right) }\leq \limsup_{r\rightarrow 0}\frac{
		r^{\sigma }}{w \left( r\right) }\lim_{r\rightarrow 0}
		r^{\frac{\alpha'-\sigma}{2} }=0,
	\end{eqnarray*}
	which contradicts with the definition of $\alpha'$. Now take $0<r\leq 1$. For any $\sigma\in \left(\alpha,\infty\right)$,
		\begin{eqnarray*}
		\int_{\left\vert y\right\vert \leq 1}\left\vert y\right\vert
		^{\sigma }d\nu\geq\int_{\left\vert y\right\vert \leq
			r}\left\vert y\right\vert ^{\sigma}d\nu=\frac{r^{\sigma }}{w
			\left( r\right) }\int_{\left\vert y\right\vert \leq 1}\left\vert
		y\right\vert ^{\sigma}d\tilde{\nu}_{r}\geq \frac{r^{\sigma }}{w\left( r\right) }\int_{\left\vert
			y\right\vert \leq 1}\left\vert y\right\vert ^{2}d\tilde{\nu}_{r}.
	\end{eqnarray*}
By (i) in \textbf{A(w,l)}, $\{\tilde{\nu}_{r}:r>0\}$ are non-degenerate. Hence 
\begin{eqnarray*}
c\frac{r^{\sigma }}{w\left( r\right) }\leq\frac{r^{\sigma }}{w\left( r\right) }\int_{\left\vert
	y\right\vert \leq 1}\left\vert y\right\vert ^{2}d\tilde{\nu}_{r}\leq \int_{\left\vert
		y\right\vert \leq 1}\left\vert y\right\vert ^{\sigma}d\nu<C
\end{eqnarray*}
for some $c,C>0$ independent of $r$. Thus $\alpha'\leq \sigma $, and thus $\alpha'\leq\alpha$.

For the other direction, assume to the contrary $\alpha' <\alpha$. Then by Lemma \ref{ess}, for $\alpha' <\sigma'<\sigma<\alpha$,
\begin{eqnarray*}
	\int_{\left\vert y\right\vert \leq 1}\left\vert y\right\vert ^{\sigma
	}d\nu &=&\int_{\left\vert y\right\vert \leq 1}
	\frac{\left\vert y\right\vert ^{\sigma ^{\prime }}}{w\left( \left\vert
		y\right\vert \right) }\left\vert y\right\vert ^{\sigma-\sigma ^{\prime }}w\left( \left\vert y\right\vert \right) d\nu \\
	&\leq &C\int_{\left\vert y\right\vert \leq 1}\left\vert y\right\vert
	^{\sigma-\sigma ^{\prime }}w\left( \left\vert
	y\right\vert \right) d\nu <\infty.
\end{eqnarray*}%
But this contradicts with the definition of $\alpha$. Therefore, $\alpha\leq\alpha'$. Combining the argument above, we obtain $\alpha'=\alpha$.
\end{proof}

According to Lemma \ref{order}, we can claim immediately that two L\'{e}vy measures which satisfy \textbf{A(w,l)} for the same $w,l$ have the same order. 

The last two lemmas of this section explain why we restricted the H\"{o}lder order $\beta \in \left( 0,1/\alpha\right)$ when defining generalized H\"{o}lder spaces in section 2.2. If $\beta$ exceeds or equals to $1/\alpha$, the function may reduce to a trivial case that is no longer of interest.
\begin{lemma}\label{es1}
	Set $\alpha'=\inf \left\{ \sigma\in\left( 0,2\right):\limsup_{r\rightarrow 0}\frac{%
		r^{\sigma }}{w \left( r\right) }=0\right\}$. \\
	\noindent a) If $0<\frac{1}{\beta }<\alpha'$ and $\sup_{x,y}\frac{\left\vert f\left( x\right) -f\left( x+y\right) \right\vert }{%
		w \left( \left\vert y\right\vert \right) ^{\beta }}<\infty $, then $f$ is a constant.\\
	
	\noindent b) If $\frac{1}{\beta }>\alpha'$ and $f$ is a bounded Lipschitz function, then for each $\varepsilon \in \left(
	0,1\right) $, there is a positive constant $C_{\varepsilon}$ depending on $\varepsilon$ but independent of $f$ so that%
	\begin{equation*}
	\sup_{x,y}\frac{\left\vert f\left( x\right)
		-f\left( x+y\right) \right\vert }{w\left( \left\vert y\right\vert
		\right)^{\beta }}\leq \varepsilon \sup_{x,y}%
	\frac{\left\vert f\left( x+y\right) -f\left( x\right) \right\vert }{%
		\left\vert y\right\vert }+C_{\varepsilon}\left\vert f\right\vert_{0}.
	\end{equation*}%
	Namely, the space $\tilde{C}^{\beta }$
	contains all bounded Lipschitz functions.
\end{lemma}

\begin{proof}
	a) Let $\varepsilon\in\left(0,\beta\alpha'-1\right)$, $\beta'\left( 1+\varepsilon \right) =\beta$. Then  $\frac{1}{\beta }\leq \frac{1}{\beta'}<\alpha'$ and we can find a sequence $y_{n}\rightarrow 0$ so that $\frac{\left\vert
		y_{n}\right\vert ^{1/\beta ^{\prime }}}{w \left( \left\vert
		y_{n}\right\vert \right) }\geq C>0$. Let $f_{\varepsilon }=f\ast w_{\varepsilon}$ where $w_{\varepsilon }$ is the standard
	mollifier. We then have%
	\begin{eqnarray*}
		\sup_{x,y}\frac{\left\vert f\left( x\right) -f\left( x+y\right) \right\vert }{%
			w \left( \left\vert y\right\vert \right) ^{\beta }} &\geq&C\frac{\left\vert f_{\varepsilon }\left( x\right) -f_{\varepsilon }\left(
			x+y_{n}\right) \right\vert }{w\left( \left\vert y_{n}\right\vert
			\right) ^{\beta }} \\
		&=&C\frac{\left\vert f_{\varepsilon }\left( x\right) -f_{\varepsilon }\left(
			x+y_{n}\right) \right\vert }{\left\vert y_{n}\right\vert ^{1+\varepsilon }}%
		\left( \frac{\left\vert y_{n}\right\vert ^{\frac{1}{\beta ^{\prime }}}}{%
			w\left( \left\vert y_{n}\right\vert \right) }\right) ^{\beta ^{\prime
			}(1+\varepsilon )} \\
		&\geq &C\frac{\left\vert f_{\varepsilon }\left( x\right) -f_{\varepsilon
			}\left( x+y_{n}\right) \right\vert }{\left\vert y_{n}\right\vert
			^{1+\varepsilon }},\quad x\in \mathbf{R}^{d}.
	\end{eqnarray*}%
	Hence $\nabla f_{\varepsilon }\left( x\right) =0,x\in \mathbf{R}^{d},\forall \varepsilon \in \left(0,\beta\alpha'-1\right)$, which implies $ f_{\varepsilon }\left( x\right) =C_{\varepsilon}, \forall \varepsilon \in \left(0,\beta\alpha'-1\right)$. Obviously, $f$ is continuous, then $f_{\varepsilon }\rightarrow f$ uniformly on any compact subsets, and thus $f$ is a constant.
	
	b) Since $\limsup_{r\rightarrow 0}\frac{r^{1/\beta }}{w\left(
		r\right) }=0$, then for each $\varepsilon \in \left( 0,1\right) $ there is $%
	\delta >0$ so that $\frac{\left\vert y\right\vert ^{\frac{1}{\beta }}}{%
		w\left( \left\vert y\right\vert \right) }\leq \varepsilon^{\frac{1}{\beta }} $ if $%
	\left\vert y\right\vert \leq \delta .$ Hence,  
	\begin{eqnarray*}
		\frac{\left\vert f\left( x\right) -f\left( x+y\right) \right\vert }{w
			\left( \left\vert y\right\vert \right) ^{\beta }} &=&\frac{\left\vert
			f\left( x\right) -f\left( x+y\right) \right\vert }{\left\vert y\right\vert }%
		\left( \frac{\left\vert y\right\vert ^{\frac{1}{\beta }}}{w\left(
			\left\vert y\right\vert \right) }\right) ^{\beta } \\
		&\leq &\varepsilon \frac{\left\vert f\left( x\right) -f\left( x+y\right)
			\right\vert }{\left\vert y\right\vert }
	\end{eqnarray*}
	if $\left\vert y\right\vert
	\leq \delta$, and
	\begin{eqnarray*}
		\frac{\left\vert f\left( x\right) -f\left( x+y\right) \right\vert }{w
			\left( \left\vert y\right\vert \right) ^{\beta }} \leq 2w\left( \delta\right)^{-\beta}l\left( 1\right)^{\beta}\left\vert f\right\vert_0\leq C_{\varepsilon}\left\vert f\right\vert_0
	\end{eqnarray*}
	if $\left\vert y\right\vert
	> \delta$.
\end{proof}

\begin{lemma}\label{es2}
	Let $\alpha'=\inf \left\{ \sigma\in\left( 0,2\right):\limsup_{r\rightarrow 0}\frac{%
		r^{\sigma }}{w \left( r\right) }=0\right\}$. \\
	\noindent a)  If $\limsup_{r\rightarrow 0}\frac{r^{\alpha'}}{w\left(
		r\right) }=0$ and $f$ is a bounded Lipschitz function, then for each $\varepsilon \in \left( 0,1\right) $ there is a positive constant $C_{\varepsilon}$ depending on $\varepsilon$ but independent of $f$ so that
	\begin{equation*}
	\sup_{x,y}\frac{\left\vert f\left( x\right)
		-f\left( x+y\right) \right\vert }{w\left( \left\vert y\right\vert
		\right) ^{1/\alpha'}}\leq \varepsilon \sup_{x,y}\frac{\left\vert f\left( x+y\right) -f\left( x\right) \right\vert }{%
		\left\vert y\right\vert }+C_{\varepsilon}\left\vert f\right\vert _{0}.
	\end{equation*}
	Namely, the space $\tilde{C}^{1/\alpha'}$ contains all bounded Lipschitz functions.\\
	\noindent b) If $\limsup_{r\rightarrow 0}\frac{r^{\alpha'}}{w \left( r\right) }\in
	\left( 0,\infty \right) $, then $\tilde{C}^{1/\alpha'}$ is the space of bounded Lipschitz functions.\\
	\noindent c) If $\limsup_{r\rightarrow 0}\frac{r^{\alpha'}}{w\left(
		r\right) }=\infty $, then $\tilde{C}^{1/\alpha'}$ consists of
	constants only.
\end{lemma}
\begin{proof}
	a) The proof is identical to part b) of Lemma \ref{es1}.	\\
	\noindent b) Let $w_{\varepsilon }$ be a standard mollifier and $f_{\varepsilon }=f\ast w_{\varepsilon}$. For any $f\in \tilde{C}^{1/\alpha'}$, there is a sequence $y_{n}\rightarrow 0$ so that $\frac{\left\vert
		y_{n}\right\vert ^{\alpha'}}{w \left( \left\vert y_{n}\right\vert
		\right) }\geq c>0$. Then
	\begin{eqnarray*}
		\frac{%
			\left\vert f_{\varepsilon }\left( x\right) -f_{\varepsilon }\left(
			x+y_{n}\right) \right\vert }{w\left( \left\vert y_{n}\right\vert
			\right) ^{1/\alpha'}}&=&\frac{\left\vert f_{\varepsilon }\left( x\right)
			-f_{\varepsilon }\left( x+y_n\right) \right\vert }{\left\vert y_n\right\vert }%
		\left( \frac{\left\vert y_n\right\vert ^{\alpha'}}{w\left(
			\left\vert y_n\right\vert \right) }\right) ^{1/\alpha'} \\
		&\geq& c\frac{\left\vert f_{\varepsilon }\left( x\right) -f_{\varepsilon }\left(
			x+y_{n}\right) \right\vert }{\left\vert y_{n}\right\vert }.
	\end{eqnarray*}%
	Thus, $\left\vert \nabla f_{\varepsilon }\right\vert _{0}\leq C\left\vert
	f\right\vert _{1/\alpha'}$, and thus $\left\vert \nabla f\right\vert _{0}\leq C\left\vert
	f\right\vert _{1/\alpha'}$.
	
	On the other hand, $\frac{\left\vert y\right\vert ^{\alpha'}}{w
		\left( \left\vert y\right\vert \right) }\leq C$ if $\left\vert y\right\vert
	\leq 1$, then   
	\begin{eqnarray*}
		\frac{\left\vert f\left( x\right) -f\left(
			x+y\right) \right\vert }{w\left( \left\vert y\right\vert \right)
			^{1/\alpha'}} &=&\frac{\left\vert f\left( x\right)
			-f\left( x+y\right) \right\vert }{\left\vert y\right\vert }%
		\left( \frac{\left\vert y\right\vert ^{\alpha'}}{w\left(
			\left\vert y\right\vert \right) }\right) ^{1/\alpha'} \\
		&\leq &C\frac{\left\vert f\left( x\right) -f\left( x+y\right) \right\vert }{\left\vert y\right\vert }
	\end{eqnarray*}%
	if $\left\vert y\right\vert \leq 1$, and for $\left\vert y\right\vert> 1$,
	\begin{eqnarray*}
		\frac{\left\vert f\left( x\right) -f\left( x+y\right) \right\vert }{w
			\left( \left\vert y\right\vert \right)^{1/\alpha'}} \leq 2w\left( 1\right)^{-1/\alpha'}l\left( 1\right)^{1/\alpha'}\left\vert f\right\vert_0\leq C\left\vert f\right\vert_0.
	\end{eqnarray*}
	Hence $f\in \tilde{C}^{1/\alpha'}$ if $f$ is a bounded Lipschitz function.\\
	\noindent c) There is a sequence $\{y_{n}:n\in\mathbf{N}\}$ so that $y_{n}\rightarrow 0$ and for any $n\in\mathbf{N}$, $\frac{\left\vert y_k\right\vert ^{\alpha'}}{w
		\left( \left\vert y_k\right\vert \right) }\geq n$ if $k\geq n$. Then for $k\geq n$,
	\begin{equation*}
	n^{\frac{1}{\alpha'}}\frac{\left\vert f_{\varepsilon }\left( x\right) -f_{\varepsilon }\left(
		x+y_k\right) \right\vert }{\left\vert y_k\right\vert }\leq \frac{\left\vert
		f_{\varepsilon }\left( x\right)-f_{\varepsilon }\left( x+y_k\right)
		\right\vert }{w\left( \left\vert y_k\right\vert \right) ^{1/\alpha'}}%
	\leq \left\vert f_{\varepsilon }\right\vert_{1/\alpha'}\leq C\left\vert f\right\vert_{1/\alpha'}.
	\end{equation*}%
	Thus $\nabla f_{\varepsilon }=0$, $\forall x\in 
	\mathbf{R}^{d}$, $\forall\varepsilon \in \left( 0,1\right)$, and thus $f$ is a constant.
\end{proof}

\section{Characterization of Spaces and Norm Equivalence}

Our target spaces of general smoothness are $\tilde{C}^{\beta}$, $\tilde{C}^{\beta}_{\infty,\infty}$, $C^{\mu,\kappa,\beta}$, $\tilde{C}^{\mu,\kappa,\beta}$ endowed with norms $\left\vert \cdot\right\vert_{\beta}$, $\left\vert \cdot\right\vert_{\beta,\infty}$, $\left\vert\cdot\right\vert_{\mu,\kappa,\beta}$, $\left\Vert\cdot\right\Vert_{\mu,\kappa,\beta}$ respectively, and our goal in this section is to establish norm equivalence among them. 

\begin{lemma}\label{equiv}
	Let $\beta\in\left( 0,\infty\right)$. If $u\in \tilde{C}^{\beta}_{\infty,\infty}\left( \mathbf{R}^d\right)$, then $u\in C\left( \mathbf{R}^d\right)$ and $u\left( x\right)=\sum_{j=0}^{\infty}\left( u\ast\varphi_j\right)\left( x\right)$. Besides,
	\begin{equation}\label{sup}
	\left\vert u\right\vert_0\leq \sum_{j=0}^{\infty}\left\vert u\ast\varphi_j\right\vert_0\leq  C\left( \beta\right)\left\vert u\right\vert_{\beta,\infty}.
	\end{equation}
\end{lemma}
\begin{proof}
	Note that $u\ast\varphi_j\in C\left(\mathbf{R}^d\right), \forall j\in \mathbf{N}$, so is $\sum_{j=0}^{n} u\ast\varphi_j,\forall n\in \mathbf{N}_{+} $. Since 
	\begin{eqnarray*}
		\sum_{j=0}^{\infty}\left\vert u\ast\varphi_j\right\vert_0 &=& \sum_{j=0}^{\infty}w\left( N^{-j}\right)^{\beta}w\left( N^{-j}\right)^{-\beta}\left\vert u\ast\varphi_j\right\vert_0\\
		&\leq& \sup_{j\geq 0}w\left( N^{-j}\right)^{-\beta}\left\vert u\ast\varphi_j\right\vert_0\sum_{j=0}^{\infty}w\left( N^{-j}\right)^{\beta}\\
		&\leq& C\left\vert u\right\vert_{\beta,\infty}\sum_{j=0}^{\infty}l\left( N^{-1}\right)^{j\beta}<\infty,
	\end{eqnarray*}
	we have $\sum_{j=0}^{n} u\ast\varphi_j\rightarrow \sum_{j=0}^{\infty}u\ast\varphi_j$ uniformly in $\mathbf{R}^d$ as $n\to\infty$. Therefore, $\sum_{j=0}^{\infty}u\ast\varphi_j\in C\left(\mathbf{R}^d\right)$, and $\sum_{j=0}^{n} u\ast\varphi_j\xrightarrow{n\to\infty} \sum_{j=0}^{\infty}u\ast\varphi_j$ in the topology of $\mathcal{S}'\left(\mathbf{R}^d\right)$. By continuity of the Fourier transform, 
	\begin{equation*}
\mathcal{F}\left(\sum_{j=0}^{\infty}u\ast\varphi_j\right)=\lim_{n\rightarrow \infty}\mathcal{F}\left(\sum_{j=0}^{n} u\ast\varphi_j\right)= \lim_{n\rightarrow \infty}\sum_{j=0}^{n} \hat{u}\hat{\varphi_j}=\sum_{j=0}^{\infty} \hat{u}\hat{\varphi_j}=\hat{u}.
	\end{equation*}And therefore, $u=\sum_{j=0}^{\infty}u\ast\varphi_j\in C\left(\mathbf{R}^d\right)$.
\end{proof}

\begin{proposition}\label{pr2}
	Let $\beta\in\left( 0,1\right)$ and 
	\begin{eqnarray}\label{asl}
	\int_0^1 l\left( t\right)^{\beta}\frac{dt}{t}+\int_1^{\infty} l\left( t\right)^{\beta}\frac{dt}{t^2}<\infty.
	\end{eqnarray}
	Then norm $\left\vert u\right\vert_{\beta}$ and norm $\left\vert u\right\vert_{\beta,\infty}$ are equivalent. Namely, there is a constant positive $C$ depending only on $d,\beta, N$ such that
	\begin{eqnarray*}
		C^{-1}\left\vert u\right\vert_{\beta}\leq \left\vert u\right\vert_{\beta,\infty}\leq C\left\vert u\right\vert_{\beta}, \forall u\in C\left( \mathbf{R}^d\right).
	\end{eqnarray*}
\end{proposition}
\begin{proof}
	Suppose $\left\vert u\right\vert_{\beta}<\infty$. If $j=0$, then 
	\begin{equation*}
	w\left(1\right)^{-\beta}\left\vert u\ast\varphi_0\right\vert_0\leq w\left( 1\right)^{-\beta}\left\vert u\right\vert_0\int\left\vert\varphi_0\left(y\right)\right\vert dy\leq C\left\vert u\right\vert_{\beta}.
	\end{equation*}
	If $j\neq 0$, then by the construction of $\varphi_j$, $\int \varphi_j\left(y\right)dy=\hat{\varphi}_j\left( 0\right)=0$. Therefore,
	\begin{eqnarray*}
		&&w\left( N^{-j}\right)^{-\beta}\left\vert u\ast\varphi_j\right\vert_0\\
		&=& w\left( N^{-j}\right)^{-\beta}\left\vert \int \left[u\left(y\right)-u\left(x\right)\right]\varphi_j\left(x-y\right)dy\right\vert_0\\
		&\leq& w\left( N^{-j}\right)^{-\beta}\left[ u\right]_{\beta}\int w\left(\left\vert y-x\right\vert\right)^{\beta}N^{jd}\left\vert \check{ \phi}\left( N^j\left(x-y\right)\right)\right\vert dy\\
		&=&w\left( N^{-j}\right)^{-\beta}\left[ u\right]_{\beta}\int w\left(N^{-j}\left\vert y\right\vert\right)^{\beta}\left\vert\check{ \phi}\left( y\right)\right\vert dy\\
		&\leq& \left[ u\right]_{\beta}\int l\left(\left\vert y\right\vert\right)^{\beta}\left\vert\check{ \phi}\left( y\right)\right\vert dy\leq C\left\vert u\right\vert_{\beta}.
	\end{eqnarray*}
	That is to say $\left\vert u\right\vert_{\beta,\infty}\leq C\left\vert u\right\vert_{\beta}$ for some constant $C\left( \beta,d\right)>0$.
	
	For the other direction, by Lemma \ref{equiv}, $\left\vert u\right\vert_0\leq C\left\vert u\right\vert_{\beta,\infty}$. Meanwhile, we can write
	\begin{equation*}
	\left[ u\right]_{\beta}=\sup_{x,y}\frac{\left\vert u\left( x+y\right)-u\left( x\right)\right\vert}{w\left(\left\vert y\right\vert\right)^{\beta}}\leq\sup_{t>0}\frac{\sup_{\left\vert y\right\vert\leq t}\left\vert u\left( x+y\right)-u\left( x\right)\right\vert_0}{w\left(t\right)^{\beta}}:=\sup_{t>0}\frac{\varpi\left( t,u\right)}{w\left(t\right)^{\beta}},
	\end{equation*}
	where $\varpi\left( t,u\right):=\sup_{\left\vert y\right\vert\leq t}\left\vert u\left( x+y\right)-u\left( x\right)\right\vert_0$ is increasing in $t$. Then, for $k\geq 0$,
	\begin{equation*}
	\varpi\left( N^{-k-1},u\right)\leq \varpi\left( t,u\right)\leq \varpi\left( N^{-k},u\right) \mbox{ if } N^{-k-1}\leq t< N^{-k},
	\end{equation*}
	and then by monotonicity of $l$, for $N^{-k-1}\leq t< N^{-k}$, 
	\begin{eqnarray*}
		l\left( N \right)^{-\beta}\frac{\varpi\left(  N^{-k-1},u\right)}{w\left( N^{-k-1}\right)^{\beta}} \leq \frac{\varpi\left( t,u\right)}{w\left(t\right)^{\beta}}\leq l\left( N \right)^{\beta}\frac{\varpi\left(  N^{-k},u\right)}{w\left( N^{-k}\right)^{\beta}}  .
	\end{eqnarray*}
	Hence,
	\begin{eqnarray*}
		\left[ u\right]_{\beta} &=& \sup_{t\geq 1}\frac{\sup_{\left\vert y\right\vert\leq t}\left\vert u\left( x+y\right)-u\left( x\right)\right\vert_0}{w\left(t\right)^{\beta}}\vee \sup_{0<t<1}\frac{\varpi\left( t,u\right)}{w\left(t\right)^{\beta}}\\
		&\leq& C\left[\left\vert u\right\vert_0\vee \sup_{k\geq 0}w\left( N^{-k}\right)^{-\beta} \varpi\left( N^{-k},u\right)\right]\\
		&\leq& C\left[\left\vert u\right\vert^{\beta}_{\infty,\infty}\vee \sup_{k\geq 0}w\left( N^{-k}\right)^{-\beta} \varpi\left( N^{-k},u\right)\right].
	\end{eqnarray*}
	
	It suffices to show that $w\left( N^{-k}\right)^{-\beta} \varpi\left( N^{-k},u\right)\leq C\left\vert u\right\vert_{\beta,\infty}$ for any $k\in\mathbf{N}$. Use the convolution functions introduced in section 2.2. Note 
	\begin{eqnarray*}
		&&\left\vert u\ast \varphi _{0}\left( x+y\right) -u\ast \varphi _{0}\left(
		x\right) \right\vert  \\
		&\leq &\int \left\vert \tilde{\varphi}_{0}\left( x+y-z\right) -\tilde{\varphi%
		}_{0}\left( x-z\right) \right\vert \left\vert u\ast \varphi _{0}\left(
		z\right) \right\vert dz \\
		&\leq &C\left( \left\vert y\right\vert \wedge 1\right) \left\vert u\ast
		\varphi _{0}\right\vert _{0},
	\end{eqnarray*}%
	and 
	\begin{eqnarray*}
		&&\left\vert u\ast \varphi _{j}\left( x+y\right) -u\ast \varphi _{j}\left(
		x\right) \right\vert  \\
		&\leq &N^{jd}\int \left\vert \tilde{\varphi}\left( N^{j}\left( x+y-z\right)
		\right) -\tilde{\varphi}\left( N^{j}\left( x-z\right) \right) \right\vert
		\left\vert u\ast \varphi _{j}\left( z\right) \right\vert dz \\
		&\leq &C\left( \left\vert N^{j}y\right\vert \wedge 1\right) \left\vert u\ast
		\varphi _{j}\right\vert _{0},\quad j\geq 1.
	\end{eqnarray*}%
	Therefore by Lemma \ref{equiv}, for each $k\in \mathbf{N}$, 
	\begin{eqnarray*}
		\varpi\left( N^{-k},u\right)  &=&\sup_{\left\vert y\right\vert \leq
			N^{-k}}\left\vert u\left( x+y\right) -u\left( x\right) \right\vert _{0} \\
		&\leq &\sup_{\left\vert y\right\vert \leq N^{-k}}\sum_{j=0}^{\infty
		}\left\vert u\ast \varphi _{j}\left( x+y\right) -u\ast \varphi _{j}\left(
		x\right) \right\vert _{0} \\
		&\leq &C\sup_{\left\vert y\right\vert \leq N^{-k}}\sum_{j=0}^{\infty }\left(
		N^{j}\left\vert y\right\vert \wedge 1\right) \left\vert u\ast \varphi
		_{j}\right\vert _{0},
	\end{eqnarray*}%
	and therefore, 
	\begin{eqnarray*}
		\varpi\left( N^{-k},u\right) &\leq& C\left\vert u\right\vert _{\beta ,\infty }\sup_{\left\vert
			y\right\vert \leq N^{-k}}\sum_{j=0}^{\infty }\left( N^{j}\left\vert
		y\right\vert\wedge 1\right) w \left( N^{-j}\right) ^{\beta }
		\\
		&\leq &C\left\vert u\right\vert _{\beta ,\infty }\left[
		\sum_{j=0}^{k}N^{j-k}w\left( N^{-j}\right) ^{\beta }+\sum_{j=k+1}^{\infty
		}w\left( N^{-j}\right) ^{\beta }\right] .
	\end{eqnarray*}
	
	Clearly, for all $j\in\mathbf{N}$, $j\leq x\leq j+1$,
	\begin{eqnarray*}
		l\left( 1\right)^{-\beta} w\left( N^{-x}\right)^{\beta}&\leq& w\left( N^{-j}\right)^{\beta}\leq l\left( N\right)^{\beta}w\left( N^{-x}\right)^{\beta},\\
		\frac{l\left( 1\right)^{-\beta}}{N} N^x w\left( N^{-x}\right)^{\beta} &\leq& N^j w\left( N^{-j}\right)^{\beta}\leq l\left( N\right)^{\beta}N^x w\left( N^{-x}\right)^{\beta}.
	\end{eqnarray*}
	Then for all $k\in\mathbf{N}$,	
	\begin{eqnarray*}
		C_{1}\int_{0}^{k+1}N^{x}w \left( N^{-x}\right) dx\leq
		\sum_{j=0}^{k}N^{j}w \left( N^{-j}\right) ^{\beta }\leq
		C_{2}\int_{0}^{k+1}N^{x}w \left( N^{-x}\right) ^{\beta }dx
	\end{eqnarray*}%
	for some positive constants $C_1,C_2$ that do not depend on $k,j$. Hence,
	\begin{eqnarray*}
		&&\sum_{j=0}^{k}N^{j-k}w\left( N^{-j}\right) ^{\beta}= N^{-k}\sum_{j=0}^{k}N^{j}w \left( N^{-j}\right) ^{\beta }\\		&\leq& C N^{-k}\int_{0}^{k+1}N^{x}w \left( N^{-x}\right) ^{\beta }dx=CN^{-k}\int_{1}^{N^{k+1}}w \left( t^{-1}\right) ^{\beta}dt\\
		&\leq&CN^{-k}w \left( N^{-k}\right) ^{\beta }\int_{1}^{N^{k+1}}l\left( \frac{N^{k}}{t}\right) ^{\beta }dt\\
		&=& Cw \left( N^{-k}\right) ^{\beta}\int_{N^{-1}}^{N^{k}}l\left( r\right) ^{\beta }r^{-2}dr.
	\end{eqnarray*}%
	Meanwhile,
	\begin{eqnarray*}
		\sum_{j=k+1}^{\infty }w\left( N^{-j}\right) ^{\beta } &\leq
		&C\int_{k+1}^{\infty }w \left( N^{-x}\right) ^{\beta }dx\leq Cw
		\left( N^{-k}\right) ^{\beta }\int_{k+1}^{\infty }l\left( N^{k}N^{-x}\right)
		^{\beta }dx \\
		&=&Cw \left( N^{-k}\right) ^{\beta }\int_{0}^{N^{-1}}l\left( r\right)
		^{\beta }\frac{dr}{r}.
	\end{eqnarray*}%
	
	Therefore, under the assumption $\eqref{asl}$, 
	\begin{eqnarray*}
		&&w\left( N^{-k}\right)^{-\beta} \varpi\left( N^{-k},u\right)\\
		&\leq &Cw\left( N^{-k}\right)^{-\beta}\left\vert u\right\vert _{\beta ,\infty }\left[
		\sum_{j=0}^{k}N^{j-k}w\left( N^{-j}\right) ^{\beta }+\sum_{j=k+1}^{\infty
		}w\left( N^{-j}\right) ^{\beta }\right]\\
		&\leq &C\left\vert u\right\vert _{\beta ,\infty }\left[
		\int_{N^{-1}}^{\infty}l\left( r\right) ^{\beta }r^{-2}dr+\int_{0}^{N^{-1}}l\left( r\right)
		^{\beta }\frac{dr}{r}\right]\leq  C\left\vert u\right\vert_{\beta,\infty}.
	\end{eqnarray*}
	That ends the proof.
\end{proof}

\noindent\textbf{Remark: }When $\mu\left( dy\right)=\frac{dy}{\left\vert y\right\vert^{d+\alpha}}$, one of L\'{e}vy measures that are of the most research interest, or when in case \cite{zh}, $w\left( t\right)=l\left( t\right)=t^{\alpha}$, $\eqref{asl}$ reduces to $\beta<1/\alpha$, which corresponds to the classical equivalence of the H\"{o}lder-Zygmund norm and the Besov norm.

The next lemma is fundamental to this paper.
\begin{lemma}\label{Lop}
	Let $\nu$ be a L\'{e}vy measure satisfying (iii) in \textbf{A(w,l)}. For any function $\varphi\in  C^{\infty}_b\left(\mathbf{R}^d\right)$,
	\begin{eqnarray*}
	L^{\tilde{\nu}_R}\varphi\left( x\right):=\int\left[ \varphi\left( x+y\right)-\varphi\left( x\right)-\chi_{\alpha}\left( y\right)y\cdot \nabla \varphi\left( x\right)\right]\tilde{\nu}_R\left(d y\right), R>0.
	\end{eqnarray*}
Then,
	\begin{equation}\label{fubi}
	\int\left\vert \varphi\left( x+y\right)-\varphi\left(  x\right)-\chi_{\alpha}\left( y\right)y\cdot \nabla \varphi\left( x\right)\right\vert\tilde{\nu}_R\left(dy\right)<C\left( \alpha,d,\varphi,\alpha_1,\alpha_2\right).
	\end{equation}
	Moreover, $L^{\tilde{\nu}_R}\varphi\in C^{\infty}_b\left(\mathbf{R}^d\right)$ and $D^{\gamma}L^{\tilde{\nu}_R}\varphi=L^{\tilde{\nu}_R} D^{\gamma}\varphi$, where $\gamma\in\mathbf{N}^d$ is a multi-index. If $\varphi\left(x\right)\in\mathcal{S}\left(\mathbf{R}^d\right)$, then $L^{\tilde{\nu}_R}\varphi\left( x\right)\in L^1\left(\mathbf{R}^d\right)$ and $\left\vert L^{\tilde{\nu}_R}\varphi\right\vert_{L^1\left(\mathbf{R}^d\right)}\leq C$ for some positive $C$ that is uniform with respect to $R$.
\end{lemma}
\begin{proof}
	Obviously,
	\begin{eqnarray}\label{r1}
		&&\int\left\vert \varphi\left( x+y\right)-\varphi\left(  x\right)-\chi_{\alpha}\left( y\right)y\cdot \nabla \varphi\left( x\right)\right\vert\tilde{\nu}_R\left(dy\right)\\
	&\leq& 1_{\alpha\in\left(0,1\right)}\int_{\left\vert y\right\vert\leq 1}\int_0^1\left\vert \nabla\varphi\left(x+\theta y\right)\right\vert \left\vert y\right\vert d\theta\tilde{\nu}_R\left(dy\right)\nonumber\\
	&&+ 1_{\alpha\in\left[1,2\right)}\int_{\left\vert y\right\vert\leq 1}\int_0^1\int_0^1\left\vert \nabla^2\varphi\left(x+\theta_1\theta_2 y\right)\right\vert \left\vert y\right\vert^2 d\theta_1 d\theta_2\tilde{\nu}_R\left(dy\right)\nonumber\\
	&&+ \int_{\left\vert y\right\vert> 1}\left(\left\vert \varphi\left(x+y\right)\right\vert+\left\vert\varphi\left( x\right)\right\vert+\chi_{\alpha}\left( y\right)\left\vert y\right\vert\left\vert \nabla \varphi\left( x\right)\right\vert\right) \tilde{\nu}_R\left(dy\right)\nonumber.
	\end{eqnarray}
	
	If $\varphi\left(x\right)\in C^{\infty}_b\left(\mathbf{R}^d\right)$, by (iii) in \textbf{A(w,l)},
	\begin{eqnarray}
	&&\int\left\vert \varphi\left( x+y\right)-\varphi\left(  x\right)-\chi_{\alpha}\left( y\right)y\cdot \nabla \varphi\left( x\right)\right\vert\tilde{\nu}_R\left(dy\right)\nonumber\\
	&\leq& 1_{\alpha\in\left(0,1\right)}\int_{\left\vert y\right\vert\leq 1}\left\vert \nabla\varphi\right\vert_0\left\vert y\right\vert \tilde{\nu}_R\left(dy\right)+1_{\alpha\in\left[1,2\right)}\int_{\left\vert y\right\vert\leq 1}\left\vert \nabla^2\varphi\right\vert_0 \left\vert y\right\vert^2 \tilde{\nu}_R\left(dy\right)\nonumber\\
	&& +\int_{\left\vert y\right\vert> 1}\left(2\left\vert \varphi\right\vert_0+\chi_{\alpha}\left( y\right)\left\vert y\right\vert\left\vert \nabla \varphi\right\vert_0\right) \tilde{\nu}_R\left(dy\right)\nonumber\\
\quad	&<&C\left( \int_{\left\vert y\right\vert\leq 1}\left\vert y\right\vert^{\alpha_1}\tilde{\nu}_R\left(dy\right)+\int_{\left\vert y\right\vert> 1}\left\vert y\right\vert^{\alpha_2}\tilde{\nu}_R\left(dy\right)\right)<C.\label{r2}
	\end{eqnarray}
	
	Since $\partial_i \varphi\in C_b^{\infty}\left(\mathbf{R}^d\right), i=1,2,\ldots,d$, the same steps can be applied to $\partial_i \varphi$. Then $\eqref{r2}$ indicates that $L^{\tilde{\nu}_R} \partial_i  \varphi\in C_b\left(\mathbf{R}^d\right)$ and $\partial_i L^{\tilde{\nu}_R}\varphi=L^{\tilde{\nu}_R} \partial_i\varphi$ by the dominated convergence theorem. Then, $L^{\tilde{\nu}_R}\varphi\in C^{\infty}_b\left(\mathbf{R}^d\right)$ and $D^{\gamma}L^{\tilde{\nu}_R}\varphi=L^{\tilde{\nu}_R} D^{\gamma}\varphi, \gamma\in \mathbf{N}^d$ is a consequence of induction.
	
	If $\varphi\left(x\right)\in \mathcal{S}\left(\mathbf{R}^d\right)$, then by $\eqref{r1}$,
	\begin{eqnarray*}
		\left\vert L^{\tilde{\nu}_R}\varphi\right\vert_{L^1\left(\mathbf{R}^d\right)}	&\leq&\int\int\left\vert \varphi\left( x+y\right)-\varphi\left(  x\right)-\chi_{\alpha}\left( y\right)y\cdot \nabla \varphi\left( x\right)\right\vert\tilde{\nu}_R\left(dy\right)dx\\
		&\leq& 1_{\alpha\in\left(0,1\right)}\int_{\left\vert y\right\vert\leq 1}\int_0^1\int\left\vert \nabla\varphi\left(x\right)\right\vert dx\left\vert y\right\vert d\theta\tilde{\nu}_R\left(dy\right)\\
		&&+ 1_{\alpha\in\left[1,2\right)}\int_{\left\vert y\right\vert\leq 1}\int_0^1\int_0^1\int\left\vert \nabla^2\varphi\left(x\right)\right\vert dx \left\vert y\right\vert^2 d\theta_1 d\theta_2\tilde{\nu}_R\left(dy\right)\\
		&&+\int_{\left\vert y\right\vert> 1}\int\left(2\left\vert\varphi\left( x\right)\right\vert+\chi_{\alpha}\left( y\right)\left\vert y\right\vert\left\vert \nabla \varphi\left( x\right)\right\vert\right) dx\tilde{\nu}_R\left(dy\right),
	\end{eqnarray*}
	again by (iii) in \textbf{A(w,l)},
	\begin{eqnarray*}
		\left\vert L^{\tilde{\nu}_R}\varphi\right\vert_{L^1\left(\mathbf{R}^d\right)}&\leq& C\left( \int_{\left\vert y\right\vert\leq 1}\left\vert y\right\vert^{\alpha_1}\tilde{\nu}_R\left(dy\right)+\int_{\left\vert y\right\vert> 1}\left\vert y\right\vert^{\alpha_2}\tilde{\nu}_R\left(dy\right)\right)\\
		&\leq&C\left( \alpha,d,\varphi,\alpha_1,\alpha_2\right).
	\end{eqnarray*}
\end{proof}

The lemma below is about integrability of $L^{\nu,\kappa}\varphi, \kappa\in\left( 0,1\right)$ and its probabilistic representation which we shall use repeatedly. 

\begin{lemma}\label{rep}
Let $\nu$ be a L\'{e}vy measure satisfying (iii) in \textbf{A(w,l)} and $\kappa\in\left( 0,1\right)$. $L^{\tilde{\nu}_R,\kappa},R>0$ is the associated operator defined as $\eqref{opp}$. Then for any $\varphi\left(x\right)\in C^{\infty}_b\left(\mathbf{R}^d\right)$, 
	\begin{eqnarray}
	L^{\tilde{\nu}_R,\kappa}\varphi\left( x\right)=C\int_0^{\infty}t^{-1-\kappa}\mathbf{E}\left[\varphi\left( x+Z_t^{\overline{\tilde{\nu}_R}}\right)-\varphi\left( x\right)\right]dt,R>0,\label{kap}
	\end{eqnarray}
	where $C^{-1}=\int_0^{\infty}t^{-\kappa-1}\left(1- e^{-t}\right)dt$ and 
	\begin{eqnarray*}
		\overline{\tilde{\nu}_R}\left( dy\right)=\frac{1}{2}\left( \tilde{\nu}_R\left( dy\right)+\tilde{\nu}_R\left(- dy\right)\right),R>0.
	\end{eqnarray*}
Besides, $L^{\tilde{\nu}_R,\kappa}\varphi\in C^{\infty}_b\left(\mathbf{R}^d\right)$. And $\left\vert L^{\tilde{\nu}_R,\kappa}\varphi\right\vert_{L^1\left( \mathbf{R}^d\right)}<C'$
for some $C'>0$ independent of $R$ if $\varphi\left(x\right)\in \mathcal{S}\left(\mathbf{R}^d\right)$.
\end{lemma}
\begin{proof}
Clearly, for all $R>0,\xi\in\mathbf{R}^d$, $\Re\psi^{\tilde{\nu}_R}\left( \xi\right)\leq 0$. Then for any $\kappa\in\left( 0,1\right)$, 
	\begin{eqnarray*}
\int_0^{\infty}t^{-\kappa-1}\left( 1-\exp\{\Re\psi^{\tilde{\nu}_R}\left( \xi\right)t\}\right)dt=\left( -\Re\psi^{\tilde{\nu}_R}\left( \xi\right)\right)^{\kappa}\int_0^{\infty}t^{-\kappa-1}\left(1- e^{-t}\right)dt.
	\end{eqnarray*}
Thus,
	\begin{eqnarray*}
		L^{\tilde{\nu}_R,\kappa}\varphi\left( x\right)
		&=& C\mathcal{F}^{-1}\left[ \int_0^{\infty}t^{-\kappa-1}\left(\exp\{\Re\psi^{\tilde{\nu}_R}\left( \xi\right)t\}-1\right)\mathcal{F}\varphi dt \right]\left( x\right),
	\end{eqnarray*}
where $C^{-1}=\int_0^{\infty}t^{-\kappa-1}\left(1- e^{-t}\right)dt$. Since $\Re\psi^{\tilde{\nu}_R}\left( \xi\right)=\psi^{\overline{\tilde{\nu}_R}}\left( \xi\right)$, by the L\'{e}vy-Khintchine formula,
\begin{eqnarray*}
	L^{\tilde{\nu}_R,\kappa}\varphi\left( x\right)
		&=& C\mathcal{F}^{-1}\left[ \int_0^{\infty}t^{-\kappa-1}\left(\exp\{\psi^{\overline{\tilde{\nu}_R}}\left( \xi\right)t\}-1\right)\mathcal{F}\varphi dt \right]\left( x\right)\\
		&=& C\mathcal{F}^{-1}\left[ \int_0^{\infty}t^{-\kappa-1}\mathbf{E}\mathcal{F}\left(\varphi\left( x+Z^{\overline{\tilde{\nu}_R}}_t\right)-\varphi\left(x\right)\right) dt \right]\left( x\right)\\
&=&C\mathcal{F}^{-1}\left[ \int_0^{\infty}t^{-\kappa-1}\mathcal{F}\mathbf{E}\left(\varphi\left( x+Z^{\overline{\tilde{\nu}_R}}_t\right)-\varphi\left(x\right)\right) dt \right]\left( x\right)
\end{eqnarray*}
if $\varphi\left(x\right)\in \mathcal{S}\left(\mathbf{R}^d\right)$. Note 
	\begin{eqnarray*}
		&& \int_0^{\infty}t^{-\kappa-1}\left\vert\mathbf{E}\left[\varphi\left( x+Z^{\overline{\tilde{\nu}_R}}_t\right)-\varphi\left(x\right)\right]\right\vert dt \\
		&\leq& 	\int_0^{1}t^{-\kappa-1}\int_0^t\left\vert L^{\overline{\tilde{\nu}_R}}\varphi\left(x+Z^{\overline{\tilde{\nu}_R}}_{r-}\right)\right\vert dr dt+\int_1^{\infty}t^{-\kappa-1}\mathbf{E}\left\vert\varphi\left( x+Z^{\overline{\tilde{\nu}_R}}_t\right)-\varphi\left(x\right)\right\vert dt,
	\end{eqnarray*}
	and note $\overline{\tilde{\nu}_R}=\widetilde{\bar{\nu}}_R$. If $\nu$ satisfies (iii) in \textbf{A(w,l)}, so does $\bar{\nu}$. Then by Lemma \ref{Lop},
		\begin{eqnarray}
		 &&\int_0^{\infty}t^{-\kappa-1}\int\left\vert\mathbf{E}\left[\varphi\left( x+Z^{\overline{\tilde{\nu}_R}}_t\right)-\varphi\left(x\right)\right]\right\vert dxdt \nonumber\\
		 &\leq& \int_0^{1}t^{-\kappa-1}\int_0^t\int\left\vert L^{\overline{\tilde{\nu}_R}}\varphi\left(x\right)\right\vert dx dr dt+2\int_1^{\infty}t^{-\kappa-1}\int\left\vert\varphi\left(x\right)\right\vert dx dt\nonumber\\
		 &\leq& C'\label{int}
	\end{eqnarray}
	 for some $C'>0$ independent of $R$. Thus Fubini's theorem applies, and thus
	\begin{eqnarray*}
		L^{\tilde{\nu}_R,\kappa}\varphi\left( x\right)
		=C \int_0^{\infty}t^{-\kappa-1}\mathbf{E}\left[\varphi\left( x+Z^{\overline{\tilde{\nu}_R}}_t\right)-\varphi\left(x\right)\right]dt.
	\end{eqnarray*}
	
	$L^1$ integrability of $L^{\tilde{\nu}_R,\kappa}\varphi$ has been shown in $\eqref{int}$.
	
	For $\varphi\in C^{\infty}_b\left(\mathbf{R}^d\right)$, we introduce $\{\zeta_n:n\in\mathbf{N}\}\subseteq C_0^{\infty}\left( \mathbf{R}^d\right)$ such that $0\leq \zeta_n\left( x\right)\leq 1,\forall n\in\mathbf{N},\forall x\in\mathbf{R}^d$ and $\zeta_n\left( x\right)= 1,\forall x\in\{x\in\mathbf{R}^d: \left\vert x\right\vert\leq n\}$. Then $\varphi\zeta_n\xrightarrow{n\to\infty} \varphi$ pointwise, which by the dominated convergence theorem implies that $\varphi\zeta_n\xrightarrow{n\to\infty} \varphi$ in the weak topology on $\mathcal{S}'\left(\mathbf{R}^d\right)$. Clearly, $\eqref{kap}$ holds for $\varphi\zeta_n$. Hence, 
	\begin{eqnarray*}
		&&<L^{\tilde{\nu}_R,\kappa}\varphi\zeta_n,\eta>\\
		&=&C\int \int_0^{\infty}t^{-\kappa-1}\mathbf{E}\left[\varphi\zeta_n\left( x+Z^{\overline{\tilde{\nu}_R}}_t\right)-\varphi\zeta_n\left(x\right)\right]dt\eta\left( x\right)dx\\
		&=&C\int \int_0^{\infty}t^{-\kappa-1}\mathbf{E}\left[\eta\left( x-Z^{\overline{\tilde{\nu}_R}}_t\right)-\eta\left( x\right)\right]dt\varphi\zeta_n\left(x\right)dx,\forall \eta\in \mathcal{S}\left( \mathbf{R}^d\right).
	\end{eqnarray*}
	Let $n\rightarrow \infty$.
	\begin{eqnarray*}
		&&\lim_{n\rightarrow \infty}<L^{\tilde{\nu}_R,\kappa}\varphi\zeta_n,\eta>\\
		&=&C\int \int_0^{\infty}t^{-\kappa-1}\mathbf{E}\left[\eta\left( x-Z^{\overline{\tilde{\nu}_R}}_t\right)-\eta\left( x\right)\right]dt\varphi\left(x\right)dx\\
		&=&C\int \int_0^{\infty}t^{-\kappa-1}\mathbf{E}\left[\varphi\left( x+Z^{\overline{\tilde{\nu}_R}}_t\right)-\varphi\left(x\right)\right]dt\eta\left( x\right)dx,\forall\eta\in \mathcal{S}\left( \mathbf{R}^d\right).
	\end{eqnarray*}
	This is to say $L^{\tilde{\nu}_R,\kappa}\varphi\zeta_n\xrightarrow{n\to\infty} \int_0^{\infty}t^{-\kappa-1}\mathbf{E}\left[\varphi\left( x+Z^{\overline{\tilde{\nu}_R}}_t\right)-\varphi\left(x\right)\right]dt$ in the topology of $\mathcal{S}'\left(\mathbf{R}^d\right)$. By continuity of the Fourier transform, 
	\begin{eqnarray*}
	&&C\mathcal{F}\int_0^{\infty}t^{-\kappa-1}\mathbf{E}\left[\varphi\left( x+Z^{\overline{\tilde{\nu}_R}}_t\right)-\varphi\left(x\right)\right]dt\\
	&=& \lim_{n\to\infty}\mathcal{F}\left[ L^{\tilde{\nu}_R,\kappa}\varphi\zeta_n\right]=-\left(-\Re\psi^{\tilde{\nu}_R}\right)^{\kappa}\lim_{n\to\infty}\mathcal{F}\left[ \varphi\zeta_n\right]\\
	&=&-\left(-\Re\psi^{\tilde{\nu}_R}\right)^{\kappa}\mathcal{F}\varphi.
	\end{eqnarray*}
	Therefore, $\eqref{opp}$ is well-defined for all functions in $C^{\infty}_b\left(\mathbf{R}^d\right)$ and $\eqref{kap}$ applies. That $L^{\tilde{\nu}_R,\kappa}\varphi\in C^{\infty}_b\left(\mathbf{R}^d\right)$ is a result of dominated convergence theorem and induction.
\end{proof}

\noindent\textbf{Remark:}  Lemma \ref{rep} claims $L^{\tilde{\nu}_R,\kappa},\forall R>0$ is a closed operation in $C^{\infty}_b\left(\mathbf{R}^d\right)$. Because of that, we can set $L^{\nu,\kappa}=L^{\nu,\kappa/2}\circ L^{\nu,\kappa/2}$ if $\kappa\in\left( 1,2\right)$, where $\circ$ means composition of two operations. Clearly, $\eqref{opp}$ is well-defined for all $\kappa\in\left( 1,2\right)$.

\begin{corollary}\label{co1}
Let $\nu$ be a L\'{e}vy measure satisfying \textbf{A(w,l)} and $\kappa\in\left(0,1\right]$. Denote $g_j=\mathcal{F}^{-1}\left[\mathcal{F} g \left(N^{-j}\cdot\right)\right],\forall g\left(x\right)\in \mathcal{S}\left(\mathbf{R}^d\right)$. Then there exists a constant $C>0$ independent of $j$ such that
\begin{eqnarray*}
\left\vert L^{\nu,\kappa}g_j\right\vert_{L^1\left( \mathbf{R}^d\right)}< C w\left( N^{-j}\right)^{-\kappa}.
\end{eqnarray*}	
\end{corollary}
\begin{proof}
If $j=0$, this is a straightforward consequence of Lemmas \ref{Lop} and \ref{rep}. Now consider $j\neq 0$. By $\eqref{alpha1}$ in \textbf{A(w,l)}, 
\begin{equation}\label{symbol}
\psi^{\nu}\left( \xi\right) = w\left( N^{-j}\right)^{-1}\psi^{\tilde{\nu}_{N^{-j}}}\left( N^{-j}\xi\right), \xi\in\mathbf{R}^d, \forall j\in\mathbf{N}_{+},
\end{equation}
therefore, by Lemma \ref{Lop},
\begin{eqnarray*}
	\left\vert L^{\nu}g_j\right\vert_{L^1\left( \mathbf{R}^d\right)} &=& \int \left\vert \mathcal{F}^{-1}\left[\psi^{\nu}\left( \xi\right)\mathcal{F}g\left( N^{-j}\xi\right)\right]\left( x\right)\right\vert dx\\
	&=& \int \left\vert \mathcal{F}^{-1}\left[w\left( N^{-j}\right)^{-1}\psi^{\tilde{\nu}_{N^{-j}}}\left( N^{-j}\xi\right)\mathcal{F}g\left( N^{-j}\xi\right)\right]\left( x\right)\right\vert dx\\
	&=& w\left( N^{-j}\right)^{-1}\int \left\vert \mathcal{F}^{-1}\left[\psi^{\tilde{\nu}_{N^{-j}}}\left( \xi\right)\mathcal{F}g\left( \xi\right)\right]\left( x\right)\right\vert dx\\
	&=& w\left( N^{-j}\right)^{-1}\left\vert L^{\tilde{\nu}_{N^{-j}}}g\right\vert_{L^1\left( \mathbf{R}^d\right)}\\
	&<&C\left( \alpha,d,\alpha_1,\alpha_2\right)w\left( N^{-j}\right)^{-1},
\end{eqnarray*}	
and by Lemma \ref{rep},
\begin{eqnarray*}
	\left\vert L^{\nu,\kappa}g_j\right\vert_{L^1\left( \mathbf{R}^d\right)}&=& \int \left\vert \mathcal{F}^{-1}\left[-\left(-\Re\psi^{\nu}\left( \xi\right)\right)^{\kappa}\mathcal{F}g\left( N^{-j}\xi\right)\right]\left( x\right)\right\vert dx\\
	&=& \int \left\vert \mathcal{F}^{-1}\left[-\left(-w\left( N^{-j}\right)^{-1}\Re\psi^{\tilde{\nu}_{N^{-j}}}\left( N^{-j}\xi\right)\right)^{\kappa}\mathcal{F}g\left( N^{-j}\xi\right)\right]\left( x\right)\right\vert dx\\
	&=& w\left( N^{-j}\right)^{-\kappa}\int \left\vert \mathcal{F}^{-1}\left[-\left(- \Re\psi^{\tilde{\nu}_{N^{-j}}}\left( \xi\right)\right)^{\kappa}\mathcal{F}g\left( \xi\right)\right]\left( x\right)\right\vert dx\\
	&=& w\left( N^{-j}\right)^{-\kappa}\left\vert L^{\tilde{\nu}_{N^{-j}},\kappa}g\right\vert_{L^1\left( \mathbf{R}^d\right)}\\
	&<&C\left( \alpha,d,\alpha_1,\alpha_2\right)w\left( N^{-j}\right)^{-\kappa}.
\end{eqnarray*}	
\end{proof}

The following two lemmas are crucial for proofs of norm equivalence.
\begin{lemma}\label{bij}
	Let $a\in\left( 0,\infty\right)$ and $\nu$ be a L\'{e}vy measure satisfying (iii) in \textbf{A(w,l)}. Then the operator $aI-L^{\nu}$ defines a bijection on $C_b^{\infty}\left( \mathbf{R}^d\right)$. Moreover, for any function $\varphi\in C_b^{\infty}\left( \mathbf{R}^d\right)$,
	\begin{eqnarray}
		&&\varphi\left( x\right) =\int_{0}^{\infty }e^{-a t}\mathbf{E}\left( aI
		-L^{\nu}\right) \varphi\left( x+Z_{t}^{\nu}\right) dt,\label{e1} \\
		&&\left( aI
		-L^{\nu}\right)^{-1}\varphi\left( x\right) =\int_{0}^{\infty }e^{-a t}\mathbf{E} \varphi\left( x+Z_{t}^{\nu}\right) dt,\quad x\in \mathbf{R}^{d},\label{e2}
	\end{eqnarray}
where $Z_{t}^{\nu}$ is the Levy process associated to $\nu$.
\end{lemma}
\begin{proof}
By Lemma \ref{Lop}, $aI
-L^{\nu}$ maps from $C_b^{\infty}\left( \mathbf{R}^d\right)$ to $C_b^{\infty}\left( \mathbf{R}^d\right)$. Apply the It\^{o} formula to $e^{-at}\varphi\left( x+Z^{\nu}_t\right)$ on $\left[0,S\right]$ with respect to $t$, and take expectation afterwards, then
\begin{eqnarray*}
	&& e^{-aS}\mathbf{E}\varphi\left( x+Z^{\nu}_S\right)-\varphi\left(x\right)\nonumber\\
	&=& \int_0^S -ae^{-at}\mathbf{E}\varphi\left( x+Z^{\nu}_t\right)dt+\int_0^S e^{- at}\mathbf{E}L^{\nu}\varphi\left( x+Z^{\nu}_t\right)dt.
\end{eqnarray*}
Note both $\varphi$ and $L^{\nu}\varphi$ are bounded. Let $S\rightarrow \infty$ and we obtain $\eqref{e1}$, which by Fubini theorem can also be written as
\begin{eqnarray}
\varphi\left( x\right) =\left( aI
-L^{\nu}\right)\int_{0}^{\infty }e^{-a t}\mathbf{E} \varphi\left( x+Z_{t}^{\nu}\right) dt,
\end{eqnarray}
namely, $aI-L^{\nu}$ is a surjection. Meanwhile, if $\varphi$ is a function in $C_b^{\infty}\left( \mathbf{R}^d\right)$ such that $\left(aI-L^{\nu}\right)\varphi=0$, then applying the same procedure, we arrive at $\eqref{e1}$, which claims $\varphi=0$ and thus $aI-L^{\nu}$ is bijective. It follows immediately that
\begin{eqnarray}
\left( aI
-L^{\nu}\right)^{-1}\varphi\left( x\right) =\int_{0}^{\infty }e^{-a t}\mathbf{E} \varphi\left( x+Z_{t}^{\nu}\right) dt,\quad x\in \mathbf{R}^{d}.\nonumber
\end{eqnarray}
\end{proof}

Similar results for $\left( I-L^{\nu}\right)^{\kappa},\kappa\in\left( 0,1\right)$ are stated in next lemma. Denote
\begin{eqnarray*}
\mathcal{A}\left( \mathbf{R}^d\right)=\{\varphi\in\mathcal{S}'\left( \mathbf{R}^d\right): \left( a-\Re \psi^{\nu}\right)^{\kappa}\mathcal{F}\varphi, \left( a-\Re \psi^{\nu}\right)^{-\kappa}\mathcal{F}\varphi\in\mathcal{S}'\left( \mathbf{R}^d\right) \}.
\end{eqnarray*}
Define for all $ \varphi\in\mathcal{A}\left( \mathbf{R}^d\right)$,
\begin{eqnarray}
\left( aI-L^{\nu}\right)^{\kappa}\varphi&=&\mathcal{F}^{-1}\left[ \left( a-\Re \psi^{\nu}\right)^{\kappa}\mathcal{F}\varphi\right],\label{oopp}\\
\left( aI-L^{\nu}\right)^{-\kappa}\varphi&=&\mathcal{F}^{-1}\left[ \left( a-\Re \psi^{\nu}\right)^{-\kappa}\mathcal{F}\varphi\right].\label{oppp}
\end{eqnarray}
Obviously, $\eqref{oopp}$ and $\eqref{oppp}$ offer a bijection on $\mathcal{A}\left( \mathbf{R}^d\right)$.
\begin{lemma}\label{rep2}
	Let $\kappa\in\left( 0,1\right)$, $a\in\left( 0,\infty\right)$ and $\nu$ be a L\'{e}vy measure satisfying (iii) in \textbf{A(w,l)}.
	Then, $C_b^{\infty}\left( \mathbf{R}^d\right)\subset\mathcal{A}\left( \mathbf{R}^d\right)$ and thus $ \left( aI-L^{\nu}\right)^{\kappa}$ is a bijection on it. Moreover, for any function $\varphi\in  C_b^{\infty}\left( \mathbf{R}^d\right)$, 
\begin{eqnarray}
	\qquad\left( aI-L^{\nu}\right)^{\kappa}\varphi\left(x\right)&=&C\int_0^{\infty}t^{-\kappa-1}\left[ \varphi\left(x\right)-e^{-at}\mathbf{E}\varphi\left( x+Z_t^{\bar{\nu}}\right)\right]dt,\quad\label{rp1}\\
	\qquad\quad\left( aI-L^{\nu}\right)^{-\kappa}\varphi\left(x\right)&=&C'\int_0^{\infty}t^{\kappa-1}e^{-at}\mathbf{E}\varphi\left( x+Z_t^{\bar{\nu}}\right)dt,\label{rp2}
\end{eqnarray}	
where $C^{-1}=\int_0^{\infty}t^{-\kappa-1}\left(1- e^{-t}\right)dt$, $C'^{-1}=\int_0^{\infty}t^{\kappa-1}e^{-t}dt$, $Z_t^{\bar{\nu}}$ is the L\'{e}vy process associated to $\bar{\nu}$ and $\bar{\nu}\left( dy\right):=\frac{1}{2}\left( \nu\left( dy\right)+\nu\left(- dy\right)\right)$.
\end{lemma}
\begin{proof}
	Since $a-\Re\psi^{\nu}\left( \xi\right)> 0,\forall \xi\in\mathbf{R}^d$ and
	\begin{eqnarray*}
		\int_0^{\infty}t^{-\kappa-1}\left( 1-\exp\{\Re\psi^{\nu}\left( \xi\right)t-at\}\right)dt=\left( a-\Re\psi^{\nu}\left( \xi\right)\right)^{\kappa}\int_0^{\infty}t^{-\kappa-1}\left(1- e^{-t}\right)dt,
	\end{eqnarray*}
then for all $\varphi\in \mathcal{S}\left( \mathbf{R}^d\right)$,
\begin{eqnarray}
\left( a-\Re\psi^{\nu}\left( \xi\right)\right)^{\kappa}\mathcal{F}\varphi&=&C\int_0^{\infty}t^{-\kappa-1}\left( 1-\exp\{\Re\psi^{\nu}\left( \xi\right)t-at\}\right)\mathcal{F}\varphi dt,\nonumber\\
&=& C\int_0^{\infty}t^{-\kappa-1}\mathcal{F}\mathbf{E}\left[ \varphi\left( x\right)-e^{-at}\varphi\left(x+Z_t^{\bar{\nu}}\right)\right]dt,\qquad\quad\label{fub}
\end{eqnarray}
where $C^{-1}=\int_0^{\infty}t^{-\kappa-1}\left(1- e^{-t}\right)dt$. Note for $t\in\left( 0,1\right)$, we have
\begin{eqnarray}
&&\left\vert\mathbf{E}\left[ \varphi\left( x\right)-e^{-at}\varphi\left(x+Z_t^{\bar{\nu}}\right)\right]\right\vert\nonumber\\
&\leq& \left\vert 1-e^{-at}\right\vert\left\vert \varphi\left( x\right)\right\vert+\mathbf{E}\int_0^t\left\vert L^{\bar{\nu}} \varphi\left(x+Z_{r-}^{\bar{\nu}}\right)\right\vert dr.\label{dec}
\end{eqnarray}
By Lemma \ref{Lop}, Fubini's theorem applies to $\eqref{fub}$, which implies
\begin{eqnarray*}
	\int_0^{\infty}t^{-\kappa-1}\mathbf{E}\left[ \varphi\left( x\right)-e^{-at}\varphi\left(x+Z_t^{\bar{\nu}}\right)\right]dt\in  C_b^{\infty}\left( \mathbf{R}^d\right),
\end{eqnarray*}
and
\begin{eqnarray*}
\left( a-\Re\psi^{\nu}\left( \xi\right)\right)^{\kappa}\mathcal{F}\varphi=C\mathcal{F}\int_0^{\infty}t^{-\kappa-1}\mathbf{E}\left[ \varphi\left( x\right)-e^{-at}\varphi\left(x+Z_t^{\bar{\nu}}\right)\right]dt\in  \mathcal{S}'\left( \mathbf{R}^d\right).
\end{eqnarray*}
Thus $\eqref{oopp}$ is well-defined. As a result,
\begin{eqnarray}
	\quad&&\left( aI-L^{\nu}\right)^{\kappa}\varphi\left( x\right)= C\int_0^{\infty}t^{-\kappa-1}\mathbf{E}\left[ \varphi\left( x\right)-e^{-at}\varphi\left(x+Z_t^{\bar{\nu}}\right)\right]dt.\label{map}
\end{eqnarray}

Similarly,
\begin{eqnarray*}
	\int_0^{\infty}t^{\kappa-1}\exp\{\Re\psi^{\nu}\left( \xi\right)t-at\}dt=\left( a-\Re\psi^{\nu}\left( \xi\right)\right)^{-\kappa}\int_0^{\infty}t^{\kappa-1}e^{-t}dt,
\end{eqnarray*}
then for all $\varphi\in \mathcal{S}\left( \mathbf{R}^d\right)$, 
\begin{eqnarray*}
	\left( a-\Re\psi^{\nu}\left( \xi\right)\right)^{-\kappa}\mathcal{F}\varphi=C'\mathcal{F}\int_0^{\infty}t^{\kappa-1}e^{-at}\mathbf{E}\varphi\left(x+Z_t^{\bar{\nu}}\right)dt\in  \mathcal{S}'\left( \mathbf{R}^d\right),
\end{eqnarray*}
where $C'^{-1}=\int_0^{\infty}t^{\kappa-1}e^{-t}dt$. And
\begin{eqnarray}
\left( aI-L^{\nu}\right)^{-\kappa}\varphi\left( x\right)= C'\int_0^{\infty}t^{\kappa-1}e^{-at}\mathbf{E}\varphi\left(x+Z_t^{\bar{\nu}}\right)dt.\label{inverse}
\end{eqnarray}

To extend $\eqref{map},\eqref{inverse}$ to all $\varphi\in C_b^{\infty}\left( \mathbf{R}^d\right)$, we repeat what we did in Lemma \ref{rep} and introduce $\{\zeta_n:n\in\mathbf{N}\}\subseteq C_0^{\infty}\left( \mathbf{R}^d\right)$ such that $0\leq \zeta_n\left( x\right)\leq 1,\forall n\in\mathbf{N},\forall x\in\mathbf{R}^d$ and $\zeta_n\left( x\right)= 1,\forall x\in\{x\in\mathbf{R}^d: \left\vert x\right\vert\leq n\}$. Then, $\varphi\zeta_n\xrightarrow{n\to\infty} \varphi$ and
\begin{eqnarray*}
\left( aI-L^{\nu}\right)^{-\kappa}\varphi\zeta_n&\xrightarrow{n\to\infty}&C' \int_0^{\infty}t^{\kappa-1}e^{-at}\mathbf{E}\varphi\left(x+Z_t^{\bar{\nu}}\right)dt,\\
\left( aI-L^{\nu}\right)^{\kappa}\varphi\zeta_n&\xrightarrow{n\to\infty}&C\int_0^{\infty}t^{-\kappa-1}\left[ \varphi\left(x\right)-e^{-at}\mathbf{E}\varphi\left( x+Z_t^{\bar{\nu}}\right)\right]dt
\end{eqnarray*}
all in the topology of $\mathcal{S}'\left( \mathbf{R}^d\right)$. Applying continuity of the Fourier transform, we know that $C_b^{\infty}\left( \mathbf{R}^d\right)\subset\mathcal{A}\left( \mathbf{R}^d\right)$ and $\eqref{rp1},\eqref{rp2}$ hold on it.
\end{proof}

\noindent\textbf{Remark:} Lemmas \ref{bij} and \ref{rep2} have shown that $\left( aI-L^{\nu}\right)^{\kappa},\kappa\in\left( 0,1\right]$ is a closed operation in $C_b^{\infty}\left( \mathbf{R}^d\right)$. Naturally, we may define $\left( aI-L^{\nu}\right)^{\kappa}$ for all $\kappa\in\left(1,2\right)$ and all $\varphi\in C_b^{\infty}\left( \mathbf{R}^d\right)$ as follows:
\begin{eqnarray*}
\left( aI-L^{\nu}\right)^{\kappa}\varphi&=&\left( aI-L^{\nu}\right)^{\kappa/2}\circ\left( aI-L^{\nu}\right)^{\kappa/2}\varphi, \\
\left( aI-L^{\nu}\right)^{-\kappa}\varphi&=&\left( aI-L^{\nu}\right)^{-\kappa/2}\circ\left( aI-L^{\nu}\right)^{-\kappa/2}\varphi.
\end{eqnarray*}
Here $\circ$ represents composition of two operations. This definition is compatible with $\eqref{oopp},\eqref{oppp}$ when $\kappa\in\left(1,2\right)$. The corollary below says that the probabilistic representation of $\left( aI-L^{\nu}\right)^{-\kappa}$ for $\kappa\in\left( 0,1\right)$ also applies to $ \kappa\in\left(1,2\right)$.
\begin{corollary}\label{exbij}
	Let $\kappa\in\left( 0,2\right)$, $a\in\left( 0,\infty\right)$ and $\nu$ be a L\'{e}vy measure satisfying (iii) in \textbf{A(w,l)}. Then $\left( aI-L^{\nu}\right)^{\kappa}$ is a bijection on $ C_b^{\infty}\left( \mathbf{R}^d\right)$. For any function $\varphi\in  C_b^{\infty}\left( \mathbf{R}^d\right)$, 
	\begin{eqnarray}
	\left( aI-L^{\nu}\right)^{-\kappa}\varphi\left(x\right)=C\int_0^{\infty}t^{\kappa-1}e^{-at}\mathbf{E}\varphi\left( x+Z_t\right)dt,\label{rp3}
	\end{eqnarray}	
	where $C$ is a constant only depending on $\kappa$, and $Z_t=Z_t^{\nu}$ if $\kappa=1$, $Z_t=Z_t^{\bar{\nu}}$ otherwise.
\end{corollary}
\begin{proof}
	That $\left( aI-L^{\nu}\right)^{\kappa},\kappa\in\left( 0,2\right)$ is a bijection follows from the definition. Suppose $\kappa\in\left( 1,2\right)$ and $\varphi\in  C_b^{\infty}\left( \mathbf{R}^d\right)$. Use $\eqref{rp2}$.
	\begin{eqnarray*}
	&&\left(aI-L^{\nu}\right)^{-\kappa}\varphi\left(x\right)\\
	&=&C\int_0^{\infty}t^{\kappa/2-1}e^{-at}\mathbf{E}\left(aI-L^{\nu}\right)^{-\kappa/2}\varphi\left( x+Z_t^{\bar{\nu}}\right)dt\\
	&=& C\int_0^{\infty}t^{\kappa/2-1}e^{-at}\mathbf{E}\int_0^{\infty}s^{\kappa/2-1}e^{-as}\mathbf{E}\varphi\left( x+Z_t^{\bar{\nu}}+Z_s^{\bar{\nu}}\right)dsdt.
	\end{eqnarray*}
$Z_t^{\bar{\nu}},Z_s^{\bar{\nu}}$ denote two independent and identically distributed L\'{e}vy processes associated to $\bar{\nu}$. Therefore,
	\begin{eqnarray*}
	&&\left(aI-L^{\nu}\right)^{-\kappa}\varphi\left(x\right)\\
	&=& C\int_0^{\infty}t^{\kappa/2-1}e^{-at}\int_0^{\infty}s^{\kappa/2-1}e^{-as}\mathbf{E}\varphi\left( x+Z_{t+s}^{\bar{\nu}}\right)dsdt\\
	&=& C\int_0^{\infty}t^{\kappa/2-1}\int_0^{\infty}s^{\kappa/2-1}e^{-at-as}\int\varphi\left( x+z\right)p\left(t+s, z\right)dzdsdt,
\end{eqnarray*}
where $p\left(t, z\right)$ is the probability density of $Z_{t}^{\bar{\nu}}$. Then by changing variables and applying Fubini theorem, we obtain
\begin{eqnarray*}
	&&\left(aI-L^{\nu}\right)^{-\kappa}\varphi\left(x\right)\\
	&=& C\int_0^{\infty}t^{\kappa-1}\int_0^{\infty}r^{\kappa/2-1}e^{-at-atr}\int\varphi\left( x+z\right)p\left(t+tr, z\right)dzdrdt\\
	&=& C\int_0^{\infty}r^{\kappa/2-1}\int\int_0^{\infty}t^{\kappa-1}e^{-at-atr}\varphi\left( x+z\right)p\left(t+tr, z\right)dtdzdr\\
   &=&C\int_0^{\infty}\frac{r^{\kappa/2-1}}{\left(1+r\right)^{\kappa}}dr\int_0^{\infty}\int t^{\kappa-1}e^{-at}\varphi\left( x+z\right)p\left(t, z\right)dzdt\\	
	&=& C\int_0^{\infty}t^{\kappa-1}e^{-at}\mathbf{E}\varphi\left( x+Z_{t}^{\bar{\nu}}\right)dt.
\end{eqnarray*}
\end{proof}

\begin{proposition}\label{pr3}
	Let $\nu$ be a L\'{e}vy measure satisfying \textbf{A(w,l)}, $\beta\in\left( 0,\infty\right),\kappa\in\left( 0,1\right]$. Then norm $\left\vert u\right\vert _{\nu,\kappa,\beta }$ and norm $\left\Vert u\right\Vert _{\nu,\kappa,\beta }$ are equivalent in $C_b^{\infty}\left( \mathbf{R}^d\right)$.	
\end{proposition}
\begin{proof}
	For the purpose of clarity, we state our proof in parts. \\
	\noindent\textbf{Part 1:} Show $\left\vert u\right\vert _{\nu,\kappa,\beta }\leq C\left\Vert u\right\Vert _{\nu,\kappa,\beta }$ for all $\kappa\in\left( 0,1\right]$.
	
	 By $\eqref{e2}$, $\eqref{rp2}$ and $\eqref{sup}$, for all $\kappa\in\left(0,1\right]$,
	\begin{eqnarray*}
	\left\vert \left( I
	-L^{\nu}\right)^{-\kappa}u\right\vert_0\leq C\left\vert u\right\vert_0\leq C \left\vert u\right\vert_{\beta,\infty}, \forall u\in C_b^{\infty}\left( \mathbf{R}^d\right).
	\end{eqnarray*}
	Since $\left( I-L^{\nu}\right)^{\kappa}$ is a bijection, it implies
	\begin{eqnarray}
	\left\vert u\right\vert_0\leq C \left\vert \left( I
	-L^{\nu}\right)^{\kappa}u\right\vert_{\beta,\infty}, \forall u\in C_b^{\infty}\left( \mathbf{R}^d\right).\label{ssup}
	\end{eqnarray}
	
	In the mean time, by $\eqref{e2}$, for all $\kappa\in\left(0,1\right]$, $j\in\mathbf{N}, u\in C_b^{\infty}\left( \mathbf{R}^d\right)$,
	\begin{eqnarray}
	\left\vert\left[\left( I
	-L^{\nu}\right)^{-\kappa}u\right]\ast\varphi_j\right\vert_0 =\left\vert\int_{0}^{\infty }t^{\kappa-1}e^{- t}\mathbf{E}\left[ u\ast\varphi_j\left( x+Z_{t}\right)\right] dt\right\vert_0\leq C\left\vert u\ast\varphi_j\right\vert_0,\nonumber
	\end{eqnarray}
	where $Z_t=Z_{t}^{\nu}$ if $\kappa=1$ and $Z_t=Z_{t}^{\bar{\nu}}$ otherwise. Given that $\left( I-L^{\nu}\right)^{\kappa}$ is bijective, it leads to
	\begin{eqnarray*}
	\left\vert u\ast\varphi_j\right\vert_0\leq  C\left\vert \left[\left( I
	-L^{\nu}\right)^{\kappa}u\right]\ast\varphi_j\right\vert_0, \forall j\in\mathbf{N},
	\end{eqnarray*} 
	namely, for all $\kappa\in\left(0,1\right]$,
	\begin{eqnarray}
	\left\vert u\right\vert_{\beta,\infty}\leq C \left\vert \left( I
	-L^{\nu}\right)^{\kappa}u\right\vert_{\beta,\infty}, \forall u\in C_b^{\infty}\left( \mathbf{R}^d\right).\label{low}
	\end{eqnarray}
	
	Therefore,
	\begin{eqnarray}\label{kap1}
		\left\vert 
		L^{\nu}u\right\vert_{\beta,\infty}\leq \left\vert \left( I
		-L^{\nu}\right)u\right\vert_{\beta,\infty}+\left\vert u\right\vert_{\beta,\infty}\leq C \left\vert \left( I-L^{\nu}\right)u\right\vert_{\beta,\infty}.
	\end{eqnarray}
Similarly, for $\kappa\in\left( 0,1\right)$, by $\eqref{kap}$ and $\eqref{rp1}$,
\begin{eqnarray*}
	\left\vert L^{\nu,\kappa}u\ast \varphi_j\right\vert_0&\leq& \left\vert \left( I-L^{\nu}\right)^{\kappa}u\ast\varphi_j\right\vert_0+ C\left\vert\int_0^{\infty}t^{-\kappa-1}\left(1-e^{-t}\right) \mathbf{E}u\ast\varphi_j\left( x+Z_t^{\bar{\nu}}\right)dt\right\vert_0\\
	&\leq& \left\vert \left( I-L^{\nu}\right)^{\kappa}u\ast\varphi_j\right\vert_0+ C\left\vert u\ast\varphi_j\right\vert_0, \forall j\in\mathbf{N},
\end{eqnarray*}
which together with $\eqref{low}$ indicate
\begin{eqnarray}\label{kap2}
\qquad	\left\vert L^{\nu,\kappa}u\right\vert_{\beta,\infty}\leq \left\vert \left( I-L^{\nu}\right)^{\kappa}u\right\vert_{\beta,\infty}+ C\left\vert u\right\vert_{\beta,\infty}\leq C\left\vert \left( I-L^{\nu}\right)^{\kappa}u\right\vert_{\beta,\infty}.
\end{eqnarray}

Combine $\eqref{ssup},\eqref{kap1},\eqref{kap2}$. $\left\vert u\right\vert _{\nu,\kappa,\beta }\leq C\left\Vert u\right\Vert _{\nu,\kappa,\beta },\forall\kappa\in\left( 0,1\right]$.

\noindent\textbf{Part 2:} Show $\left\Vert u\right\Vert _{\nu,\kappa,\beta }\leq C\left\vert u\right\vert _{\nu,\kappa,\beta }$ for all $\kappa\in\left( 0,1\right)$. 

By $\eqref{kap}$ and $\eqref{rp1}$ again,
\begin{eqnarray*} 
\left\vert \left( I-
	L^{\nu}\right)^{\kappa}u\ast\varphi_j\right\vert_0&\leq&\left\vert L^{\nu,\kappa}u\ast \varphi_j\right\vert_0+ C\left\vert\int_0^{\infty}t^{-\kappa-1}\left(1-e^{-t}\right) \mathbf{E}u\ast\varphi_j\left( x+Z_t^{\bar{\nu}}\right)dt\right\vert_0\\
	&\leq&\left\vert L^{\nu,\kappa}u\ast \varphi_j\right\vert_0+ C\left\vert u\ast\varphi_j\right\vert_0,\forall j\in\mathbf{N}.
\end{eqnarray*}
 This is to say 
	\begin{eqnarray}\label{lower}
	\left\vert \left( I-
	L^{\nu}\right)^{\kappa}u\right\vert_{\beta,\infty}\leq \left\vert L^{\nu,\kappa}u\right\vert_{\beta,\infty}+ C\left\vert u\right\vert_{\beta,\infty}.
	\end{eqnarray}
	
	It then suffices to prove $\left\vert u\right\vert_{\beta,\infty}\leq C\left( \left\vert u\right\vert_{0}+\left\vert L^{\nu,\kappa}u\right\vert_{\beta,\infty}\right)$ for $\kappa\in\left(0,1\right)$. Note,
		\begin{eqnarray}\label{var0}
	\left\vert u\ast\varphi_0\right\vert_{0}\leq \left\vert\varphi_0\right\vert_{L^1\left(\mathbf{R}^d\right)}\left\vert u\right\vert_0\leq C\left\vert u\right\vert_0.
	\end{eqnarray}
For $j\neq 0$,
\begin{eqnarray*}
	\left\vert u\ast\varphi_j\right\vert_0
	&=&\left\vert \mathcal{F}^{-1}\left[\left( -\Re
	\psi^{\nu}\left( \xi\right)\right)^{-\kappa}\widehat{\tilde{\varphi}_j}\left( \xi\right)\hat{\varphi_j}\left( \xi\right)\left( -\Re
	\psi^{\nu}\left( \xi\right)\right)^{\kappa}\mathcal{F}u\right]\right\vert_0\\
	&=&\left\vert \left( \mathcal{F}^{-1}g_j\right)\ast \left(L^{\nu,\kappa}u\ast\varphi_j\right)\right\vert_0,
\end{eqnarray*}
where 
\begin{eqnarray*}
	g_j\left(\xi\right)=-\left( -\Re
	\psi^{\nu}\left( \xi\right)\right)^{-\kappa}\mathcal{F}\tilde{\varphi}\left(N^{-j} \xi\right).
\end{eqnarray*}

We would like to show that $\left\vert \mathcal{F}^{-1}g_j\right\vert_{L^1\left( \mathbf{R}^d\right)}<C$ for some $C$ independent of $j$. Because in that case,
\begin{eqnarray*}
	\left\vert u\ast\varphi_j\right\vert_0\leq \left\vert \mathcal{F}^{-1}g_j\right\vert_{L^1\left( \mathbf{R}^d\right)}\left\vert L^{\nu,\kappa}u\ast\varphi_j\right\vert_0\leq C\left\vert L^{\nu,\kappa}u\ast\varphi_j\right\vert_0, j\in\mathbf{N}_{+},
\end{eqnarray*}
which together with $\eqref{var0}$ lead to $\left\vert u\right\vert_{\beta,\infty}\leq C\left( \left\vert u\right\vert_{0}+\left\vert L^{\nu,\kappa}u\right\vert_{\beta,\infty}\right)$. Indeed, 
\begin{eqnarray*}
&&\int \left\vert \mathcal{F}^{-1}g_j\left( x\right)\right\vert dx\\
&\leq& C\left\vert \mathcal{F}^{-1}\int_0^{\infty}t^{\kappa-1}\exp\{\Re\psi^{\nu}\left( \xi\right)t\}\mathcal{F}\tilde{\varphi}\left(N^{-j} \xi\right)dt\right\vert_{L^1\left( \mathbf{R}^d\right)}\\
&=& C\left\vert \mathcal{F}^{-1}\int_0^{\infty}t^{\kappa-1}\exp\{w\left( N^{-j}\right)^{-1}\Re\psi^{\tilde{\nu}_{N^{-j}}}\left( N^{-j}\xi\right)t\}\mathcal{F}\tilde{\varphi}\left(N^{-j} \xi\right)dt\right\vert_{L^1\left( \mathbf{R}^d\right)}\\
&=& C\left\vert \mathcal{F}^{-1}\int_0^{\infty}t^{\kappa-1}\exp\{w\left( N^{-j}\right)^{-1}\Re\psi^{\tilde{\nu}_{N^{-j}}}\left( \xi\right)t\}\mathcal{F}\tilde{\varphi}\left( \xi\right)dt\right\vert_{L^1\left( \mathbf{R}^d\right)}.
\end{eqnarray*}
Note $\Re\psi^{\tilde{\nu}_{N^{-j}}}\left( \xi\right)=\psi^{\overline{\tilde{\nu}_{N^{-j}}}}\left( \xi\right)$, where $\overline{\tilde{\nu}_{N^{-j}}}\left( dy\right)=\widetilde{\bar{\nu}}_{N^{-j}}$. Then,
\begin{eqnarray*}
\int \left\vert \mathcal{F}^{-1}g_j\left( x\right)\right\vert dx\leq C\left\vert \mathcal{F}^{-1}\int_0^{\infty}t^{\kappa-1}\mathcal{F}\mathbf{E}\tilde{\varphi}\left( x+Z_{w\left( N^{-j}\right)^{-1}t}^{\widetilde{\bar{\nu}}_{N^{-j}}}\right) dt\right\vert_{L^1\left( \mathbf{R}^d\right)}.
\end{eqnarray*}
Recall that $supp \mathcal{F}\tilde{\varphi}=\{\xi: N^{-2}\leq \left\vert\xi\right\vert\leq N^2 \}$. By Lemma \ref{lemma2} in Appendix, there are positive constants $C_1,C_2$ independent of $j$ ($j\neq 0$), such that 
\begin{eqnarray*}
	\int_{\mathbf{R}^d}\left\vert\mathbf{E}\tilde{\varphi}\left(x+Z_t^{\widetilde{\bar{\nu}}_{N^{-j}}}\right)\right\vert dx \leq C_1 e^{-C_2 t}.
\end{eqnarray*}
Therefore,
\begin{eqnarray*}
	\int \left\vert \mathcal{F}^{-1}g_j\left( x\right)\right\vert dx &\leq& C\left\vert \int_0^{\infty}t^{\kappa-1}\mathbf{E}\tilde{\varphi}\left( x+Z_{w\left( N^{-j}\right)^{-1}t}^{\widetilde{\bar{\nu}}_{N^{-j}}}\right) dt\right\vert_{L^1\left( \mathbf{R}^d\right)}\\
	&\leq& C\int_0^{\infty}t^{\kappa-1}\exp\{-C_2w\left( N^{-j}\right)^{-1}t\}dt\\
	&\leq& Cw\left( N^{-j}\right)^{\kappa}\leq C
\end{eqnarray*}
for some $C$ independent of $j$.

\noindent\textbf{Part 3:} Show $\left\Vert u\right\Vert _{\nu,\kappa,\beta }\leq C\left\vert u\right\vert _{\nu,\kappa,\beta }$ for all $\kappa=1$. 

Since $\left\vert \left( I-
	L^{\nu}\right)u\right\vert_{\beta,\infty}\leq \left\vert 
	L^{\nu}u\right\vert_{\beta,\infty}+\left\vert u\right\vert_{\beta,\infty}$. Similarly as Part 2, we just need to show $\left\vert u\right\vert_{\beta,\infty}\leq C\left( \left\vert u\right\vert_{0}+\left\vert 
	L^{\nu}u\right\vert_{\beta,\infty}\right)$. Again,
	\begin{eqnarray*}
	\left\vert u\ast\varphi_0\right\vert_{0}\leq \left\vert\varphi_0\right\vert_{L^1\left(\mathbf{R}^d\right)}\left\vert u\right\vert_0\leq C\left\vert u\right\vert_0.
	\end{eqnarray*}
	And for $j\neq 0$,
	\begin{eqnarray*}
		\left\vert u\ast\varphi_j\right\vert_0
		&=&\left\vert \mathcal{F}^{-1}\left[\left( 
		\psi^{\nu}\left( \xi\right)\right)^{-1}\widehat{\tilde{\varphi}_j}\left( \xi\right)\hat{\varphi_j}\left( \xi\right)
		\psi^{\nu}\left( \xi\right)\mathcal{F}u\right]\right\vert_0\\
		&=&\left\vert \left( \mathcal{F}^{-1}g_j\right)\ast \left(L^{\nu}u\ast\varphi_j\right)\right\vert_0\\
		&\leq& \left\vert \mathcal{F}^{-1}g_j\right\vert_{L^1\left( \mathbf{R}^d\right)}\left\vert L^{\nu}u\ast\varphi_j\right\vert_0,
	\end{eqnarray*}
	where 
	\begin{eqnarray*}
		g_j\left(\xi\right)=\left(
		\psi^{\nu}\left( \xi\right)\right)^{-1}\mathcal{F}\tilde{\varphi}\left(N^{-j} \xi\right).
	\end{eqnarray*}

 By Lemma \ref{lemma1} in Appendix, for each $j\in\mathbf{N}_{+}$, $\Re\psi^{\nu}<0$, thus $g_j\left(\xi\right)$ is well-defined. The rest of this proof is devoted to looking for an upper bound of $\left\vert \mathcal{F}^{-1}g_j\right\vert_{L^1\left( \mathbf{R}^d\right)}$ that is uniform with respect to $j$. Analogously as before, applying Lemma \ref{lemma2} in Appendix,
\begin{eqnarray*}
\int \left\vert \mathcal{F}^{-1}g_j\left( x\right)\right\vert dx
&=& \left\vert \mathcal{F}^{-1}\left[\int_0^{\infty}\exp\{\psi^{\nu}\left( \xi\right)t\}  \mathcal{F}\tilde{\varphi}\left(N^{-j} \xi\right)dt\right]\right\vert_{L^1\left( \mathbf{R}^d\right)}\\
&\leq& \int_{\mathbf{R}^d} \int_0^{\infty}\left\vert\mathbf{E}\tilde{\varphi}\left( x+Z^{\tilde{\nu}_{N^{-j}}}_{w\left( N^{-j}\right)^{-1}t}\right)\right\vert dt dx\\
&\leq& Cw\left( N^{-j}\right)\leq C.
\end{eqnarray*}

This concludes the proof.
\end{proof}

\begin{proposition}\label{pr4}
Let $\nu$ be a L\'{e}vy measure satisfying \textbf{A(w,l)}, $\beta\in\left( 0,\infty\right),\kappa\in\left( 0,1\right]$. Then norm $\left\vert u\right\vert _{\nu,\kappa,\beta }$ and norm $\left\vert u\right\vert_{\kappa+\beta,\infty}$ are equivalent in $C_b^{\infty}\left( \mathbf{R}^d\right)$.	
\end{proposition}
\begin{proof}
We first assume the finiteness of $\left\vert u\right\vert_{\kappa+\beta,\infty}$. It was showed in Lemma \ref{equiv} that $\left\vert u\right\vert_0\leq C\left\vert u\right\vert_{\kappa+\beta,\infty}$. To prove $\left\vert L^{\nu,\kappa}u\right\vert_{\beta,\infty}\leq C\left\vert u\right\vert_{\kappa+\beta,\infty}$ for some $C>0$, it suffices to show for each $j\in \mathbf{N}$,  
\begin{eqnarray*}
\left\vert \left(L^{\nu,\kappa}u\right)\ast \varphi_j\right\vert_0\leq C w\left( N^{-j}\right)^{-\kappa}\left\vert u\ast \varphi_j\right\vert_0, \quad \kappa\in\left( 0,1\right].
\end{eqnarray*}
In fact, by Corollary \ref{co1},
\begin{eqnarray*}
	&&\left\vert \left(L^{\nu,\kappa}u\right)\ast \varphi_j\right\vert_0=\left\vert L^{\nu,\kappa}\left(u\ast \varphi_j\ast \tilde{\varphi}_j\right)\right\vert_0=\left\vert \left(L^{\nu,\kappa}\tilde{\varphi}_j\right)\ast\left(u\ast \varphi_j\right) \right\vert_0\\
	&\leq& \left\vert L^{\nu,\kappa}\tilde{\varphi}_j\right\vert_{L^1\left( \mathbf{R}^d\right)}\left\vert u\ast \varphi_j\right\vert_0\leq C w\left( N^{-j}\right)^{-\kappa}\left\vert u\ast \varphi_j\right\vert_0.
\end{eqnarray*}
This is to say for all $\kappa\in\left( 0,1\right]$,
\begin{eqnarray}\label{lmu}
\left\vert u\right\vert _{\nu,\kappa,\beta }=\left\vert u\right\vert_0+\left\vert L^{\nu,\kappa}u\right\vert_{\beta,\infty}<C\left\vert u\right\vert_{\kappa+\beta,\infty}.
\end{eqnarray}

To show $\left\vert u\right\vert_{\kappa+\beta,\infty}<C\left\vert u\right\vert _{\nu,\kappa,\beta }$, according to Proposition \ref{pr3}, we just need to prove $\left\vert u\right\vert_{\kappa+\beta,\infty}<C\left\Vert u\right\Vert _{\nu,\kappa,\beta }$. By $\eqref{e2}$,
\begin{eqnarray*}
&&\left\vert\left( I
-L^{\nu}\right)^{-1}\left(u\ast\varphi_j\right)\right\vert_0\\
&=&\left\vert\int_{0}^{\infty }e^{- t}\mathcal{F}^{-1}\left[\exp\{\psi^{\nu}\left( \xi\right)t\}\widehat{\tilde{\varphi}}_j\left( \xi\right)\hat{\varphi_j}\left( \xi\right)\hat{ u}\left( \xi\right)\right] dt\right\vert_0,\forall j\in\mathbf{N}.
\end{eqnarray*}
By $\eqref{rp2}$, for all $\kappa\in\left( 0,1\right)$,
\begin{eqnarray*}
	&&\left\vert\left( I
	-L^{\nu}\right)^{-\kappa}\left(u\ast\varphi_j\right)\right\vert_0\\
	&=&C\left\vert\int_{0}^{\infty }t^{\kappa-1}e^{- t}\mathcal{F}^{-1}\left[\exp\{\psi^{\bar{\nu}}\left( \xi\right)t\}\widehat{\tilde{\varphi}}_j\left( \xi\right)\hat{\varphi_j}\left( \xi\right)\hat{ u}\left( \xi\right)\right] dt\right\vert_0,\forall j\in\mathbf{N}.
\end{eqnarray*}

First we consider $j=0$. Set $Z_t=Z_{t}^{\nu}$ if $\kappa=1$ and $Z_t=Z_{t}^{\bar{\nu}}$ otherwise. For all $\kappa\in\left( 0,1\right]$,
\begin{eqnarray}
	\left\vert\left( I
	-L^{\nu}\right)^{-\kappa}\left(u\ast\varphi_0\right)\right\vert_0 
	&\leq& \left\vert u\ast\varphi_0\right\vert_0\int_{0}^{\infty }t^{\kappa-1}e^{-t}\left\vert\mathbf{E}\tilde{\varphi}_0\left(\cdot+Z_t\right)\right\vert_{L^1\left( \mathbf{R}^d\right)}dt\nonumber\\
	&\leq& C\left\vert u\ast\varphi_0\right\vert_0.\label{q2}
\end{eqnarray}

For $j\neq 0$, use $\eqref{symbol}$.
\begin{eqnarray*}
&&\left\vert\left( I
-L^{\nu}\right)^{-1}\left(u\ast\varphi_j\right)\right\vert_0\\
&=&\left\vert\int_{0}^{\infty }e^{-t}\mathcal{F}^{-1}\left[\exp\{w\left( N^{-j}\right)^{-1}\psi^{\tilde{\nu}_{N^{-j}}}\left( N^{-j}\xi\right)t\}\mathcal{F}\tilde{\varphi}\left(N^{-j} \xi\right)\right]\ast \left(u\ast\varphi_j\right) dt\right\vert_0\\
&\leq& \left\vert u\ast\varphi_j\right\vert_0\int_{0}^{\infty }e^{-t}\left\vert\mathcal{F}^{-1}\left[\exp\{\psi^{\tilde{\nu}_{N^{-j}}}\left( \xi\right)w\left( N^{-j}\right)^{-1}t\}\mathcal{F}\tilde{\varphi}\left( \xi\right)\right]\right\vert_{L^1\left( \mathbf{R}^d\right)}dt\\
&\leq& \left\vert u\ast\varphi_j\right\vert_0\int_{0}^{\infty }e^{-t}\left\vert\mathbf{E}\tilde{\varphi}\left(\cdot+Z_{w\left( N^{-j}\right)^{-1}t}^{\tilde{\nu}_{N^{-j}}}\right)\right\vert_{L^1\left( \mathbf{R}^d\right)}dt,
\end{eqnarray*}
which, by Lemma \ref{lemma2} in Appendix, leads to
\begin{eqnarray}
&&\left\vert\left( I
	-L^{\nu}\right)^{-1}\left(u\ast\varphi_j\right)\right\vert_0\label{q3}\\
&\leq& C\left\vert u\ast\varphi_j\right\vert_0\int_{0}^{\infty }e^{-C_2w\left( N^{-j}\right)^{-1} t}dt\leq Cw\left( N^{-j}\right)\left\vert u\ast\varphi_j\right\vert_0.\nonumber
\end{eqnarray}
Similarly, for $\kappa\in\left( 0,1\right)$,
\begin{eqnarray}
	&&\left\vert\left( I
	-L^{\nu}\right)^{-\kappa}\left(u\ast\varphi_j\right)\right\vert_0\nonumber\\
	&\leq& C\left\vert u\ast\varphi_j\right\vert_0\int_{0}^{\infty }t^{\kappa-1}e^{-t}\left\vert\mathbf{E}\tilde{\varphi}\left(\cdot+Z_{w\left( N^{-j}\right)^{-1}t}^{\widetilde{\bar{\nu}}_{N^{-j}}}\right)\right\vert_{L^1\left( \mathbf{R}^d\right)}dt\nonumber\\
	&\leq& Cw\left( N^{-j}\right)^{\kappa}\left\vert u\ast\varphi_j\right\vert_0.\label{q4}
\end{eqnarray}

Combine $\eqref{q2}-\eqref{q4}$.
\begin{eqnarray*}
	\left\vert\left( I
	-L^{\nu}\right)^{-\kappa}u\right\vert_{\kappa+\beta,\infty}\leq C\left\vert u\right\vert_{\beta,\infty},\forall\kappa\in\left(0,1\right],\forall u\in C_b^{\infty}\left( \mathbf{R}^d\right).
\end{eqnarray*}
By Lemmas \ref{bij} and \ref{rep2}, that means
\begin{eqnarray*}
	\left\vert u\right\vert_{\kappa+\beta,\infty}\leq C\left\vert \left( I
	-L^{\nu}\right)^{\kappa}u\right\vert_{\beta,\infty},\forall u\in C_b^{\infty}\left( \mathbf{R}^d\right).
\end{eqnarray*}

Therefore, $\left\vert u\right\vert_{\kappa+\beta,\infty}<C\left\Vert u\right\Vert _{\nu,\kappa,\beta }<C\left\vert u\right\vert _{\nu,\kappa,\beta }$.
\end{proof}

\begin{corollary}\label{pr44}
	Let $\nu$ be a L\'{e}vy measure satisfying \textbf{A(w,l)}, $\beta\in\left( 0,\infty\right),\kappa\in\left( 0,1\right]$. Then norm $\left\Vert u\right\Vert _{\nu,\kappa,\beta }$ and norm $\left\vert u\right\vert_{\kappa+\beta,\infty}$ are equivalent in $C_b^{\infty}\left( \mathbf{R}^d\right)$.
\end{corollary}
\begin{proof}
	This is a consequence of Propositions \ref{pr3} and \ref{pr4}.
\end{proof}

\begin{corollary}\label{co2}
	Let $\nu$ be a L\'{e}vy measure satisfying \textbf{A(w,l)}, $\kappa\in\left(0,2\right)$ and $\beta\in\left(0,\infty\right)$. $u\in C_b^{\infty}\left(\mathbf{R}^d\right)\cap \tilde{C}^{\kappa+\beta}_{\infty,\infty}\left(\mathbf{R}^d\right)$. Then there exists a constant $C>0$ independent of $j$ such that
	\begin{eqnarray*}
		\left\vert L^{\nu,\kappa}u\right\vert_{0}\leq \left\vert L^{\nu,\kappa}u\right\vert_{\beta,\infty}\leq C \left\vert u\right\vert_{\beta+\kappa,\infty}.
	\end{eqnarray*}	
\end{corollary}
\begin{proof}
	By Proposition \ref{pr4}, if $\kappa\in\left(0,1\right]$,
	\begin{eqnarray*}
\left\vert L^{\nu,\kappa}u\right\vert_{0}\leq \left\vert L^{\nu,\kappa}u\right\vert_{\beta,\infty}\leq C\left\vert u\right\vert_{\kappa+\beta,\infty}.
	\end{eqnarray*}
Now suppose $\kappa\in\left(1,2\right)$. $L^{\nu,\kappa}u:=L^{\nu,\kappa/2}\circ L^{\nu,\kappa/2}u$. Then by Corollary \ref{co1},
 	\begin{eqnarray*}
 	&&\left\vert L^{\nu,\kappa}u\ast \varphi_j\right\vert_{0}=\left\vert u\ast \varphi_j\ast L^{\nu,\kappa/2}\tilde{\varphi}_j\ast L^{\nu,\kappa/2}\tilde{\varphi}_j\right\vert_{0}\\
 	&\leq& \left\vert u\ast \varphi_j\right\vert_{0}\left\vert L^{\nu,\kappa/2}\tilde{\varphi}_j\right\vert_{L^1\left(\mathbf{R}^d\right)}\left\vert L^{\nu,\kappa/2}\tilde{\varphi}_j\right\vert_{L^1\left(\mathbf{R}^d\right)}\\
 	&\leq& Cw\left( N^{-j}\right)^{-\kappa}\left\vert u\ast \varphi_j\right\vert_{0},\forall j\in\mathbf{N}.
 \end{eqnarray*}
 Therefore, $L^{\nu,\kappa} u\in\tilde{C}^{\beta}_{\infty,\infty}\left(\mathbf{R}^d\right)$ and $\left\vert L^{\nu,\kappa}u\right\vert_{0}\leq \left\vert L^{\nu,\kappa}u\right\vert_{\beta,\infty}\leq C\left\vert u\right\vert_{\kappa+\beta,\infty}$.
 \end{proof}

\begin{proposition}\label{pr5}
	Let $0<\beta'<\beta$. Then for any $ \varepsilon\in\left(0,1\right)$ and any bounded function $u$ in $\mathbf{R}^d$,	
	\begin{eqnarray*}
		\left\vert u\right\vert_{\beta',\infty}\leq\varepsilon\left\vert u\right\vert_{\beta,\infty}+C_{\varepsilon}\left\vert u\right\vert_{0},
	\end{eqnarray*}
	where $C_{\varepsilon}$ is independent of $u$.
\end{proposition}
\begin{proof}
	It is sufficient to show that for $\forall j\in\mathbf{N}$, 	
	\begin{eqnarray*}
		w\left( N^{-j}\right)^{-\beta'}\left\vert u\ast\varphi_j\right\vert_0\leq \varepsilon\left\vert u\right\vert_{\beta,\infty}+C_{\varepsilon}\left\vert u\right\vert_{0} .
	\end{eqnarray*}
	
	Apply Young's inequality. For any $\epsilon\in\left(0,1\right)$ and any pair of $p,q$ such that $\frac{1}{p}+\frac{1}{q}=1$,
	\begin{eqnarray*}
		\left\vert u\ast\varphi_j\right\vert_0 &=& \left( \epsilon w\left( N^{-j}\right)^{\frac{\beta'-\beta}{p}}\left\vert u\ast\varphi_j\right\vert_0^{1/p}\right)\left(\epsilon^{-1} w\left( N^{-j}\right)^{\frac{\beta-\beta'}{p}}\left\vert u\ast\varphi_j\right\vert_0^{1/q}\right)\\
		&\leq& \frac{\epsilon^p w\left( N^{-j}\right)^{\beta'-\beta}}{p}\left\vert u\ast\varphi_j\right\vert_0+\frac{ w\left( N^{-j}\right)^{\frac{\left(\beta-\beta'\right)q}{p}}}{q\epsilon^q}\left\vert u\ast\varphi_j\right\vert_0,\forall j\in\mathbf{N},
	\end{eqnarray*}
	thus,
	\begin{eqnarray*}
		&&w\left( N^{-j}\right)^{-\beta'}\left\vert u\ast\varphi_j\right\vert_0\\
		&\leq& \frac{\epsilon^p w\left( N^{-j}\right)^{-\beta}}{p}\left\vert u\ast\varphi_j\right\vert_0+\frac{ w\left( N^{-j}\right)^{\frac{\left(\beta-\beta'\right)q}{p}-\beta'}}{q\epsilon^q}\left\vert u\ast\varphi_j\right\vert_0\\
		&\leq& \frac{\epsilon^p }{p}\left\vert u\right\vert_{\beta,\infty}+\frac{1}{q\epsilon^q}w\left( N^{-j}\right)^{\frac{\left(\beta-\beta'\right)q}{p}-\beta'}\left\vert u\ast\varphi_j\right\vert_0,\forall j\in\mathbf{N}.
	\end{eqnarray*}
	
	Choose $p,q$ such that $\frac{\left(\beta-\beta'\right)q}{p}-\beta'\geq 0$, then for some $C>0$,
	\begin{eqnarray*}
		\frac{1}{q\epsilon^q}w\left( N^{-j}\right)^{\frac{\left(\beta-\beta'\right)q}{p}-\beta'}\left\vert u\ast\varphi_j\right\vert_0\leq \frac{C}{q\epsilon^q}\left\vert u\ast\varphi_j\right\vert_0\leq \frac{C}{q\epsilon^q}\left\vert u\right\vert_0,\forall j\in\mathbf{N}.
	\end{eqnarray*}
	Take $\epsilon$ such that $\frac{\epsilon^p }{p}=\varepsilon$ and this is the end the proof.
\end{proof}

\begin{proposition}\label{app}
Let $\beta\in\left(0,\infty\right)$, $u\in\tilde{C}^{\beta}_{\infty,\infty}\left(\mathbf{R}^d\right)$. Then there exists a sequence $ u_n\in C_b^{\infty}\left(\mathbf{R}^d\right)$ such that
\begin{eqnarray*}
\left\vert u\right\vert_{\beta,\infty}\leq \liminf_n\left\vert u_n\right\vert_{\beta,\infty},\quad\left\vert u_n\right\vert_{\beta,\infty}\leq C\left\vert u\right\vert_{\beta,\infty} 
\end{eqnarray*}
for some $C>0$ that only depends on $d,N$, and for any $0<\beta'<\beta$,
\begin{eqnarray*}
\left\vert u_n-u\right\vert_{\beta',\infty}\rightarrow 0 \mbox{ as }n\rightarrow \infty.
\end{eqnarray*}
\end{proposition}	
\begin{proof}
Set $u_n\left( x\right)=\sum_{j=0}^{n+2}\left( u\ast\varphi_j\right)\left( x\right),n\in\mathbf{N}$. Then 
\begin{eqnarray*}
\left\vert u_n\right\vert_{\beta,\infty}=\sup_j\left\vert \sum_{k=0}^{n+2}u\ast \varphi_k\ast\varphi_j\right\vert_0 w\left( N^{-j}\right)^{-\beta}.
\end{eqnarray*}
By construction of $\varphi_j,j\in\mathbf{N}$ in this note, if $j\geq 1,n\geq j-1$,
\begin{eqnarray*}
	\left\vert \sum_{k=0}^{n+2}u\ast \varphi_k\ast\varphi_j\right\vert_0=\left\vert\sum_{k=j-1}^{j+1}u\ast \varphi_k\ast\varphi_j\right\vert_0=\left\vert u\ast \varphi_j\right\vert_0.
\end{eqnarray*}
If $j\geq 2,n< j-1$,
\begin{eqnarray*}
	\left\vert \sum_{k=0}^{n+2}u\ast \varphi_k\ast\varphi_j\right\vert_0&\leq& \left\vert u\ast \varphi_j\ast\varphi_{j-1}\right\vert_0+\left\vert u\ast \varphi_{j}\ast\varphi_j\right\vert_0\\
	&=&2\left\vert \mathcal{F}^{-1}\phi\right\vert_{L^1\left(\mathbf{R}^d\right)}\left\vert u\ast \varphi_j\right\vert_0.
\end{eqnarray*}
Besides,
\begin{eqnarray*}
	\left\vert \sum_{k=0}^{n+2}u\ast \varphi_k\ast\varphi_0\right\vert_0&\leq& \left\vert u\ast \varphi_0\ast\varphi_{0}\right\vert_0+\left\vert u\ast \varphi_{0}\ast\varphi_1\right\vert_0\\
	&=&\left(\left\vert \varphi_0\right\vert_{L^1\left(\mathbf{R}^d\right)}+\left\vert \varphi_1\right\vert_{L^1\left(\mathbf{R}^d\right)}\right)\left\vert u\ast \varphi_j\right\vert_0.
\end{eqnarray*}
Therefore, for all $n\in\mathbf{N}$,
\begin{eqnarray*}
	\left\vert u_n\right\vert_{\beta,\infty}\leq C\sup_j\left\vert u\ast \varphi_j\right\vert_0 w\left( N^{-j}\right)^{-\beta}\leq C\left\vert u\right\vert_{\beta,\infty}.
\end{eqnarray*}

On the other hand, by Lemma \ref{equiv}, $u\left( x\right)=\sum_{k=0}^{\infty}\left( u\ast\varphi_k\right)\left( x\right)$. Then in the same vein as above,
\begin{eqnarray*}
	\left\vert u\ast\varphi_j\right\vert_{0}&=& \left\vert\sum_{k=0}^{n+2} u\ast\varphi_k\ast \varphi_j+\sum_{k=n+3}^{\infty} u\ast\varphi_k\ast \varphi_j\right\vert_0\\
	&=&\left\vert u_n\ast\varphi_j\right\vert_{0}, \quad\forall n\geq j-1,\forall j\in\mathbf{N},
\end{eqnarray*}
thus, 
\begin{eqnarray*}
	\left\vert u\ast\varphi_j\right\vert_{0}w\left( N^{-j}\right)^{-\beta}\leq\left\vert u_n\right\vert_{\beta,\infty}, \quad\forall n\geq j-1,\forall j\in\mathbf{N},
\end{eqnarray*}
and thus $\left\vert u\right\vert_{\beta,\infty}\leq \liminf_n\left\vert u_n\right\vert_{\beta,\infty}$.

At last, 
\begin{eqnarray*}
	\left\vert u-u_n\right\vert_{\beta',\infty}&=&\sup_j \left\vert\sum_{k=n+3}^{\infty} u\ast\varphi_k\ast \varphi_j\right\vert_0w\left( N^{-j}\right)^{-\beta'}\\
	&=&\sup_{j\geq n+2} \left\vert\sum_{k=n+3}^{\infty} u\ast\varphi_k\ast \varphi_j\right\vert_0 w\left( N^{-j}\right)^{-\beta}w\left( N^{-j}\right)^{\beta-\beta'}\\
	&\leq&C\sup_{j\geq n+2} \left\vert u\ast \varphi_j\right\vert_0 w\left( N^{-j}\right)^{-\beta}w\left( N^{n-j}\right)^{\beta-\beta'}l\left( N^{-n}\right)^{\beta-\beta'}\\
	&\leq& C\left\vert u\right\vert_{\beta,\infty}l\left( N^{-n}\right)^{\beta-\beta'}\rightarrow 0 \mbox{ as }n\rightarrow \infty.
\end{eqnarray*}
\end{proof}

Using the approximating sequence introduced in the lemma above, we can extend $L^{\nu,\kappa}u,\kappa\in\left( 0,2\right)$ to all $u\in\tilde{C}^{\kappa+\beta}_{\infty,\infty}\left(\mathbf{R}^d\right),\beta>0$ as follows:
\begin{eqnarray*}
L^{\nu,\kappa}u\left(x\right)=\lim_{n\rightarrow \infty}L^{\nu,\kappa}u_n\left(x\right),x\in\mathbf{R}^d.
\end{eqnarray*}
The next proposition justifies this definition and addresses continuity of the operator defined in this sense. 
\begin{proposition}\label{cont}
	Let $\nu$ be a L\'{e}vy measure satisfying \textbf{A(w,l)}, $\beta\in\left(0,\infty\right)$ and $\kappa\in\left(0,2\right)$. Then $\eqref{opp}$ is well-defined for all $\kappa$ and all $u\in\tilde{C}^{\kappa+\beta}_{\infty,\infty}\left(\mathbf{R}^d\right)$,
	\begin{eqnarray}\label{ext}
	L^{\nu,\kappa}u\left(x\right)=\lim_{n\rightarrow \infty}L^{\nu,\kappa}u_n\left(x\right),x\in\mathbf{R}^d,
	\end{eqnarray}
and this convergence is uniform with respect to $x$. Moreover, 
	\begin{eqnarray*}
	\left\vert L^{\nu,\kappa}u\right\vert_{0}\leq\left\vert L^{\nu,\kappa}u\right\vert_{\beta,\infty}&\leq& C\left\vert u\right\vert_{\kappa+\beta,\infty}
	\end{eqnarray*}
	for some $C>0$ independent of $u$.
\end{proposition}
\begin{proof}
	Since $u\in\tilde{C}^{\kappa+\beta}_{\infty,\infty}\left(\mathbf{R}^d\right)$, by Proposition \ref{app}, there is a a sequence $ u_n\in C_b^{\infty}\left(\mathbf{R}^d\right)$ such that
	\begin{eqnarray*}
		\left\vert u\right\vert_{\kappa+\beta,\infty}\leq \liminf_n\left\vert u_n\right\vert_{\kappa+\beta,\infty},\quad\left\vert u_n\right\vert_{\kappa+\beta,\infty}\leq C\left\vert u\right\vert_{\kappa+\beta,\infty} 
	\end{eqnarray*}
	for some $C>0$ independent of $u$, and for any $0<\beta'<\beta$,
	\begin{eqnarray*}
		\left\vert u_n-u\right\vert_{\kappa+\beta',\infty}\rightarrow 0 \mbox{ as }n\rightarrow \infty,
	\end{eqnarray*}
which, according to Lemma \ref{equiv} and $\eqref{sup}$, indicates $ u\in C\left(\mathbf{R}^d\right)$. Meanwhile, $\left\vert u_n-u\right\vert_0\to 0$ as $n\rightarrow \infty$ and thus $u_n\xrightarrow{n\to\infty}u$ in the weak topology of $ \mathcal{S}'\left(\mathbf{R}^d\right)$. For such a sequence, by Corollary \ref{co2},
   \begin{eqnarray*}
	\left\vert L^{\nu,\kappa}u_n\right\vert_{0}&\leq& \left\vert L^{\nu,\kappa}u_n\right\vert_{\beta,\infty}\leq C \left\vert u_n\right\vert_{\beta+\kappa,\infty},\\
	\left\vert L^{\nu,\kappa}u_n-L^{\nu,\kappa}u_m\right\vert_{0}&\leq& C \left\vert u_n-u_m\right\vert_{\beta'+\kappa,\infty}\xrightarrow{n,m\to\infty}0.
   \end{eqnarray*}	
Therefore, both $L^{\nu,\kappa}u_n,\forall n\in\mathbf{N}$ and $\lim_{n\to\infty}L^{\nu,\kappa}u_n$ are continuous functions, and therefore,
\begin{eqnarray*}
	\mathcal{F}\left[\lim_{n\rightarrow \infty}L^{\nu}u_n\right]=\lim_{n\rightarrow \infty}\psi^{\nu}\mathcal{F}u_n&=&\psi^{\nu}\mathcal{F}u\in\mathcal{S}'\left(\mathbf{R}^d\right),\\
	\mathcal{F}\left[\lim_{n\rightarrow \infty}L^{\nu,\kappa}u_n\right]=\lim_{n\rightarrow \infty}-\left(-\Re\psi^{\nu}\right)^{\kappa} \mathcal{F}u_n&=&-\left(-\Re\psi^{\nu}\right)^{\kappa} \mathcal{F}u\in\mathcal{S}'\left(\mathbf{R}^d\right),\kappa\neq 1.
\end{eqnarray*}
Namely,
\begin{eqnarray*}
L^{\nu,\kappa}u\left(x\right)=\lim_{n\rightarrow \infty}L^{\nu,\kappa}u_n\left(x\right),x\in\mathbf{R}^d.
\end{eqnarray*}
Clearly, this convergence is uniform over $x$.
Now given any $\beta\in\left(0,\infty\right)$, 
\begin{eqnarray*}
	&&w\left( N^{-j}\right)^{-\beta}\left\vert L^{\nu,\kappa}u\ast \varphi_j\right\vert_0= \lim_{n\rightarrow\infty}w\left( N^{-j}\right)^{-\beta}\left\vert L^{\nu,\kappa}u_n\ast \varphi_j\right\vert_0\\
	&\leq&\limsup_{n\rightarrow\infty}\left\vert u_n\right\vert_{\beta+\kappa,\infty}\leq C \left\vert u\right\vert_{\beta+\kappa,\infty},\forall j\in\mathbf{N}.
\end{eqnarray*}
Namely, $\left\vert L^{\nu,\kappa}u\right\vert_{\beta,\infty}\leq C\left\vert u\right\vert_{\beta+\kappa,\infty}$.
\end{proof}

\begin{theorem}\label{thm4}
	Let $\nu$ be a L\'{e}vy measure satisfying \textbf{A(w,l)}, $\beta\in\left( 0,\infty\right),\kappa\in\left( 0,1\right]$. Then norm $\left\vert u\right\vert _{\nu,\kappa,\beta }$ and norm $\left\vert u\right\vert_{\kappa+\beta,\infty}$ are equivalent.	
\end{theorem}
\begin{proof}
	As a consequence of $\eqref{sup}$ and Proposition \ref{cont}, there exists a positive constant $C$ independent of $u$ such that
	\begin{eqnarray}
	\left\vert u\right\vert_0+\left\vert L^{\nu,\kappa}u\right\vert_{\beta,\infty}\leq C\left\vert u\right\vert_{\kappa+\beta,\infty}.
	\end{eqnarray}
	
	Now suppose $\left\vert u\right\vert_0+\left\vert L^{\nu,\kappa}u\right\vert_{\beta,\infty}<\infty$. First, $u$ is a bounded function. Then $u_n=\sum_{i=0}^{n+2}u\ast \varphi_i\in C_b^{\infty}\left( \mathbf{R}^d\right),\forall n\in\mathbf{N}$. Meanwhile, recall that 
	\begin{eqnarray*}
	\left(L^{\nu,\kappa}u\right)_n:= \sum_{i=0}^{n+2}\left(L^{\nu,\kappa}u\right)\ast \varphi_i=L^{\nu,\kappa}u_n\in \tilde{C}^{\kappa+\beta}_{\infty,\infty}\left(\mathbf{R}^d\right)
	\end{eqnarray*}
approximates $L^{\nu,\kappa}u$ and $\left\vert L^{\nu,\kappa}u_n\right\vert_{\beta,\infty}\leq C\left\vert L^{\nu,\kappa}u\right\vert_{\beta,\infty}$. Therefore, by Proposition \ref{pr4},
\begin{eqnarray*}
\left\vert u_n\right\vert_{\kappa+\beta,\infty}\leq C\left(\left\vert u_n\right\vert_0+ \left\vert L^{\nu,\kappa}u_n\right\vert_{\beta,\infty}\right)\leq C\left\vert L^{\nu,\kappa}u_n\right\vert_{\beta,\infty}\leq C\left\vert L^{\nu,\kappa}u\right\vert_{\beta,\infty}.
\end{eqnarray*}
That is to say, for any $j\in\mathbf{N}$,
\begin{eqnarray*}
	w\left(N^{-j}\right)^{-\kappa-\beta}\left\vert u_n\ast\varphi_j\right\vert_{0}\leq C\left\vert L^{\nu,\kappa}u\right\vert_{\beta,\infty}.
\end{eqnarray*}

It suffices to observe that for $j\geq 2, n\geq j-1$, 
\begin{eqnarray*}
	\left\vert u_n\ast\varphi_j\right\vert_{0}=\left\vert u\ast\tilde{\varphi}_j\ast\varphi_j\right\vert_{0}=\left\vert u\ast\varphi_j\right\vert_{0}\leq Cw\left(N^{-j}\right)^{\kappa+\beta}\left\vert L^{\nu,\kappa}u\right\vert_{\beta,\infty},
\end{eqnarray*}
and $\left\vert u_n\ast\varphi_j\right\vert_{0}=\left\vert u\ast\varphi_j\right\vert_{0}\leq C\left\vert u\right\vert_{0}$, $j=0$ or $1$, $n\geq j-1$. Therefore, $\left\vert u\right\vert_{\kappa+\beta,\infty}\leq C\left(\left\vert u\right\vert_{0}+\left\vert L^{\nu,\kappa}u\right\vert_{\beta,\infty}\right)$.
\end{proof}

\begin{theorem}\label{thm5}
	Let $\nu$ be a L\'{e}vy measure satisfying \textbf{A(w,l)}, $\beta\in\left( 0,\infty\right),\kappa\in\left( 0,1\right]$. Then norm $\left\Vert u\right\Vert _{\nu,\kappa,\beta }$ and norm $\left\vert u\right\vert_{\kappa+\beta,\infty}$ are equivalent.	
\end{theorem}
\begin{proof}
$\kappa=1$ has been covered by Theorem \ref{thm4} and Proposition \ref{pr5}. Let us consider $\kappa\in\left( 0,1\right)$. 

First assume the finiteness of $\left\vert u\right\vert_{\kappa+\beta,\infty}$. Then by Lemma \ref{equiv}, $u$ is a bounded and continuous function. Set $u_n=\sum_{i=0}^{n+2}u\ast \varphi_i\in C_b^{\infty}\left( \mathbf{R}^d\right),n\in\mathbf{N}$. We have known that $\left\vert u_n-u\right\vert_{0}\leq C\left\vert u_n-u\right\vert_{\kappa+\beta',\infty}\xrightarrow{n\to\infty}0,\forall \beta'\in\left( 0,\beta\right)$. Hence, 
\begin{eqnarray*}
	\mathcal{F}\left[\lim_{n\rightarrow \infty}\left(I-L^{\nu}\right)^{\kappa}u_n\right]=\lim_{n\rightarrow \infty}\left(1-\Re\psi^{\mu}\right)^{\kappa} \mathcal{F}u_n=\left(1-\Re\psi^{\mu}\right)^{\kappa} \mathcal{F}u\in\mathcal{S}'\left(\mathbf{R}^d\right),
\end{eqnarray*}
namely, $\left(I-L^{\nu}\right)^{\kappa}u$ is well-defined and $\left(I-L^{\nu}\right)^{\kappa}u=\lim_{n\rightarrow \infty}\left(I-L^{\nu}\right)^{\kappa}u_n$. By Lemma \ref{equiv} and Corollary \ref{pr44},
\begin{eqnarray*}
	&&\lim_{n,m\rightarrow \infty}\left\vert\left(I-L^{\nu}\right)^{\kappa}u_n-\left(I-L^{\nu}\right)^{\kappa}u_m\right\vert_0\\
	&\leq& C\lim_{n,m\rightarrow \infty}\left\vert\left(I-L^{\nu}\right)^{\kappa}u_n-\left(I-L^{\nu}\right)^{\kappa}u_m\right\vert_{\beta'}\leq C\lim_{n,m\rightarrow \infty}\left\vert u_n-u_m\right\vert_{\kappa+\beta'}=0.
\end{eqnarray*}
Then the convergence is uniform on $\mathbf{R}^d$. Hence, for any $j\in\mathbf{N}$,
\begin{eqnarray*}
	&&w\left(N^{-j}\right)^{-\beta}\left\vert\left(I-L^{\nu}\right)^{\kappa}u\ast\varphi_j\right\vert_0=w\left(N^{-j}\right)^{-\beta}\lim_{n\rightarrow \infty}\left\vert\left(I-L^{\nu}\right)^{\kappa}u_n\ast\varphi_j\right\vert_{0}\\
	&\leq& \lim_{n\rightarrow \infty}\left\vert\left(I-L^{\nu}\right)^{\kappa}u_n\ast\varphi_j\right\vert_{\beta}\leq C\lim_{n\rightarrow \infty}\left\vert u_n\right\vert_{\kappa+\beta}\leq C\left\vert u\right\vert_{\kappa+\beta},
\end{eqnarray*}
i.e. $\left\Vert u\right\Vert _{\nu,\kappa,\beta }\leq C\left\vert u\right\vert_{\kappa+\beta}$.

If $\left\Vert u\right\Vert _{\nu,\kappa,\beta }$ is finite, then the approximating functions of $\left(I-L^{\nu}\right)^{\kappa}u$
\begin{eqnarray*}
\left(\left(I-L^{\nu}\right)^{\kappa}u\right)_n=\left(I-L^{\nu}\right)^{\kappa}u_n\in C_b^{\infty}\left( \mathbf{R}^d\right)\cap \tilde{C}^{\beta}_{\infty,\infty}\left( \mathbf{R}^d\right).
\end{eqnarray*}
Because $\left(I-L^{\nu}\right)^{\kappa}$ is a bijection on $ C_b^{\infty}\left( \mathbf{R}^d\right)$, $u_n\in  C_b^{\infty}\left( \mathbf{R}^d\right)$. Then Corollary \ref{pr44} implies immediately 
\begin{eqnarray*}
	\left\vert u_n\right\vert_{0}\leq C\left\vert u_n\right\vert_{\kappa+\beta,\infty}\leq C\left\vert\left(I-L^{\nu}\right)^{\kappa}u_n\right\vert_{\beta,\infty}\leq C\left\vert\left(I-L^{\nu}\right)^{\kappa}u\right\vert_{\beta,\infty}.
\end{eqnarray*}

For $j\geq 2, n\geq j-1$, 
\begin{eqnarray*}
	\left\vert u_n\ast\varphi_j\right\vert_{0}=\left\vert u\ast\tilde{\varphi}_j\ast\varphi_j\right\vert_{0}=\left\vert u\ast\varphi_j\right\vert_{0}\leq Cw\left(N^{-j}\right)^{\kappa+\beta}\left\vert \left(I-L^{\nu}\right)^{\kappa}u\right\vert_{\beta,\infty},
\end{eqnarray*}
and for $j=0$ or $1$, $n\geq j-1$, 
\begin{eqnarray*}
\left\vert u_n\ast\varphi_j\right\vert_{0}=\left\vert u\ast\varphi_j\right\vert_{0}\leq C\sup_{n}\left\vert u_n\right\vert_{0}\leq C\left\vert\left(I-L^{\nu}\right)^{\kappa}u\right\vert_{\beta,\infty}.
\end{eqnarray*}
Therefore, $\left\vert u\right\vert_{\kappa+\beta,\infty}\leq C\left\vert \left(I-L^{\nu}\right)^{\kappa}u\right\vert_{\beta,\infty}$.
\end{proof}

\begin{theorem}\label{thm6}
	Let $\nu$ be a L\'{e}vy measure satisfying \textbf{A(w,l)}, $\beta\in\left( 0,\infty\right),\kappa\in\left( 0,1\right]$. Then norm $\left\Vert u\right\Vert _{\nu,\kappa,\beta }$ and norm $\left\vert u\right\vert _{\nu,\kappa,\beta }$ are equivalent.	
\end{theorem}
\begin{proof}
	This is an immediate consequence of Theorems \ref{thm4} and \ref{thm5}.
\end{proof}

\section{Solution Estimates for Smooth Inputs}
\subsection{Existence and Uniqueness}

\begin{theorem}\label{thm1}
	Let $\nu$ be a L\'{e}vy measure, $\alpha\in\left( 0,2\right), \beta\in\left(0,1\right), \lambda\geq 0$. Assume that $f\left( t,x\right)\in  C_b^{\infty}\left(H_T\right)\cap \tilde{C}^{\beta}_{\infty,\infty}\left(H_T\right)$. Then there is a unique solution $u\in\left( t,x\right)\in C_b^{\infty}\left(H_T\right)$ to 
	\begin{eqnarray}\label{eeq1}
	\partial_t u\left( t,x\right)&=&L^{\nu}u\left( t,x\right)-\lambda u\left( t,x\right)+f\left( t,x\right), \\
	u\left( 0,x\right)&=& 0,\qquad\left( t,x\right)\in \left[0,T\right]\times\mathbf{R}^d.\nonumber
	\end{eqnarray}
\end{theorem}
\begin{proof}
	\textsc{Existence. }Denote $F\left(r,Z^{\nu}_r\right)=e^{-\lambda\left( r-s\right)}f\left(s, x+Z^{\nu}_r-Z^{\nu}_s\right), s\leq r\leq t,$ and apply the It\^{o} formula to $F\left(r,Z^{\nu}_r\right)$ on $\left[s,t\right]$.
	\begin{eqnarray*}
		&& e^{-\lambda\left( t-s\right)}f\left(s, x+Z^{\nu}_t-Z^{\nu}_s\right)-f\left( s,x\right)\\
		&=& -\lambda\int_s^t  F\left(r,Z^{\nu}_{r}\right)dr+\int_s^t\int \chi_{\alpha}\left( y\right)y\cdot\nabla F\left(r, Z^{\nu}_{r-}\right)\tilde{J}\left(dr,dy\right)\\
		&&+ \int_s^t \int \left[ F\left(r,Z^{\nu}_{r-}+y\right)-F\left(r,Z^{\nu}_{r-}\right)-\chi_{\alpha}\left(y\right) y\cdot \nabla F\left(r,Z^{\nu}_{r-}\right)\right]J\left( dr,dy\right).
	\end{eqnarray*}
	Take expectation for both sides and use the stochastic Fubini theorem,
	\begin{eqnarray*}
		&&  e^{-\lambda\left( t-s\right)}\mathbf{E}f\left(s, x+Z^{\nu}_t-Z^{\nu}_s\right)-f\left( s,x\right)\\
		&=& -\lambda\int_s^t  e^{-\lambda\left( r-s\right)}\mathbf{E}f\left(s, x+Z^{\nu}_r-Z^{\nu}_s\right)dr+\int_s^t L^{\nu} e^{-\lambda\left( r-s\right)}\mathbf{E}f\left(s, x+Z^{\nu}_r-Z^{\nu}_s\right)dr.
	\end{eqnarray*}
	Integrate both sides over $\left[0,t\right]$ with respect to $s$ and obtain
	\begin{eqnarray*}
		&&  \int_0^t e^{-\lambda\left( t-s\right)}\mathbf{E}f\left(s, x+Z^{\nu}_t-Z^{\nu}_s\right)ds-\int_0^t f\left( s,x\right)ds\\
		&=& -\lambda\int_0^t\int_0^r  e^{-\lambda\left( r-s\right)}\mathbf{E}f\left(s, x+Z^{\nu}_r-Z^{\pi}_s\right)dsdr\\
		&&+\int_0^t L^{\nu}\int_0^r  e^{-\lambda\left( r-s\right)}\mathbf{E}f\left(s, x+Z^{\nu}_r-Z^{\nu}_s\right)dsdr,
	\end{eqnarray*}
	which shows $u\left(t,x\right)=\int_0^t e^{-\lambda\left( t-s\right)}\mathbf{E}f\left(s, x+Z_{t-s}^{\nu}\right)ds$ solves $\eqref{eeq1}$ in the integral sense. Obviously, as a result of  the dominated convergence theorem and Fubini's theorem, $u\in C^{\infty}_b\left( H_T\right)$. And by the equation, $u$ is continuously differentiable in $t$.
	
	\textsc{Uniqueness. }Suppose there are two solutions $u_1,u_2$ solving the equation, then $u:=u_1-u_2$ solves 
	\begin{eqnarray}\label{uni}
	\partial_t u\left( t,x\right)&=&L^{\nu} u\left(t,x\right)-\lambda u\left(t,x\right),\\
	u\left( 0,x\right)&=& 0.\nonumber
	\end{eqnarray}
	
	Fix any $t\in\left[0,T\right]$. Apply the It\^{o} formula to $v\left( t-s,Z^{\nu}_s\right):=e^{-\lambda s}u\left(t-s,x+Z^{\nu}_s\right)$, $0\leq s\leq t,$ over $\left[0,t\right]$ and take expectation for both sides of the resulting identity, then 
	\begin{eqnarray*}
	u\left( t,x\right)=-\mathbf{E}\int_0^t e^{-\lambda s}\left[ \left(-\partial_t u-\lambda u+L^{\nu}u\right)\circ\left(t-s,x+Z^{\nu}_{s-}\right)\right]ds=0.
	\end{eqnarray*}
\end{proof}

\subsection{H\"{o}lder-Zygmund Estimates of the Solution}

Since $f\left(t,x\right)\in C_b^{\infty}\left( H_T\right)\cap \tilde{C}^{\beta}_{\infty,\infty}\left(H_T\right)$, by Lemma \ref{equiv}, 
\begin{eqnarray*}
f\left(t,x\right)&=&\left(f\left(t,\cdot\right)\ast \varphi_0\left(\cdot\right)\right)\left(x\right)+\sum_{j=1}^{\infty}\left(f\left(t,\cdot\right)\ast \varphi_j\left(\cdot\right)\right)\left(x\right)\\
&:=&f_0\left(t,x\right)+\sum_{j=1}^{\infty}f_j\left(t,x\right).
\end{eqnarray*}
Accordingly, $u_j\left(t,x\right)=u\left(t,x\right)\ast\varphi_j\left(x\right)=\int_0^t e^{-\lambda\left( t-s\right)}\mathbf{E}f_j\left(s, x+Z_{t-s}^{\nu}\right)ds, j=0,1\ldots$ is the solution to $\eqref{eeq1}$ with input $f_j=f\ast \varphi_j$. Then by Lemmas \ref{rep} and \ref{Lop}, for $\kappa\in\left(0,1\right]$,
\begin{eqnarray*}
	&&L^{\mu,\kappa}u_j\left(t,x\right)=u_j\ast L^{\mu,\kappa}\tilde{\varphi}_j\\
	&=&\int_0^t e^{-\lambda\left( t-s\right)}\mathbf{E}\int f_j\left(s, x-z+Z_{t-s}^{\nu}\right)L^{\mu,\kappa}\tilde{\varphi}_j\left( z\right)dz ds\\
	&=&\int_0^t e^{-\lambda\left( t-s\right)}\int f_j\left(s, z\right)\mathbf{E}L^{\mu,\kappa}\tilde{\varphi}_j\left( x-z+Z_{t-s}^{\nu}\right)dz ds,
\end{eqnarray*}
and then
\begin{eqnarray}
L^{\mu,\kappa}u_j\left(t,x\right)
	&=& \int_0^t\int \mathcal{F}^{-1}\left[e^{\left(\psi ^{\nu }\left( \xi \right)-\lambda\right)\left(t-s\right)}\widehat{L^{\mu,\kappa}\tilde{\varphi}_j}\right]\left( x-z\right)f_j\left(s, z\right)dz ds \nonumber\\
	&=& C \int_0^t \int \mathcal{F}^{-1}\left[\tilde{F}_{t-s}^{j,\kappa}\left( \xi \right)\right]\left( z\right)f_j\left(s, x-z\right)dzds, \label{rp}
\end{eqnarray}
where for $j\in\mathbf{N}$,
\begin{eqnarray}
	\quad&&\tilde{F}_{t}^{j,\kappa}\left( \xi \right) :=	\left\{\begin{array}{ll}\label{rr}
	-e^{\left(\psi ^{\nu }\left( \xi \right)-\lambda\right)t} \left( -\Re\psi^{\mu}\left( \xi\right)\right)^{\kappa} \hat{\tilde{\varphi}}_j\left(\xi\right) ,& \xi \in \mathbf{R}^{d},\kappa\in\left( 0,1\right),\\
	e^{\left(\psi ^{\nu }\left( \xi \right)-\lambda\right)t} \psi^{\mu}\left( \xi\right) \hat{\tilde{\varphi}}_j\left(\xi\right), & \xi \in \mathbf{R}^{d},\kappa=1,\\
	e^{\left(\psi ^{\nu }\left( \xi \right)-\lambda\right)t}  \hat{\tilde{\varphi}}_j\left(\xi\right), & \xi \in \mathbf{R}^{d},\kappa=0.
	\end{array}\right. 
\end{eqnarray}

In particular, when $j\in\mathbf{N}_{+}$,
\begin{eqnarray}
\mathcal{F}^{-1}\left[\tilde{F}_{t}^{j,\kappa}\left( \xi \right)\right]\left( x\right)=e^{-\lambda t}w\left( N^{-j}\right)^{-\kappa}N^{jd}H^{j,\kappa}_{w\left( N^{-j}\right)^{-1}t}\left( N^j x\right),\label{rrp}
\end{eqnarray}
where for $j\in\mathbf{N}_{+}$,
\begin{eqnarray*}\label{rrr}
H_{t}^{j,\kappa}:=\left\{\begin{array}{ll}
		\mathcal{F}^{-1}\left[-\exp\{\psi ^{\tilde{\nu}_{N^{-j}} }\left( \xi \right)t\} \left( -\Re\psi^{\tilde{\mu}_{N^{-j}}}\left( \xi\right)\right)^{\kappa} \hat{\tilde{\varphi}}\left(\xi\right)\right] ,& \kappa\in\left( 0,1\right),\\
		\mathcal{F}^{-1}\left[\exp\{\psi ^{\tilde{\nu}_{N^{-j}} }\left( \xi \right)t\} \psi^{\tilde{\mu}_{N^{-j}}}\left( \xi\right) \hat{\tilde{\varphi}}\left(\xi\right)\right] ,& \kappa=1,\\
		\mathcal{F}^{-1}\left[\exp\{\psi ^{\tilde{\nu}_{N^{-j}} }\left( \xi \right)t\} \hat{\tilde{\varphi}}\left(\xi\right)\right] ,& \kappa=0.
	\end{array}\right. 
\end{eqnarray*}
\begin{lemma}\label{lem1}
Let $\kappa\in\left[ 0,1\right]$. For all $t\in\left[0,T\right]$ and $j\in\mathbf{N}_{+}$, there is $C_1,C_2>0$ depending only on $\alpha,d,N,\alpha_1,\alpha_2,\kappa$ such that
\begin{eqnarray*}
	\int \left\vert H^{j,\kappa}_{t}\left(  x\right)\right\vert dx\leq C_1e^{-C_2 t}.
\end{eqnarray*}
\end{lemma}
\begin{proof}
	Recall that $\hat{\tilde{\varphi}}\left(\xi\right)=\phi\left( N\xi\right)+\phi\left( \xi\right)+\phi\left( N^{-1}\xi\right)$. If we introduce $\tilde{\phi}$ such that $\hat{\tilde{\phi}}\left(\xi\right)=\phi\left( N^2\xi\right)+\phi\left( N\xi\right)+\phi\left( \xi\right)+\phi\left( N^{-1}\xi\right)+\phi\left( N^{-2}\xi\right)$, then  $\hat{\tilde{\varphi}}=\hat{\tilde{\varphi}}\hat{\tilde{\phi}}$ and $\hat{\tilde{\varphi}},\hat{\tilde{\phi}}\in C_0^{\infty}\left(\mathbf{R}^d\right)$. We write
	\begin{eqnarray*}
		H_{t}^{j,0} &=& \mathcal{F}^{-1}[\hat{\tilde{\phi}}\left( \xi \right)\exp \left\{ \psi ^{\tilde{\nu}_{N^{-j}}}\left( \xi \right)
		t\right\}\hat{\tilde{\varphi}}\left( \xi \right) ],\\
		H_{t}^{j,1} &=& \mathcal{F}^{-1}[\psi^{\tilde{\mu}_{N^{-j}}}\left( \xi\right) \hat{\tilde{\phi}}\left( \xi \right)\exp \left\{ \psi ^{\tilde{\nu}_{N^{-j}}}\left( \xi \right)
		t\right\}\hat{\tilde{\varphi}}\left( \xi \right) ],\\
		H_{t}^{j,\kappa} &=& \mathcal{F}^{-1}[-\left( -\Re\psi^{\tilde{\mu}_{N^{-j}}}\left( \xi\right)\right)^{\kappa} \hat{\tilde{\phi}}\left( \xi \right)\exp \left\{ \psi ^{\tilde{\nu}_{N^{-j}}}\left( \xi \right)
		t\right\}\hat{\tilde{\varphi}}\left( \xi \right) ],\quad \kappa\in\left( 0,1\right).
	\end{eqnarray*}

	Thus, for all $\kappa\in\left[ 0,1\right]$, 
	\begin{eqnarray*}
		H_{t}^{j,\kappa}\left( x\right) = \int\left[L^{\tilde{\mu}_{N^{-j}},\kappa}\tilde{\phi}\left(x-z\right)\right]\cdot\left[\mathbf{E}\tilde{\varphi}\left( z+Z_t^{\tilde{\nu}_{N^{-j}}} \right)\right]dz,x\in\mathbf{R}^d,
	\end{eqnarray*}
and thus
\begin{eqnarray*}
	\left\vert H_{t}^{j,\kappa}\right\vert_{L^1\left(\mathbf{R}^d\right)}
	\leq \left\vert L^{\tilde{\mu}_{N^{-j}},\kappa}\tilde{\phi}\right\vert_{L^1\left(\mathbf{R}^d\right)}\left\vert \mathbf{E}\tilde{\varphi}\left( \cdot+Z_t^{\tilde{\nu}_{N^{-j}}} \right)\right\vert_{L^1\left(\mathbf{R}^d\right)}.
\end{eqnarray*}
	
	Since $\nu$ verifies \textbf{A(w,l)}, by Lemma \ref{lemma2}, there exist positive constants $C_1,C_2$ depending only on $\alpha,d,N,\alpha_1,\alpha_2$, such that
	\begin{eqnarray*}
	\left\vert\mathbf{E}\tilde{\varphi}\left( \cdot+Z_t^{\tilde{\nu}_{N^{-j}}} \right)\right\vert_{L^1\left(\mathbf{R}^d\right)} <C_1e^{-C_2 t}.
	\end{eqnarray*}
Combining Lemmas \ref{Lop} and \ref{rep}, we then arrive at the conclusion.
\end{proof}
\begin{lemma}\label{lem2}
Let $\kappa\in\left[ 0,1\right]$. For all $t\in\left[0,T\right]$ and $j\in\mathbf{N}$, there is $C>0$ depending only on $\alpha,d,N,\alpha_1,\alpha_2$ such that
\begin{eqnarray*}
\int \left\vert\mathcal{F}^{-1}\left[\tilde{F}_{t}^{0,\kappa}\left( \xi \right)\right]\left( x\right)\right\vert dx&<&C,\\
\int_0^t\int \left\vert\mathcal{F}^{-1}\left[\tilde{F}_{r}^{j,\kappa}\left( \xi \right)\right]\left( x\right)\right\vert dxdr&<&C,j\in\mathbf{N}_{+}.
\end{eqnarray*}
\end{lemma}	
\begin{proof}
First by Lemmas \ref{Lop} and \ref{rep},
	\begin{eqnarray*}
	\int \left\vert\mathcal{F}^{-1}\left[\tilde{F}_{t}^{0,\kappa}\left( \xi \right)\right]\left( x\right)\right\vert dx\leq \int \left\vert \mathbf{E}L^{\mu,\kappa}\tilde{\varphi}_0\left(x+Z^{\nu}_t\right)\right\vert dx\leq \int \left\vert L^{\mu,\kappa}\tilde{\varphi}_0\left(x\right)\right\vert dx<C.
	\end{eqnarray*}

For $j\in\mathbf{N}_{+}$, use Lemma \ref{lem1}.
\begin{eqnarray*}
	&&\int_0^t\int \left\vert\mathcal{F}^{-1}\left[\tilde{F}_{r}^{j,\kappa}\left( \xi \right)\right]\left( x\right)\right\vert dxdr\\
	&\leq& \int_0^t\int \left\vert w\left( N^{-j}\right)^{-\kappa}H^{j,\kappa}_{w\left( N^{-j}\right)^{-1}r}\left(x\right)\right\vert dxdr\\
	&\leq& w\left( N^{-j}\right)^{1-\kappa}\int_0^{\infty}\int \left\vert H^{j,\kappa}_{r}\left(x\right)\right\vert dxdr<C.
\end{eqnarray*}
\end{proof}
\begin{corollary}\label{col1}
	Let $\kappa\in\left[ 0,1\right]$ and $u$ be the solution to $\eqref{eeq1}$ and $\mu$ be the reference measure. Then there exists $C>0$ depending only on $\alpha,d,N,\kappa,\alpha_1,\alpha_2, T$ such that
	\begin{eqnarray*}
	\left\vert L^{\mu,\kappa}u_j\right\vert_0\leq C\left\vert f_j\right\vert_0, j\in\mathbf{N}.
	\end{eqnarray*} 
\end{corollary}
\begin{proof}
	Recall that 
	\begin{eqnarray*}
		L^{\mu,\kappa}u_j\left(t,x\right)=C \int_0^t \int \mathcal{F}^{-1}\left[\tilde{F}_{t-s}^{j,\kappa}\left( \xi \right)\right]\left( z\right)f_j\left(s, x-z\right)dzds, j\in\mathbf{N}. 
	\end{eqnarray*}
Therefore, by Lemma \ref{lem2}, for all $t\in\left[0,T\right]$,
\begin{eqnarray*}
	\left\vert L^{\mu,\kappa}u_j\right\vert_0\leq C\left\vert f_j\right\vert_0 \int_0^T \int \left\vert \mathcal{F}^{-1}\left[\tilde{F}_{t-s}^{j,\kappa}\left( \xi \right)\right]\left( z\right)\right\vert dzds \leq C\left\vert f_j\right\vert_0, j\in\mathbf{N}. 
\end{eqnarray*}
\end{proof}
\begin{lemma}\label{lem3}
	Let $\kappa\in\left[ 0,1\right]$ and $\mu$ be the reference measure. Both $\mu$ and $\nu$ satisfy \textbf{A(w,l)}. Then there is $C>0$ depending only on $\alpha,d,N,\alpha_1,\alpha_2,\kappa$, such that for all $0\leq s<t$,
	\begin{eqnarray*}
		\left\vert\mathbf{E}\left[ L^{\mu,\kappa}\tilde{\varphi}_0\left(\cdot+Z^{\nu}_{t}\right)-L^{\mu,\kappa}\tilde{\varphi}_0\left(\cdot+Z^{\nu}_{s}\right)\right]\right\vert_{L^1\left(\mathbf{R}^d\right)}\leq C\left( t-s\right).
	\end{eqnarray*}
\end{lemma}
\begin{proof}
	Denote $\bar{\varphi}_0=\mathcal{F}^{-1}\left[ \mathcal{F}\tilde{\varphi}_0\left( \xi\right)+\phi\left( N^{-2}\xi\right)\right]$, then $\bar{\varphi}_0\in\mathcal{S}\left( \mathbf{R}^d\right)$ and $\mathcal{F}\tilde{\varphi}_0\mathcal{F}\bar{\varphi}_0=\mathcal{F}\tilde{\varphi}_0$. And then, 
	\begin{eqnarray*}
		&&\left\vert\mathbf{E}\left[ L^{\mu}\tilde{\varphi}_0\left(\cdot+Z^{\nu}_{t}\right)-L^{\mu}\tilde{\varphi}_0\left(\cdot+Z^{\nu}_{s}\right)\right]\right\vert_{L^1\left(\mathbf{R}^d\right)}\\
		&=&\left\vert\mathcal{F}^{-1}[\psi^{\mu}\left( \xi\right) \hat{\tilde{\varphi}}_0\left( \xi \right)\left(e^{ \psi ^{\nu}\left( \xi \right)
			t}-e^{ \psi ^{\nu}\left( \xi \right)
			s}\right)\hat{\bar{\varphi}}_0\left( \xi \right) ]\right\vert_{L^1\left(\mathbf{R}^d\right)}\\
		&\leq&\left\vert L^{\mu}\tilde{\varphi}_0\right\vert_{L^1\left(\mathbf{R}^d\right)}\left\vert \mathbf{E}\left[\bar{\varphi}_0\left( \cdot+Z_{t}^{\nu} \right)-\bar{\varphi}_0\left( \cdot+Z_{s}^{\nu} \right)\right]\right\vert_{L^1\left(\mathbf{R}^d\right)}\\
		&\leq& \left\vert L^{\mu}\tilde{\varphi}_0\right\vert_{L^1\left(\mathbf{R}^d\right)}\left\vert \mathbf{E}\int_s^t L^{\nu}\bar{\varphi}_0\left( \cdot+Z_{r-}^{\nu} \right)dr\right\vert_{L^1\left(\mathbf{R}^d\right)},
		\end{eqnarray*}
 thus, by Lemmas \ref{Lop}, \ref{rep},
 	\begin{eqnarray*}
 	\left\vert\mathbf{E}\left[ L^{\mu}\tilde{\varphi}_0\left(\cdot+Z^{\nu}_{t}\right)-L^{\mu}\tilde{\varphi}_0\left(\cdot+Z^{\nu}_{s}\right)\right]\right\vert_{L^1\left(\mathbf{R}^d\right)}\leq C\left( t-s\right).
 \end{eqnarray*}
 	Similarly, for $\kappa\in\left(0,1\right)$,
	\begin{eqnarray*}
		&&\left\vert\mathbf{E}\left[ L^{\mu,\kappa}\tilde{\varphi}_0\left(\cdot+Z^{\nu}_{t}\right)-L^{\mu,\kappa}\tilde{\varphi}_0\left(\cdot+Z^{\nu}_{s}\right)\right]\right\vert_{L^1\left(\mathbf{R}^d\right)}\\
		&=&\left\vert\mathcal{F}^{-1}[-\left(-\Re\psi^{\tilde{\mu}_{N^{-j}}}\left( \xi\right)\right)^{\kappa} \hat{\tilde{\varphi}}_0\left( \xi \right)\left(e^{ \psi ^{\nu}\left( \xi \right)
			t}-e^{ \psi ^{\nu}\left( \xi \right)
			s}\right)\hat{\bar{\varphi}}_0\left( \xi \right) ]\right\vert_{L^1\left(\mathbf{R}^d\right)}\\
		&\leq& \left\vert L^{\mu,\kappa}\tilde{\varphi}_0\right\vert_{L^1\left(\mathbf{R}^d\right)}\left\vert \mathbf{E}\int_s^t L^{\nu}\bar{\varphi}_0\left( \cdot+Z_{r-}^{\nu} \right)dr\right\vert_{L^1\left(\mathbf{R}^d\right)}\leq C\left( t-s\right),
	\end{eqnarray*}
and for $\kappa=0$,
\begin{eqnarray*}
	&&\left\vert\mathbf{E}\left[ \tilde{\varphi}_0\left(\cdot+Z^{\nu}_{t}\right)-\tilde{\varphi}_0\left(\cdot+Z^{\nu}_{s}\right)\right]\right\vert_{L^1\left(\mathbf{R}^d\right)}\\
	&\leq& \left\vert \tilde{\varphi}_0\right\vert_{L^1\left(\mathbf{R}^d\right)}\left\vert \mathbf{E}\int_s^t L^{\nu}\bar{\varphi}_0\left( \cdot+Z_{r-}^{\nu} \right)dr\right\vert_{L^1\left(\mathbf{R}^d\right)}\leq C\left( t-s\right).
\end{eqnarray*}
\end{proof}

The next Lemma is a stronger version of Lemma \ref{lem3} because the Fourier transform of the underlying Schwartz function has a compact support that is away from 0.
\begin{lemma}\label{lem4}
	Let $\kappa\in\left[ 0,1\right]$ and $j\in\mathbf{N}_{+}$.  Then there are $C_1, C_2>0$ depending only on $\alpha,d,N,\alpha_1,\alpha_2,\kappa$, such that for all $0\leq s<t$,
	\begin{eqnarray*}
		\int \left\vert H^{j,\kappa}_{t}\left(  x\right)-H^{j,\kappa}_{s}\left(  x\right)\right\vert dx\leq C_1e^{-C_2 s}\left( t-s\right).
	\end{eqnarray*}
\end{lemma}
\begin{proof}
	Similarly as what we did in Lemma \ref{lem1}, we introduce $\tilde{\phi}$ such that $\hat{\tilde{\phi}}\left(\xi\right)=\phi\left( N^2\xi\right)+\hat{\tilde{\varphi}}\left(\xi\right)+\phi\left( N^{-2}\xi\right)$. As a consequence, $\hat{\tilde{\varphi}}=\hat{\tilde{\varphi}}\hat{\tilde{\phi}}\hat{\tilde{\phi}}$, and $\hat{\tilde{\varphi}},\hat{\tilde{\phi}},\hat{\tilde{\phi}}\in C_0^{\infty}\left(\mathbf{R}^d\right)$. Then
		\begin{eqnarray*}
	H_{t}^{j,0}-H_{s}^{j,0}=\mathcal{F}^{-1}[\hat{\tilde{\phi}}\left( \xi \right)e^{ \psi ^{\tilde{\nu}_{N^{-j}}}\left( \xi \right)
			s}\hat{\tilde{\phi}}\left( \xi \right)\left(e^{ \psi ^{\tilde{\nu}_{N^{-j}}}\left( \xi \right)
			\left(	t-s\right)}-1\right)\hat{\tilde{\varphi}}\left( \xi \right) ],
	\end{eqnarray*}
and
		\begin{eqnarray*}
		&&H_{t}^{j,1}-H_{s}^{j,1}\\
		&=&\mathcal{F}^{-1}[\psi^{\tilde{\mu}_{N^{-j}}}\left( \xi\right) \hat{\tilde{\phi}}\left( \xi \right)e^{ \psi ^{\tilde{\nu}_{N^{-j}}}\left( \xi \right)
			s}\hat{\tilde{\phi}}\left( \xi \right)\left(e^{ \psi ^{\tilde{\nu}_{N^{-j}}}\left( \xi \right)
	\left(	t-s\right)}-1\right)\hat{\tilde{\varphi}}\left( \xi \right) ],
	\end{eqnarray*}
and for $\kappa\in\left( 0,1\right)$,
	\begin{eqnarray*}
	&&H_{t}^{j,\kappa}-H_{s}^{j,\kappa}\\
	&=&\mathcal{F}^{-1}[-\left( -\Re\psi^{\tilde{\mu}_{N^{-j}}}\left( \xi\right)\right)^{\kappa} \hat{\tilde{\phi}}\left( \xi \right)e^{ \psi ^{\tilde{\nu}_{N^{-j}}}\left( \xi \right)
		s}\hat{\tilde{\phi}}\left( \xi \right)\left(e^{ \psi ^{\tilde{\nu}_{N^{-j}}}\left( \xi \right)
		\left(	t-s\right)}-1\right)\hat{\tilde{\varphi}}\left( \xi \right) ].
	\end{eqnarray*}
	
	Thus, for all $\kappa\in\left[ 0,1\right]$, 
	\begin{eqnarray*}
		&&H_{t}^{j,\kappa}-H_{s}^{j,\kappa }\\
		&=&L^{\tilde{\mu}_{N^{-j}},\kappa}\tilde{\phi}\left(\cdot\right)\ast\mathbf{E}\tilde{\phi}\left( \cdot+Z_s^{\tilde{\nu}_{N^{-j}}} \right)\ast\mathbf{E}\left[\tilde{\varphi}\left( \cdot+Z_{t-s}^{\tilde{\nu}_{N^{-j}}} \right)-\tilde{\varphi}\left( \cdot \right)\right]\\
		&=&L^{\tilde{\mu}_{N^{-j}},\kappa}\tilde{\phi}\left(\cdot\right)\ast\mathbf{E}\tilde{\phi}\left( \cdot+Z_s^{\tilde{\nu}_{N^{-j}}} \right)\ast\mathbf{E}\int_0^{t-s} L^{\tilde{\nu}_{N^{-j}}}\tilde{\varphi}\left( \cdot+Z_{r-}^{\tilde{\nu}_{N^{-j}}} \right)dr,
	\end{eqnarray*}
	and thus by Lemmas \ref{Lop}, \ref{rep} and \ref{lemma2},
	\begin{eqnarray*}
		&&\int \left\vert H^{j,\kappa}_{t}\left(  x\right)-H^{j,\kappa}_{s}\left(  x\right)\right\vert dx\\
		&\leq& \int_0^{t-s}\left\vert L^{\tilde{\mu}_{N^{-j}},\kappa}\tilde{\phi}\right\vert_{L^1\left(\mathbf{R}^d\right)}dr\left\vert \mathbf{E}\tilde{\phi}\left( \cdot+Z_s^{\tilde{\nu}_{N^{-j}}} \right)\right\vert_{L^1\left(\mathbf{R}^d\right)}\left\vert L^{\tilde{\nu}_{N^{-j}}}\tilde{\varphi}\right\vert_{L^1\left(\mathbf{R}^d\right)}\\
		&\leq& C_1e^{-C_2 s}\left( t-s\right).
	\end{eqnarray*}
\end{proof}

\begin{lemma}\label{lem5}
	Let $u$ be the solution to $\eqref{eeq1}$, $\mu$ be the reference measure and $\kappa\in\left[ 0,1\right]$. Then there exists $C>0$ depending only on $\alpha,d,N,\alpha_1,\alpha_2,\kappa, T$ such that for all $0\leq s<t\leq T$,
	\begin{eqnarray*}
		\left\vert L^{\mu,\kappa}u_0\left( t,x\right)-L^{\mu,\kappa}u_0\left( s,x\right)\right\vert&\leq& C\left( t-s\right)\left\vert f_0\right\vert_0,\forall x\in\mathbf{R}^d,\\
		\left\vert L^{\mu,\kappa}u_j\left( t,x\right)-L^{\mu,\kappa}u_j\left( s,x\right)\right\vert&\leq& C\left( t-s\right)^{1-\kappa}\left\vert f_j\right\vert_0, \forall x\in\mathbf{R}^d,j\in\mathbf{N}_{+}.
	\end{eqnarray*} 
\end{lemma}
\begin{proof}
	According to $\eqref{rp}$,
	\begin{eqnarray*}
		&&\left\vert L^{\mu,\kappa}u_j\left(t,x\right)-L^{\mu,\kappa}u_j\left(s,x\right)\right\vert\\
		&\leq& C\left\vert f_j\right\vert_0 \int_s^t \int\left\vert \mathcal{F}^{-1}\left[\tilde{F}_{t-r}^{j,\kappa}\left( \xi \right)\right]\left( z\right)\right\vert dzdr \\
		&&+ C\left\vert f_j\right\vert_0\int_0^s \int\left\vert \mathcal{F}^{-1}\left[\tilde{F}_{t-r}^{j,\kappa}\left( \xi \right)-\tilde{F}_{s-r}^{j,\kappa}\left( \xi \right)\right]\left( z\right)\right\vert dzdr \\
		&:=& C\left\vert f_j\right\vert_0\left( I_1+I_2\right), \quad j\in\mathbf{N}. 
	\end{eqnarray*}

When $j=0$, Lemma \ref{lem2} implies
\begin{eqnarray*}
	I_1=\int_s^t \int\left\vert \mathcal{F}^{-1}\left[\tilde{F}_{t-r}^{j,\kappa}\left( \xi \right)\right]\left( z\right)\right\vert dzdr \leq C\left( t-s\right),\forall \kappa\in\left[0,1\right]. 
\end{eqnarray*}
When $j\neq 0$, recall $\eqref{rrp}$.
\begin{eqnarray*}
	I_1 &\leq& \int_s^t \int\left\vert w\left( N^{-j}\right)^{-\kappa}H^{j,\kappa}_{w\left( N^{-j}\right)^{-1}\left(t-r\right)}\left( z\right)\right\vert dzdr  \\
	&=& w\left( N^{-j}\right)^{1-\kappa}\int_0^{w\left( N^{-j}\right)^{-1}\left(t-s\right)} \int\left\vert H^{j,\kappa}_{r}\left( z\right)\right\vert dzdr . 
\end{eqnarray*}
If $w\left( N^{-j}\right)^{-1}\left(t-s\right)\leq 1$, since $\int\left\vert H^{j,\kappa}_{r}\left( z\right)\right\vert dz<C$ by Lemma \ref{lem1},
\begin{eqnarray*}
I_1\leq Cw\left( N^{-j}\right)^{1-\kappa}w\left( N^{-j}\right)^{-1}\left(t-s\right)\leq C\left(t-s\right)^{1-\kappa}.
\end{eqnarray*}	
If $w\left( N^{-j}\right)^{-1}\left(t-s\right)> 1$, again use Lemma \ref{lem1}.
\begin{eqnarray*}
	I_1\leq w\left( N^{-j}\right)^{1-\kappa}\int_0^{\infty} \int\left\vert H^{j,\kappa}_{r}\left( z\right)\right\vert dzdr\leq Cw\left( N^{-j}\right)^{1-\kappa}<C\left(t-s\right)^{1-\kappa}.
\end{eqnarray*}	

Next we investigate $I_2$. Recall definitions $\eqref{rr}-\eqref{rrr}$. When $j=0$ and $\kappa\in\left[0,1\right]$,
\begin{eqnarray*}
	I_2&=&\int_0^s \int\left\vert \mathcal{F}^{-1}\left[\tilde{F}_{t-r}^{0,\kappa}\left( \xi \right)-\tilde{F}_{s-r}^{0,\kappa}\left( \xi \right)\right]\left( z\right)\right\vert dzdr\\
	&\leq& \left\vert e^{-\lambda\left( t-s\right)}-1\right\vert\int_0^s e^{-\lambda\left( s-r\right)}\int\left\vert \mathcal{F}^{-1}\left[-e^{\psi ^{\nu }\left( \xi \right)\left(t-r\right)} \mathcal{F}L^{\mu,\kappa}\tilde{\varphi}_0\left(\xi\right)\right]\left( z\right)\right\vert dzdr\\
	&+&\int_0^s \int\left\vert \mathcal{F}^{-1}\left[-\left(e^{\psi ^{\nu }\left( \xi \right)\left(t-r\right)}-e^{\psi ^{\nu }\left( \xi \right)\left(s-r\right)}\right) \mathcal{F}L^{\mu,\kappa}\tilde{\varphi}_0\left(\xi\right)\right]\left( z\right)\right\vert dzdr\\
	&:=& I_{21}+I_{22}.
\end{eqnarray*}
Therefore, by Lemmas \ref{Lop} and \ref{rep},
\begin{eqnarray*}
	I_{21}&\leq& \frac{2T}{\lambda}\left\vert e^{-\lambda\left( t-s\right)}-1\right\vert \left\vert\mathbf{E} L^{\mu,\kappa}\tilde{\varphi}_0\left(\cdot+Z^{\nu}_{t-r}\right)\right\vert_{L^1\left(\mathbf{R}^d\right)}\\
	&\leq& \frac{2T}{\lambda}\left\vert e^{-\lambda\left( t-s\right)}-1\right\vert \left\vert L^{\mu,\kappa}\tilde{\varphi}_0\right\vert_{L^1\left(\mathbf{R}^d\right)}\leq C\left( t-s\right),\quad\kappa\in\left[0,1\right].
\end{eqnarray*}
Meanwhile, by Lemma \ref{lem3},
\begin{eqnarray*}
	I_{22}&\leq& T \left\vert\mathbf{E}\left[ L^{\mu,\kappa}\tilde{\varphi}_0\left(\cdot+Z^{\nu}_{t-r}\right)-L^{\mu,\kappa}\tilde{\varphi}_0\left(\cdot+Z^{\nu}_{s-r}\right)\right]\right\vert_{L^1\left(\mathbf{R}^d\right)}\\
	&\leq&  C\left( t-s\right),\quad\kappa\in\left[0,1\right].
\end{eqnarray*}

When $j\neq 0$,
\begin{eqnarray*}
	I_2
	&\leq& \left\vert e^{-\lambda\left( t-s\right)}-1\right\vert w\left( N^{-j}\right)^{-\kappa}\int_0^s e^{-\lambda\left( s-r\right)} \int\left\vert H^{j,\kappa}_{w\left( N^{-j}\right)^{-1}\left(s-r\right)}\left( z\right)\right\vert dzdr\\
	&+&w\left( N^{-j}\right)^{-\kappa}\int_0^s \int\left\vert H^{j,\kappa}_{w\left( N^{-j}\right)^{-1}\left(t-r\right)}\left( z\right)-H^{j,\kappa}_{w\left( N^{-j}\right)^{-1}\left(s-r\right)}\left( z\right)\right\vert dzdr\\
	&:=& I'_{21}+I'_{22}.
\end{eqnarray*}

By Lemma \ref{lem1},
\begin{eqnarray*}
	I'_{21}
	&\leq& \left\vert e^{-\lambda\left( t-s\right)}-1\right\vert w\left( N^{-j}\right)^{-\kappa}\int_0^s e^{-\lambda r} \int\left\vert H^{j,\kappa}_{w\left( N^{-j}\right)^{-1}r}\left( z\right)\right\vert dzdr\\
	&\leq& C\left\vert e^{-\lambda\left( t-s\right)}-1\right\vert w\left( N^{-j}\right)^{-\kappa}\int_0^s e^{-\lambda r} e^{-C'w\left( N^{-j}\right)^{-1}r}dr.
\end{eqnarray*}
If $w\left( N^{-j}\right)^{-1}\left(t-s\right)\leq 1$, 
\begin{eqnarray*}
	I'_{21}&\leq& C\left\vert e^{-\lambda\left( t-s\right)}-1\right\vert w\left( N^{-j}\right)^{-\kappa}\int_0^s e^{-\lambda r} dr\\
	&\leq& C\left(t-s\right)w\left( N^{-j}\right)^{-\kappa}\leq C \left(t-s\right)^{1-\kappa}.
\end{eqnarray*}
If $w\left( N^{-j}\right)^{-1}\left(t-s\right)> 1$, use Lemma \ref{lem1}.
\begin{eqnarray*}
	I'_{21}
	&\leq& C w\left( N^{-j}\right)^{-\kappa}\int_0^{s} e^{-C'w\left( N^{-j}\right)^{-1}r}dr\leq C w\left( N^{-j}\right)^{1-\kappa}\leq C \left(t-s\right)^{1-\kappa}.
\end{eqnarray*}

On the other hand,
\begin{eqnarray*}
	I'_{22}&=&w\left( N^{-j}\right)^{1-\kappa}\int_0^{w\left( N^{-j}\right)^{-1}s} \int\left\vert H^{j,\kappa}_{w\left( N^{-j}\right)^{-1}\left(t-s\right)+r}\left( z\right)-H^{j,\kappa}_{r}\left( z\right)\right\vert dzdr\\
	&\leq& w\left( N^{-j}\right)^{1-\kappa}\int_0^{\infty} \int\left\vert H^{j,\kappa}_{w\left( N^{-j}\right)^{-1}\left(t-s\right)+r}\left( z\right)-H^{j,\kappa}_{r}\left( z\right)\right\vert dzdr.
\end{eqnarray*}
If $w\left( N^{-j}\right)^{-1}\left(t-s\right)\leq 1$, use Lemma \ref{lem4}.
\begin{eqnarray*}
	I'_{22}
	\leq C w\left( N^{-j}\right)^{1-\kappa}w\left( N^{-j}\right)^{-1}\left(t-s\right)\leq C \left(t-s\right)^{1-\kappa}.
\end{eqnarray*}
If $w\left( N^{-j}\right)^{-1}\left(t-s\right)> 1$, use Lemma \ref{lem1}.
\begin{eqnarray*}
	I'_{22}
	\leq 2 w\left( N^{-j}\right)^{1-\kappa}\int_0^{\infty} \int\left\vert H^{j,\kappa}_{r}\left( z\right)\right\vert dzdr\leq C w\left( N^{-j}\right)^{1-\kappa}\leq C \left(t-s\right)^{1-\kappa}.
\end{eqnarray*}

This is the end of the proof.
\end{proof}
\begin{theorem}\label{thm11}
	Let $\nu$ be a L\'{e}vy measure satisfying \textbf{A(w,l)} and $\beta\in\left( 0,\infty\right)$. Then the unique solution $u\in\left( t,x\right)$ to $\eqref{eeq1}$ satisfies
	\begin{eqnarray}
	\left\vert u\right\vert_{\beta,\infty}&\leq& C\left( \lambda^{-1}\wedge T\right) \left\vert f\right\vert_{\beta,\infty},\label{est5}\\
	\left\vert u\right\vert_{1+\beta,\infty}&\leq& C\left\vert f\right\vert_{\beta,\infty}\label{est1}
	\end{eqnarray}
	for some $C$ depending on $\alpha,\alpha_1,\alpha_2,N, \beta, d,T$. Meanwhile, for all $\kappa\in\left[ 0,1\right]$, there exists a constant $C$ depending on $\alpha,\kappa,\beta,\alpha_1,\alpha_2,N, d, T,\nu$ such that for all $0\leq s<t\leq T$,
	\begin{equation}\label{est2}
	\left\vert u\left(t,\cdot\right)-u\left(s,\cdot\right)\right\vert_{\kappa+\beta,\infty}\leq C\left\vert t-s\right\vert^{1-\kappa}\left\vert f\right\vert_{\beta,\infty}.
	\end{equation}
\end{theorem}
\begin{proof}
	Denote as before $u_j=u\ast \varphi_j, j\in\mathbf{N}$. We have known that 
	\begin{eqnarray*}
	u_j\left(t,x\right)=\int_0^t e^{-\lambda\left( t-s\right)}\mathbf{E}f_j\left(s, x+Z_{t-s}^{\nu}\right)ds,\forall j\in\mathbf{N}.
	\end{eqnarray*}
 Obviously, 
 \begin{eqnarray*}
 	\left\vert u_j\right\vert_0\leq \left\vert f_j\right\vert_0\int_0^t e^{-\lambda s}ds\leq C\left( \lambda^{-1}\wedge T\right) \left\vert f_j\right\vert_0,\forall j\in\mathbf{N},
 \end{eqnarray*}
which implies $\left\vert u\right\vert_{\beta,\infty}\leq C\left( \lambda^{-1}\wedge T\right) \left\vert f\right\vert_{\beta,\infty}$, $\forall \beta\in\left( 0,\infty\right)$. Recall $\eqref{sup}$. 
 	\begin{eqnarray*}
 	\left\vert u\right\vert_{0}\leq\left\vert u\right\vert_{\beta,\infty}\leq C\left\vert f\right\vert_{\beta,\infty}.
 \end{eqnarray*}
In the mean time, note $L^{\mu}u\ast \varphi_j= L^{\mu}u_j$, and by taking $\kappa=1$ in Corollary \ref{col1}, $\left\vert L^{\mu}u_j\right\vert_0\leq C\left\vert f_j\right\vert_0$. This is to say, $\left\vert L^{\mu}u\right\vert_{\beta,\infty}\leq C\left\vert f\right\vert_{\beta,\infty}$. By Proposition \ref{pr4}, 
	\begin{eqnarray*}
	\left\vert u\right\vert_{1+\beta,\infty}\leq C\left(\left\vert u\right\vert_{0}+\left\vert L^{\mu}u\right\vert_{\beta,\infty}\right)\leq C\left\vert f\right\vert_{\beta,\infty}.
\end{eqnarray*}

Similarly, by Lemma \ref{lem5}, we know that for all $j\in\mathbf{N}$,
	\begin{eqnarray*}
	\left\vert L^{\mu,\kappa}u_j\left( t,x\right)-L^{\mu,\kappa}u_j\left( s,x\right)\right\vert\leq C\left( t-s\right)^{1-\kappa}\left\vert f_j\right\vert_0,\forall x\in\mathbf{R}^d,\kappa\in\left[ 0,1\right],
\end{eqnarray*} 
namely, for all $\beta\in\left( 0,\infty\right)$,
	\begin{eqnarray*}
	\left\vert L^{\mu,\kappa}u\left( t,\cdot\right)-L^{\mu,\kappa}u\left( s,\cdot\right)\right\vert_{\beta,\infty}\leq C\left( t-s\right)^{1-\kappa}\left\vert f\right\vert_{\beta,\infty},\kappa\in\left[ 0,1\right].
\end{eqnarray*} 
Therefore, for all $\kappa\in\left[0,1\right]$ and all $0\leq s<t\leq T$,
\begin{eqnarray*}
&&\left\vert u\left(t,\cdot\right)-u\left(s,\cdot\right)\right\vert_{\mu,\kappa,\beta}\\
&\leq&\left\vert u\left(t,\cdot\right)-u\left(s,\cdot\right)\right\vert_{\beta,\infty}+\left\vert L^{\mu,\kappa}u\left(t,\cdot\right)-L^{\mu,\kappa}u\left(s,\cdot\right)\right\vert_{\beta,\infty}\\
&\leq& C\left\vert t-s\right\vert^{1-\kappa}\left\vert f\right\vert_{\beta,\infty}.
\end{eqnarray*}
By Proposition \ref{pr4}, this is equivalent to
\begin{eqnarray*}
	\left\vert u\left(t,\cdot\right)-u\left(s,\cdot\right)\right\vert_{\kappa+\beta,\infty}\leq C\left\vert t-s\right\vert^{1-\kappa}\left\vert f\right\vert_{\beta,\infty}.
\end{eqnarray*}
\end{proof}

\section{Proof of Theorem 1.1: Generalized H\"{o}lder-Zygmund inputs}

	\textsc{Existence and Estimates. }Given $f\in \tilde{C}^{\beta}_{\infty,\infty}\left( H_T\right)$, by Proposition \ref{app}, we can find a sequence of functions $f_n$ in $C_b^{\infty}\left( H_T\right)$ such that 
	\begin{equation*}
	\left\vert f_n\right\vert_{\beta,\infty}\leq C\left\vert f\right\vert_{\beta,\infty},\quad\left\vert f\right\vert_{\beta,\infty}\leq\liminf_{n}\left\vert f_n\right\vert_{\beta,\infty},
	\end{equation*}	
	and for any $0<\beta'<\beta$,
	\begin{equation*}
	\left\vert f_n-f\right\vert_{\beta',\infty}\rightarrow 0 \mbox{ as }n\rightarrow\infty.
	\end{equation*}
	According to Theorems \ref{thm1} and \ref{thm11}, for each pair of functions $f_m,f_n$, there are corresponding solutions $u_m, u_n\in C_b^{\infty}\left( H_T\right)$ verifying
	\begin{eqnarray*}
		\left\vert u_m- u_n\right\vert_{1+\beta',\infty}\leq C\left\vert f_m-f_n\right\vert_{\beta',\infty}\rightarrow 0,  \mbox{ as }m,n\rightarrow\infty
	\end{eqnarray*}
	for all $\beta'\in\left( 0,\beta\right)$, which by Proposition \ref{pr4} implies
	\begin{eqnarray*}
		\left\vert u_n-u_m\right\vert_0\rightarrow 0 \mbox{ as } m,n\rightarrow\infty.
	\end{eqnarray*}
	Clearly, $\{u_n:n\geq 0\}$ has a limit in the space of continuous functions. We denote it by $u$. $\lim_{n\rightarrow\infty}\left\vert u_n-u\right\vert_0=0$. Therefore, for any given $j\in\mathbf{N}$, 
\begin{eqnarray*}
	&&w\left( N^{-j}\right)^{-1-\beta}\left\vert u\ast \varphi_j\right\vert_0= \lim_{n\rightarrow\infty}w\left( N^{-j}\right)^{-1-\beta}\left\vert u_n\ast \varphi_j\right\vert_0\\
	&\leq&\limsup_{n\rightarrow\infty}\left\vert u_n\right\vert_{1+\beta,\infty}\leq C \limsup_{n\rightarrow\infty}\left\vert f_n\right\vert_{\beta,\infty}\leq C\left\vert f\right\vert_{\beta,\infty},
\end{eqnarray*}
which indicates $u\in \tilde{C}^{1+\beta}_{\infty,\infty}\left( H_T\right)$ and $\left\vert u\right\vert_{1+\beta,\infty}\leq C\left\vert f\right\vert_{\beta,\infty}$. Meanwhile, for any given $j\in\mathbf{N}$ and any $\beta'\in\left( 0,\beta\right)$, 
\begin{eqnarray*}
	&&\lim_{n\rightarrow\infty}w\left( N^{-j}\right)^{-1-\beta'}\left\vert \left( u_n-u\right)\ast \varphi_j\right\vert_0\\
	&=&\lim_{n\rightarrow\infty}\lim_{m\rightarrow\infty}w\left( N^{-j}\right)^{-1-\beta'}\left\vert \left( u_n-u_m\right)\ast \varphi_j\right\vert_0\\
	&\leq&\lim_{n\rightarrow\infty}\lim_{m\rightarrow\infty}\left\vert \left( u_n-u_m\right)\right\vert_{1+\beta',\infty}\\
	&\leq& C \lim_{n\rightarrow\infty}\lim_{m\rightarrow\infty}\left\vert f_n-f_m\right\vert_{\beta',\infty}= 0,
\end{eqnarray*}
Namely, for all $\beta'\in\left( 0,\beta\right)$.
\begin{eqnarray}\label{supp}
\lim_{n\rightarrow\infty}\left\vert  u_n-u\right\vert_{1+\beta',\infty}\rightarrow 0 ,  \mbox{ as }n\rightarrow\infty.
\end{eqnarray}
	
 Analogously, for any given $j\in\mathbf{N}$, 
\begin{eqnarray*}
	&&w\left( N^{-j}\right)^{-\beta}\left\vert u\ast \varphi_j\right\vert_0= \lim_{n\rightarrow\infty}w\left( N^{-j}\right)^{-\beta}\left\vert u_n\ast \varphi_j\right\vert_0\\
	&\leq&\limsup_{n\rightarrow\infty}\left\vert u_n\right\vert_{\beta,\infty}\leq C\left( \lambda^{-1}\wedge T\right) \limsup_{n\rightarrow\infty}\left\vert f_n\right\vert_{\beta,\infty}\leq C\left( \lambda^{-1}\wedge T\right) \left\vert f\right\vert_{\beta,\infty}.
\end{eqnarray*}
This implies $\left\vert u\right\vert_{\beta,\infty}\leq C\left( \lambda^{-1}\wedge T\right) \left\vert f\right\vert_{\beta,\infty}$.
	
	Using Theorems \ref{thm1} and \ref{thm11}, we can show in the same vein that for all $0\leq s\leq t\leq T$, $\kappa\in\left[ 0,1\right]$,
	\begin{eqnarray*}
		u\left(t,\cdot\right)-u\left(s,\cdot\right)&=&\lim_{n\rightarrow \infty}\left( u_n\left(t,\cdot\right)-u_n\left(s,\cdot\right)\right)\in \tilde{C}^{\kappa+\beta}_{\infty,\infty}\left( \mathbf{R}^d\right),\\
		\left\vert u\left(t,\cdot\right)-u\left(s,\cdot\right)\right\vert_{\kappa+\beta,\infty}&\leq& C\limsup_{n\rightarrow\infty}\left\vert u_n\left(t,\cdot\right)-u_n\left(s,\cdot\right)\right\vert_{\kappa+\beta,\infty}\\
		&\leq& C\left\vert t-s\right\vert^{1-\kappa}\left\vert f\right\vert_{\beta,\infty}.
	\end{eqnarray*}
	
	Now we claim that such a function $u$ solves $\eqref{eeq1}$, i.e.,
	\begin{eqnarray}\label{eq2}
	\quad &&u\left( t,x\right)=\int_0^t \left[ L^{\nu}u\left( r,x\right)-\lambda u\left( r,x\right)+f\left( r,x\right)\right]dr, \quad\left( t,x\right)\in \left[0,T\right]\times\mathbf{R}^d.
	\end{eqnarray}
	Indeed, according to $\eqref{supp}$ and Proposition \ref{cont}, $L^{\nu}u=\lim_{n\rightarrow\infty}L^{\nu}u_n$ and $\lim_{n\rightarrow\infty}\left\vert L^{\nu}u_n-L^{\nu}u\right\vert_0=0$. Passing the limit on both sides of 
	\begin{eqnarray*}
		u_n\left( t,x\right)=\int_0^t \left[ L^{\nu}u_n\left( r,x\right)-\lambda u_n\left( r,x\right)+f_n\left( r,x\right)\right]dr, 
	\end{eqnarray*}
	we obtain $\eqref{eq2}$.
	
	\textsc{Uniqueness. }Suppose there are two solutions $u_1,u_2\in \tilde{C}^{1+\beta}_{\infty,\infty}\left( H_T\right)$ to $\eqref{eq2}$, then $u:=u_1-u_2$ solves
	\begin{eqnarray*}
		u\left( t,x\right)=\int_0^t \left[ L^{\nu}u\left( r,x\right)-\lambda u\left( r,x\right)\right]dr, \quad\left( t,x\right)\in \left[0,T\right]\times\mathbf{R}^d.
	\end{eqnarray*}
	By Proposition \ref{app}, there is a sequence of functions $u_n\in C_b^{\infty}\left( H_T\right)$ such that for any $0<\beta'<\beta$, 
	\begin{equation}\label{mmm}
	\left\vert u_n-u\right\vert_{1+\beta',\infty}\rightarrow 0 \mbox{ as } n\rightarrow\infty.
	\end{equation}

	Clearly, $\tilde{u}_n\left( t,x\right):=\int_0^t u_n\left( s,x\right)ds$ solves
	\begin{eqnarray*}
		\int_0^t u_n\left( s,x\right)ds&=&\int_0^t \lbrack L^{\nu}\int_0^s u_n\left( r,x\right)dr-\lambda \int_0^s u_n\left( r,x\right)dr+\Big( u_n\left( s,x\right)\\
		&&-L^{\nu}\int_0^s u_n\left( r,x\right)dr+\lambda\int_0^s u_n\left( r,x\right)dr\Big)\rbrack ds.
	\end{eqnarray*}
	By Lemma \ref{Lop}, $ u_n\left( t,x\right)-L^{\nu}\int_0^t u_n\left( s,x\right)ds+\lambda\int_0^t u_n\left( s,x\right)ds\in C_b^{\infty}\left( H_T\right)$. Then according to Theorem \ref{thm11}, 
	\begin{eqnarray*}
		\left\vert \tilde{u}_n\right\vert_{1+\beta',\infty}\leq\left\vert u_n\left( t,x\right)-L^{\nu}\int_0^t u_n\left( s,x\right)ds+\lambda\int_0^t u_n\left( s,x\right)ds\right\vert_{\beta',\infty}.
	\end{eqnarray*}
Use $\eqref{mmm}$, $\eqref{sup}$ and Proposition \ref{cont}.
	\begin{eqnarray*}
		&&\left\vert \int_0^t u\left( s,x\right)ds\right\vert_{1+\beta',\infty}= \lim_{n\rightarrow\infty}\left\vert \int_0^t u_n\left( s,x\right)ds\right\vert_{1+\beta',\infty}\\
		&\leq&\liminf_{n\rightarrow\infty}\left\vert u_n\left( t,x\right)-L^{\nu}\int_0^t u_n\left( s,x\right)ds+\lambda\int_0^t u_n\left( s,x\right)ds\right\vert_{\beta',\infty}\\
		&\leq& \left\vert u\left( t,x\right)-L^{\nu}\int_0^t u\left( s,x\right)ds+\lambda\int_0^t u\left( s,x\right)ds\right\vert_{\beta',\infty}=0.
	\end{eqnarray*}
By $\eqref{sup}$ again, $\int_0^t u\left( s,x\right)ds=0$ for all $t\in\left[ 0,T\right]$ and $x\in\mathbf{R}^d$, and thus $u=0$ $\left(t,x\right)$-a.e..

\section{Appendix}

We simply state a few results that were used in this paper. Please look up in the references for proofs if you are interested. 

Recall all parameters introduced in assumptions \textbf{A(w,l)}.
\begin{lemma}\cite[Lemma 7]{cr}\label{lemma1}
	Let $\nu$ be a L\'{e}vy measure of order $\alpha$ and $w$ be a scaling function. \\
	\noindent (i) Suppose there exists $N_2>0$ such that for all $R>0$,
	\begin{eqnarray*}
	\int \left(\left\vert y\right\vert \wedge 1\right)\tilde{\nu}_{R}\left( dy\right)&\leq&N_2 \mbox{ if } \alpha\in\left( 0,1\right),\\
	\int \left(\left\vert y\right\vert^2 \wedge 1\right)\tilde{\nu}_{R}\left( dy\right)&\leq&N_2 \mbox{ if } \alpha=1,\\
	\int \left(\left\vert y\right\vert^2 \wedge \left\vert y\right\vert\right)\tilde{\nu}_{R}\left( dy\right)&\leq&N_2 \mbox{ if } \alpha\in\left(1,2\right).
	\end{eqnarray*}
	Then there is a constant $C>0$ depending only on $c_1,N_0,N_1,N_2$ such that for all $\xi\in\mathbf{R}^d$,
	\begin{eqnarray*}
	\int\left[ 1-\cos\left( 2\pi\xi\cdot y\right)\right]\nu\left(dy\right)&\leq& C w\left( \left\vert \xi\right\vert^{-1}\right)^{-1},\\
	\int\left[ \sin\left( 2\pi\xi\cdot y\right)-2\pi\chi_{\alpha}\left( y\right) \xi\cdot y\right]\nu\left(dy\right)&\leq& C w\left( \left\vert \xi\right\vert^{-1}\right)^{-1},
	\end{eqnarray*}
where $w\left( \left\vert \xi\right\vert^{-1}\right)^{-1}:=0$ if $\xi=0$.\\
\noindent (ii) Suppose there is $n_1>0$ such that for all $R>0$ and all $\xi\in\{\xi\in\mathbf{R}^d:\left\vert \xi\right\vert=1\}$,
\begin{eqnarray*}
	\int_{\left\vert y\right\vert \leq 1}\left\vert \xi\cdot y\right\vert^2 \tilde{\nu}_R\left(dy\right)\geq n_1.
\end{eqnarray*}
Then there is a constant $c>0$ depending only on $c_1,N_0,N_1,N_2,n_1$ such that for all $\xi\in\mathbf{R}^d$,
\begin{eqnarray*}
	\int\left[ 1-\cos\left( 2\pi\xi\cdot y\right)\right]\nu\left(dy\right)\geq c w\left( \left\vert \xi\right\vert^{-1}\right)^{-1},
\end{eqnarray*}
where $w\left( \left\vert \xi\right\vert^{-1}\right)^{-1}:=0$ if $\xi=0$.
\end{lemma}

\begin{lemma}\cite[Lemma 5]{cr}\label{lemma3}
	Let $\nu$ be a L\'{e}vy measure satisfying \textbf{A(w,l)}. $Z^{\tilde{\nu}_R}_t$ is the L\'{e}vy process associated to $\tilde{\nu}_R,R>0$. For each $t, R$, $Z^{\tilde{\nu}_R}_t$ has a bounded and continuous density function $p^R\left( t,x\right),t\in\left( 0,\infty\right), x\in\mathbf{R}^d$. And $p^R\left( t,x\right)$ has bounded and continuous derivatives up to order $4$. Meanwhile, for any multi-index $\left\vert\vartheta\right\vert\leq 4$, 
\begin{eqnarray*}
	\int\left\vert \partial^{\vartheta}p^R\left( t,x\right)\right\vert dx &\leq&  C\gamma\left( t\right)^{-\left\vert\vartheta\right\vert},\\
	\sup_{x\in\mathbf{R}^d}\left\vert \partial^{\vartheta}p^R\left( t,x\right)\right\vert  &\leq&  C\gamma\left( t\right)^{-d-\left\vert\vartheta\right\vert},
\end{eqnarray*}
where $C>0$ is independent of $t,R$. For any $\beta\in\left(0,1\right)$ such that $\left\vert\vartheta\right\vert+\beta<4$,
\begin{eqnarray*}
	\int\left\vert \partial^{\beta}\partial^{\vartheta}p^R\left( t,x\right)\right\vert dx &\leq&  C\gamma\left( t\right)^{-\left\vert\vartheta\right\vert-\beta}.
\end{eqnarray*}
For any $a>0$, there is a constant $C>0$ independent of $t,R$, so that 
	\begin{eqnarray*}
	\int_{\left\vert x\right\vert>a}\left\vert \partial^{\vartheta}p^R\left( t,x\right)\right\vert dx &\leq&  C\left( \gamma\left( t\right)^{2-\left\vert\vartheta\right\vert}+t\gamma\left( t\right)^{-\left\vert\vartheta\right\vert}\right). 
	\end{eqnarray*}
\end{lemma}	
 
\begin{lemma}\cite[Lemma 2]{cr2}\label{lemma2}
	Let $\nu$ be a L\'{e}vy measure satisfying \textbf{A(w,l)} and $Z^{\tilde{\nu}_R}_t$ be the L\'{e}vy process associated to $\tilde{\nu}_R$. Then for any $\varphi,\varphi_0\in \mathcal{S}\left( \mathbf{R}^d\right)$ such that $\mathcal{F}\varphi_0\in C_0^{\infty}\left( \mathbf{R}^d\right)$, $supp \left(\mathcal{F}\varphi\right)\subseteq \{\xi: 0<R_1\leq \left\vert \xi\right\vert\leq R_2\}$, and 
	\begin{eqnarray*}
	\max_{\left\vert\gamma\right\vert\leq \left[d/2\right]+3}\left\vert D^{\gamma}\hat{\varphi}\left(\xi\right)\right\vert \leq N_2, R_1\leq \left\vert\xi\right\vert\leq R_2.
	\end{eqnarray*}
	Then there are constants $C_1,C_2>0$ depending only on $c_1,N_0,N_1,N_2,R_1,R_2,d$ such that 
	\begin{eqnarray*}
		\int\left( 1+\left\vert x\right\vert^{\alpha_2}\right)\left\vert \mathbf{E}\varphi\left( x+Z^{\tilde{\nu}_R}_t\right)\right\vert dx&\leq& C_1 e^{-C_2 t}, \quad t\geq 0,\\
		\int\left\vert x\right\vert^{\alpha_2}\left\vert \mathbf{E}\varphi_0\left( x+Z^{\tilde{\nu}_R}_t\right)\right\vert dx &\leq& C_1 \left( 1+t\right), \quad t\geq 0,\\
		\int \left\vert \mathbf{E}\varphi_0\left( x+Z^{\tilde{\nu}_R}_t\right)\right\vert dx &\leq& C_1, \quad t\geq 0.
	\end{eqnarray*}
\end{lemma}

\end{document}